\documentclass[preprint,12pt]{elsarticle}

\usepackage{amssymb,amsmath,amsthm}
\usepackage{latexsym}
\usepackage{empheq}
\usepackage{enumerate,verbatim}
\usepackage[title]{appendix}
\usepackage{lineno,hyperref,color,bm}
\usepackage{booktabs}

\usepackage{graphicx}
\usepackage{subcaption}
\usepackage{caption}
\DeclareCaptionSubType{figure}

\usepackage[ruled,vlined,linesnumbered]{algorithm2e}

\SetKwInput{KwIn}{Input}
\SetKwInput{KwOut}{Output}
\SetKw{Return}{return}
\usepackage{geometry}
\newtheorem{theorem}{Theorem}[section]
\newtheorem{lemma}[theorem]{Lemma}
\newtheorem{remark}[theorem]{Remark}
\newtheorem{corollary}[theorem]{Corollary}
\newtheorem{proposition}[theorem]{Proposition}
\newtheorem{definition}[theorem]{Definition}

\graphicspath{{fig/}}
\makeatletter
\def\ps@pprintTitle{%
	\let\@oddhead\@empty
	\let\@evenhead\@empty
	\let\@oddfoot\@empty
	\let\@evenfoot\@empty
}
\makeatother

\providecommand{\argmax}{\operatorname*{arg\,max}}

\begin{document}

\begin{frontmatter}

\title{Sparse Recovery via $\ell_1^2-\eta\ell_2^2$ Minimization\tnoteref{fundingnote}}
\tnotetext[fundingnote]{This work was funded by the National Natural Science Foundation of China (12171339, 12471296).}

\author[label1]{Lang Yu}
\ead{langyu9602@gmail.com}
\author[label1]{Nanjing Huang\corref{mycorrespondingauthor}}
\ead{nanjinghuang@hotmail.com}

\cortext[mycorrespondingauthor]{Corresponding author}
\address[label1]{School of Mathematics, Sichuan University, 610064, Chengdu, P.R. China}

\begin{abstract}
The weighted difference of squared norms (WDSN) penalty $\ell_1^2-\eta\ell_2^2$ with $0\leq \eta\leq 1$ has attracted considerable attention due to its strong sparsity-promoting ability and favorable reconstruction performance in compressed sensing and inverse problems. However, exact recovery guarantees and restricted isometry property (RIP) analysis for WDSN minimization have not yet been established. In this paper, we address this gap. First, we establish sufficient conditions for the exact recovery of $k$-sparse signals based on the null space property (NSP). Then, under the $\delta_{2k}$-RIP condition, we derive stable recovery guarantees for both $k$-sparse signals and general signals, and characterize upper bounds on the reconstruction error. Furthermore, we propose a WDSN-based regularized model to handle both noiseless and noisy observations in a unified framework. To design an efficient algorithm, we derive an explicit formula for the proximal operator of the WDSN functional. Based on this proximal solver, we develop a suitable variable-splitting scheme within the alternating direction method of multipliers (ADMM) and establish its global convergence under some mild conditions. Finally, numerical experiments show that the proposed method outperforms the iterative half variation method in both noiseless and noisy sparse recovery tasks.
\end{abstract}

\begin{keyword}
compressed sensing \sep sparse signal recovery \sep $\ell_1^2-\eta\ell_2^2$ nonconvex minimization \sep restricted isometry property \sep proximal operator \sep nonconvex ADMM
\end{keyword}

\end{frontmatter}

%% main text
\section{Introduction}

Compressed sensing and sparse inverse problems are concerned with the recovery of sparse or compressible signals from highly undersampled linear measurements. A standard observation model is
$$
	\bm{b}=\bm{A}\bm{x}+\bm{e},
$$
where $\bm{A}\in\mathbb{R}^{m\times n}$ with $m\ll n$, $\bm{x}\in\mathbb{R}^n$ is the unknown signal, and $\bm{e}$ denotes measurement noise. The ideal recovery model seeks the sparsest feasible solution through $\ell_0$-minimization. This formulation, however, is combinatorial and generally NP-hard \cite{Natarajan1995SparseApproximate}. A classical alternative is to replace the $\ell_0$ measure by the convex $\ell_1$ norm, leading to basis pursuit and its noise-aware variants. The success of $\ell_1$-minimization has been justified by the null space property (NSP), the restricted isometry property (RIP), and related recovery conditions \cite{Donoho2006Compressed,CandesTao2005Decoding,CandesRombergTao2006Stable,Candes2008RIP,FoucartRauhut2013CompressedSensing}. In particular, RIP-based analysis provides a standard framework for deriving exact recovery conditions, stability estimates, and explicit reconstruction error bounds.

Although $\ell_1$-minimization is computationally tractable and theoretically well understood, the $\ell_1$ norm is still a convex surrogate of the $\ell_0$ measure and may introduce bias in sparse estimation. This limitation has motivated a broad class of nonconvex sparsity-promoting models, including $\ell_p$ minimization with $0<p<1$, the difference-of-norms penalty $\ell_1-\ell_2$, ratio-type penalties such as $\ell_1/\ell_2$, and more general $\ell_p-\alpha\ell_q$ or $\ell_p/\ell_q$ formulations \cite{Chartrand2007Exact,ChartrandStaneva2008RIPNonconvex,ZhangLi2019OptimalRIP,YinLouHeXin2015L1L2,XuNarayanTranWebster2021Ratio,ZhouYu2021QRatio,GeXieChen2024UniformRIP}. These nonconvex models are designed to promote sparsity more strongly than the $\ell_1$ norm and have shown advantages in sparse recovery, especially for coherent sensing systems. At the same time, their nonconvexity and nonseparability make the analysis of exact recovery, stability, and algorithmic convergence considerably more delicate.

Among nonconvex nonseparable penalties, the $\ell_1-\ell_2$ model has attracted particular attention. Yin et al. \cite{YinLouHeXin2015L1L2} studied $\ell_1-\ell_2$ minimization for compressed sensing and demonstrated its effectiveness in sparse recovery. Lou and Yan \cite{LouYan2018Fast} derived an explicit proximal operator for the $\ell_1-\ell_2$ metric and incorporated it into forward--backward splitting and the alternating direction method of multipliers (ADMM), leading to efficient numerical algorithms. More recently, Ge, Chen, and Ng \cite{GeChenNg2021NewRIP} developed refined RIP analysis for $\ell_1-\ell_2$ minimization and established improved exact and stable recovery guarantees. These studies show that NSP/RIP analysis is essential for understanding the recovery capability of nonconvex sparse penalties and for explaining their advantages over classical convex relaxation.

Despite these developments, the $\ell_1-\eta\ell_2$ penalty still has an important analytical limitation. The term $\|\bm{x}\|_2$ is not differentiable at the origin, and iterative schemes involving $\ell_1-\eta\ell_2$ may produce singular expressions of the form $\bm{x}/\|\bm{x}\|_2$. This difficulty complicates convergence analysis and often requires additional assumptions to avoid zero or near-zero iterates \cite{DingHan2019AlphaL1BetaL2,LiDing2026WDSNNonlinear}. A natural way to remove this singularity while retaining a difference-of-norms structure is to replace the $\ell_2$ term by its squared counterpart. To this end, Li and Ding \cite{LiDing2025WDSNLinear} introduced the weighted difference-of-squared-norms (WDSN) sparsity functional
$$
	\mathcal{R}_{\eta}(\bm{x}) := \|\bm{x}\|_1^2-\eta\|\bm{x}\|_2^2, \qquad 0\leq \eta\leq 1
$$
and established existence, stability, sparsity of regularized solutions for linear ill-posed problems. They further extended their results to nonlinear ill-posed inverse problems in \cite{LiDing2026WDSNNonlinear}. 

It is worth mentioning that $\mathcal{R}_{\eta}$ remains nonsmooth because of the squared $\ell_1$ term, the use of $\|\bm{x}\|_2^2$ avoids the singularity associated with $\|\bm{x}\|_2$ at the origin. Moreover, as $\eta$ increases, the negative squared $\ell_2$ term enhances the nonconvex sparsity-promoting effect. Another useful feature is the two-homogeneity of $\mathcal{R}_{\eta}$, which matches the homogeneity of the quadratic data-fidelity term. This suggests that, in a regularized least-squares model, the relative balance between data fidelity and WDSN regularization is naturally preserved under amplitude rescaling of the measurements. Thus, WDSN provides a useful compromise between sparse promotion and analytical tractability.

It should be noticed that the work of \cite{LiDing2026WDSNNonlinear, LiDing2025WDSNLinear} are mainly developed within Tikhonov-type regularization frameworks. However, to our best knowledge, the compressed sensing question of whether WDSN minimization admits exact recovery guarantees and RIP-based stable recovery bounds has not yet been systematically studied. Motivated by this gap, this paper studies two constrained sparse recovery models induced by the WDSN functional. Specifically, we consider the following WDSN minimization
\begin{subequations}\label{eq:WDSN_models}
	\begin{empheq}[left=\empheqlbrace]{align}
		\mathcal{P}_{\eta}: \quad &\min_{\bm{x}\in\mathbb{R}^n}  \mathcal{R}_{\eta}(\bm{x}) \quad \mbox{ s.t. }\bm{A}\bm{x}=\bm{b}, \label{eq:WDSN_exact} \\
		\mathcal{P}_{\eta}^{\varepsilon}: \quad &\min_{\bm{x}\in\mathbb{R}^n}  \mathcal{R}_{\eta}(\bm{x}) \quad \mbox{ s.t. } \|\bm{A}\bm{x}-\bm{b}\|_2\le \varepsilon, \label{eq:WDSN_stable}
	\end{empheq}
\end{subequations}
where the equality-constrained model $\mathcal{P}_{\eta}$ corresponds to noiseless measurements to be used for exact recovery analysis, and the inequality-constrained model $\mathcal{P}_{\eta}^{\varepsilon}$ corresponds to noise measurements to be used for stable recovery analysis.

This paper addresses the theoretical and computational issues for $\mathcal{P}_{\eta}$ and $\mathcal{P}_{\eta}^{\varepsilon}$. It is worth noting  that we are faced with several challenges in this study. On the theoretical side, the WDSN functional is nonconvex and nonseparable, and the negative quadratic term $-\eta\|\bm{x}\|_2^2$ destroys the standard cone structure used in classical $\ell_1$ recovery analysis, which leads to the fact that the usual NSP and RIP techniques cannot be applied directly. Thus, one has to establish new estimates for the recovery error that incorporate the $\eta$-dependent term and still yield explicit $\delta_{2k}$-based recovery bounds. On the computational side, the squared $\ell_1$ term couples all coordinates and WDSN is nonconvex in general, which results in the complexity of computation. Therefore, it is necessary to   design an efficient solver. 

The main contributions of this paper are fourfold. First, we provide an NSP-type sufficient condition for exact recovery via the equality-constrained WDSN (EC-WDSN) model $\mathcal{P}_{\eta}$, thereby giving a theoretical characterization of its sparse recovery capability. Second, we develop a $\delta_{2k}$-RIP-based stable recovery theory for the noise-constrained WDSN (NC-WDSN) model $\mathcal{P}_{\eta}^{\varepsilon}$, covering both $k$-sparse and general signals and yielding explicit reconstruction error estimates. Third, we formulate a WDSN regularization model and derive the closed-form proximal operator of the WDSN functional, which provides the basis for efficient numerical implementation. Fourth, we propose a variable-splitting ADMM algorithm tailored to WDSN minimization and prove its global convergence under some mild assumptions.

The remainder of this paper is organized as follows. Section~\ref{sec-exact} establishes an exact recovery guarantee through a NSP adapted to the WDSN functional. Sections~\ref{sec:stable-recovery-delta-2k} and~\ref{sec:stable-recovery-general-delta-2k} develop RIP-based stable recovery results for sparse and general signals, respectively. Section~\ref{sec:prox-l1sq-l2sq} derives the closed-form proximal mapping of $\mathcal{R}_{\eta}$. Sections~\ref{sec:admm-prox} and~\ref{sec:admm-convergence} present the ADMM algorithm and its convergence analysis. Section~\ref{sec-num} reports the numerical experiments, and Section~\ref{conclusion} concludes the paper.

\section{Exact Recovery via Null Space Property}
\label{sec-exact}

In this section, we establish an exact recovery guarantee for the EC-WDSN minimization problem $\mathcal{P}_{\eta}$. Unlike the RIP-based analysis, the proof in this section is based directly on the geometry of the null space of the sensing matrix $A$. The null space property (NSP) is more intrinsic than the restricted isometry property (RIP) in the sense that it characterizes exact recovery through the interaction between the null space of $A$ and sparse supports.

For the WDSN penalty $\mathcal{R}_{\eta}$, which is generally nonconvex when $\eta>0$, the following $\eta$-augmented NSP is a sufficient null-space condition adapted to the additional negative quadratic term in $\mathcal R_\eta$.

\begin{definition}[$\eta$-Augmented NSP]\label{def:ansp}
	Let $k$ be a positive integer and $0\le \eta\le 1$. A matrix $\bm{A}\in\mathbb{R}^{m\times n}$ is said to satisfy the $\eta$-augmented NSP of order $k$, abbreviated as the $(\eta,k)$-ANSP, if for every nonzero vector $\bm{v}\in \mathcal{N}(\bm{A})\setminus\{0\}$, and every index set $\mathcal{S}\subset\{1,\ldots,n\}$ with $|\mathcal{S}|\le k$, one has
	\begin{equation}
		\label{eq:ansp-def}
		\|\bm{v}_{\mathcal{S}}\|_1+\eta\|\bm{v}\|_2
		<
		\|\bm{v}_{\mathcal{S}^c}\|_1. 
	\end{equation}
\end{definition}

\begin{remark}
	\label{rem:ansp-classical}
	When $\eta=0$, the $(\eta,k)$-ANSP states that, for every $\bm{v}\in\mathcal{N}(\bm{A})\setminus\{0\}$ and every index set $\mathcal{S}$ with $|\mathcal{S}|\le k$, one has $\|\bm{v}_{\mathcal{S}}\|_1 < \|\bm{v}_{\mathcal{S}^c}\|_1$. This is exactly the classical NSP for $\ell_1$ minimization.
\end{remark}

\begin{lemma}[]\label{lem:objective-induced-cone}
	Let $\eta \in [0,1]$ and define $\mathcal{R}_{\eta}(\bm{x}):=\|\bm{x}\|_1^2-\eta\|\bm{x}\|_2^2$. 
	Let $\mathcal{C}\subseteq\mathbb{R}^n$ be a feasible set, and let
	$\bm{\hat{x}}\in\operatorname*{arg\,min}_{\bm{x}\in\mathcal{C}}\mathcal{R}_{\eta}(\bm{x})$. 
	Suppose that $\bm{x}^*\in\mathcal{C}$, set $\bm{h}:=\bm{\hat{x}}-\bm{x}^*$, and let
	$\mathcal{S}\subseteq\{1,\ldots,n\}$ satisfy $\operatorname{supp}(\bm{x}^*)\subseteq \mathcal{S}$. 
	Then
	\begin{equation}
		\label{eq:objective-induced-cone}
		\|\bm{h}_{\mathcal{S}^c}\|_1
		\le
		\|\bm{h}_{\mathcal{S}}\|_1+\eta\|\bm{h}\|_2 .
	\end{equation}
\end{lemma}

\begin{proof}
	Since $\bm{x}^*\in\mathcal{C}$ and $\bm{\hat{x}}$ is a global minimizer of $\mathcal{R}_{\eta}$ over $\mathcal{C}$, we have
	\begin{equation}
		\label{eq:objective-induced-optimality}
		\|\bm{\hat{x}}\|_1^2-\eta\|\bm{\hat{x}}\|_2^2
		\le
		\|\bm{x}^*\|_1^2-\eta\|\bm{x}^*\|_2^2 .
	\end{equation}
	Rearranging \eqref{eq:objective-induced-optimality} gives
	\begin{equation}
		\label{eq:objective-induced-square-diff}
		\|\bm{\hat{x}}\|_1^2-\|\bm{x}^*\|_1^2
		\le
		\eta\left(\|\bm{\hat{x}}\|_2^2-\|\bm{x}^*\|_2^2\right).
	\end{equation}
	If $\|\bm{\hat{x}}\|_1\le \|\bm{x}^*\|_1$, then \eqref{eq:objective-induced-cone} follows immediately. Hence we may assume that $\|\bm{\hat{x}}\|_1>\|\bm{x}^*\|_1$. In this case, \eqref{eq:objective-induced-square-diff} excludes $\eta=0$, and therefore $\eta>0$. It follows from the difference-of-squares identity in \eqref{eq:objective-induced-square-diff} that
	\begin{equation}
		\label{eq:objective-induced-difference-squares}
		\bigl(\|\bm{\hat{x}}\|_1-\|\bm{x}^*\|_1\bigr)
		\bigl(\|\bm{\hat{x}}\|_1+\|\bm{x}^*\|_1\bigr)
		\le
		\eta
		\bigl(\|\bm{\hat{x}}\|_2-\|\bm{x}^*\|_2\bigr)
		\bigl(\|\bm{\hat{x}}\|_2+\|\bm{x}^*\|_2\bigr).
	\end{equation}
	The left-hand side of \eqref{eq:objective-induced-difference-squares} is positive, hence $\|\bm{\hat{x}}\|_2>\|\bm{x}^*\|_2$. Since $\|\bm{\hat{x}}\|_2-\|\bm{x}^*\|_2\le\|\bm{\hat{x}}-\bm{x}^*\|_2=\|\bm{h}\|_2$ and $\|\bm{\hat{x}}\|_2+\|\bm{x}^*\|_2\le\|\bm{\hat{x}}\|_1+\|\bm{x}^*\|_1$, it follows from \eqref{eq:objective-induced-difference-squares} that $\bigl(\|\bm{\hat{x}}\|_1-\|\bm{x}^*\|_1\bigr)\bigl(\|\bm{\hat{x}}\|_1+\|\bm{x}^*\|_1\bigr)\le\eta\|\bm{h}\|_2\bigl(\|\bm{\hat{x}}\|_1+\|\bm{x}^*\|_1\bigr)$. Dividing by the positive quantity $\|\bm{\hat{x}}\|_1+\|\bm{x}^*\|_1$ gives $\|\bm{\hat{x}}\|_1-\|\bm{x}^*\|_1 \le \eta\|\bm{h}\|_2$. Since $\operatorname{supp}(\bm{x}^*)\subseteq \mathcal{S}$, we have $\bm{x}^*_{\mathcal{S}^c}=0$, and hence
	\begin{equation*}
		\|\bm{\hat{x}}\|_1
		=
		\|\bm{x}^*_{\mathcal{S}}+\bm{h}_{\mathcal{S}}\|_1+\|\bm{h}_{\mathcal{S}^c}\|_1
		\ge
		\|\bm{x}^*\|_1-\|\bm{h}_{\mathcal{S}}\|_1+\|\bm{h}_{\mathcal{S}^c}\|_1 .
	\end{equation*}
	Therefore, we have $\|\bm{h}_{\mathcal{S}^c}\|_1-\|\bm{h}_{\mathcal{S}}\|_1\le\|\bm{\hat{x}}\|_1-\|\bm{x}^*\|_1\le\eta\|\bm{h}\|_2$, where the last inequality follows from $\|\bm{\hat{x}}\|_1-\|\bm{x}^*\|_1 \le \eta\|\bm{h}\|_2$. This is precisely \eqref{eq:objective-induced-cone}, and the proof is complete.
\end{proof}

\begin{theorem}\label{thm:uniform-exact-recovery-ansp}
	Assume that $\bm{A}$ satisfies the $(\eta,k)$-ANSP. Then, for every $k$-sparse vector $\bm{x}^*\in\mathbb{R}^n$ with $\bm{b}=\bm{A}\bm{x}^*$, the vector $\bm{x}^*$ is the unique global minimizer of $\mathcal{P}_{\eta}$.
\end{theorem}

\begin{proof}
	Let $\bm{x}^*$ be an arbitrary $k$-sparse vector and set $\mathcal{S}:=\operatorname{supp}(\bm{x}^*)$. Then $|\mathcal{S}|\le k$ and $\bm{x}^*$ is feasible for $\mathcal{P}_{\eta}$ with $\bm{b}=\bm{A}\bm{x}^*$. Let $\bm{x}$ be any feasible point of $\mathcal{P}_{\eta}$ and define
	$\bm{h}:=\bm{x}-\bm{x}^*$. Since $\bm{A}\bm{x}=\bm{b}$ and
	$\bm{A}\bm{x}^*=\bm{b}$, we have $\bm{A}\bm{h}=\bm{0}$, that is, $ \bm{h}\in\mathcal{N}(\bm{A})$. If $\bm{h}=\bm{0}$, then $\bm{x}=\bm{x}^*$. Hence it remains to consider
	the case $\bm{h}\ne\bm{0}$. Since $\mathcal{S}=\operatorname{supp}(\bm{x}^*)$, one has
	$\bm{x}^*_{\mathcal{S}^c}=\bm{0}$. From $\bm{x}=\bm{x}^*+\bm{h}$, it follows that $\bm{x}_{\mathcal{S}}=\bm{x}^*_{\mathcal{S}}+\bm{h}_{\mathcal{S}}$ and $\bm{x}_{\mathcal{S}^c}=\bm{h}_{\mathcal{S}^c}$. Therefore, one has
	\begin{equation}\label{eq:support-l1-decomposition}
		\begin{aligned}
			\|\bm{x}\|_1=\|\bm{x}_{\mathcal{S}}\|_1+\|\bm{x}_{\mathcal{S}^c}\|_1 =\|\bm{x}^*_{\mathcal{S}}+\bm{h}_{\mathcal{S}}\|_1+\|\bm{h}_{\mathcal{S}^c}\|_1 \ge \|\bm{x}^*_{\mathcal{S}}\|_1-\|\bm{h}_{\mathcal{S}}\|_1+\|\bm{h}_{\mathcal{S}^c}\|_1 .
		\end{aligned}
	\end{equation}
	Here the last inequality follows from the triangle inequality. Since$\|\bm{x}^*_{\mathcal{S}}\|_1=\|\bm{x}^*\|_1$, \eqref{eq:support-l1-decomposition} implies
	\begin{equation}\label{eq:uniform-l1-lower}
		\|\bm{x}\|_1-\|\bm{x}^*\|_1
		\ge
		\|\bm{h}_{\mathcal{S}^c}\|_1-\|\bm{h}_{\mathcal{S}}\|_1 .
	\end{equation}
	Because $\bm{A}$ satisfies the $(\eta,k)$-ANSP, and because $\bm{h}\in\mathcal{N}(\bm{A})\setminus\{\bm{0}\}$ and $|\mathcal{S}|\le k$, we have
	\begin{equation}\label{eq:uniform-ansp}
		\|\bm{h}_{\mathcal{S}^c}\|_1
		>
		\|\bm{h}_{\mathcal{S}}\|_1+\eta\|\bm{h}\|_2 .
	\end{equation}
	Combining \eqref{eq:uniform-l1-lower} with \eqref{eq:uniform-ansp} yields
	\begin{equation}\label{eq:uniform-l1-gap}
		\|\bm{x}\|_1-\|\bm{x}^*\|_1
		>
		\eta\|\bm{h}\|_2 .
	\end{equation}
	Since $0\le\eta\le 1$, \eqref{eq:uniform-l1-gap} implies
	$\|\bm{x}\|_1-\|\bm{x}^*\|_1>0$.
	Hence $\|\bm{x}\|_1>\|\bm{x}^*\|_1$. From the definition of $\mathcal{R}_{\eta}$, we have
	\begin{equation}\label{eq:uniform-objective-difference}
		\begin{aligned}
			\mathcal{R}_{\eta}(\bm{x})-\mathcal{R}_{\eta}(\bm{x}^*)
			&=
			\bigl(\|\bm{x}\|_1-\|\bm{x}^*\|_1\bigr)
			\bigl(\|\bm{x}\|_1+\|\bm{x}^*\|_1\bigr) \\
			&\quad
			-\eta
			\bigl(\|\bm{x}\|_2-\|\bm{x}^*\|_2\bigr)
			\bigl(\|\bm{x}\|_2+\|\bm{x}^*\|_2\bigr).
		\end{aligned}
	\end{equation}
	The reverse triangle inequality gives $\|\bm{x}\|_2-\|\bm{x}^*\|_2\le \|\bm{x}-\bm{x}^*\|_2=\|\bm{h}\|_2$, and since $\|\bm{z}\|_2\le\|\bm{z}\|_1$ for every $\bm{z}\in\mathbb{R}^n$, we also have $\|\bm{x}\|_2+\|\bm{x}^*\|_2 \le \|\bm{x}\|_1+\|\bm{x}^*\|_1$. It follows that
	\begin{equation}\label{eq:uniform-l2-square-bound}
		\|\bm{x}\|_2^2-\|\bm{x}^*\|_2^2
		\le
		\|\bm{h}\|_2
		\bigl(\|\bm{x}\|_1+\|\bm{x}^*\|_1\bigr).
	\end{equation}
	On the other hand, \eqref{eq:uniform-l1-gap} implies
	\begin{equation}\label{eq:uniform-l1-square-lower}
		\|\bm{x}\|_1^2-\|\bm{x}^*\|_1^2
		>
		\eta\|\bm{h}\|_2
		\bigl(\|\bm{x}\|_1+\|\bm{x}^*\|_1\bigr).
	\end{equation}
	Combining \eqref{eq:uniform-l2-square-bound} and \eqref{eq:uniform-l1-square-lower}, we obtain $\|\bm{x}\|_1^2-\|\bm{x}^*\|_1^2 > \eta\left(\|\bm{x}\|_2^2-\|\bm{x}^*\|_2^2\right)$, which is equivalent to $\mathcal{R}_{\eta}(\bm{x}) > \mathcal{R}_{\eta}(\bm{x}^*)$. Thus every feasible point $\bm{x}\ne\bm{x}^*$ has strictly larger objective value than $\bm{x}^*$. Hence $\bm{x}^*$ is the unique global minimizer of $\mathcal{P}_{\eta}$. Since $\bm{x}^*$ was arbitrary among all $k$-sparse vectors, the conclusion holds uniformly for every $k$-sparse signal. This proves the theorem.
\end{proof}

The $(\eta,k)$-ANSP in Definition~\ref{def:ansp} serves as a sufficient null space condition for exact recovery by the EC-WDSN model. We next derive a spreadness-type null space condition, expressed through the ratio $\|\bm{v}\|_1/\|\bm{v}\|_2$, that implies the $(\eta,k)$-ANSP.

\begin{definition}[Null-space $\ell_1/\ell_2$ spreadness bound]
	\label{def:null-space-l1-l2-ratio}
	Let $k\in\mathbb{N}$ and $\omega>0$. We say that $\bm{A}$ satisfies the null-space $\ell_1/\ell_2$ spreadness bound of order $k$ with constant $\omega$ if, for every $\bm{v}\in\mathcal{N}(\bm{A})\setminus\{\bm{0}\}$, one has
	\begin{equation}
		\label{eq:null-space-l1-l2-ratio}
		\|\bm{v}\|_2
		\le
		\frac{\omega}{\sqrt{k}}\|\bm{v}\|_1 .
	\end{equation}
\end{definition}

\begin{proposition}[] \label{prop:nsp-spreadness-to-ansp}
	Let $0 \le \eta\le 1$. Assume that $\bm{A}$ satisfies the $\ell_1$-NSP of order $k$ with constant $\gamma\in(0,1)$, namely, 
	$\|\bm{v}_{\mathcal{S}}\|_1\le \gamma\|\bm{v}_{\mathcal{S}^c}\|_1$
	for every $\bm{v}\in\mathcal{N}(\bm{A})\setminus\{\bm{0}\}$ and every index set $\mathcal{S}\subset\{1,\ldots,n\}$ with $|\mathcal{S}|\le k$. Suppose further that $\bm{A}$ satisfies the null-space $\ell_1/\ell_2$ spreadness bound \eqref{eq:null-space-l1-l2-ratio}. If
	$\gamma+\eta\omega(1+\gamma)/\sqrt{k}<1$,
	then $\bm{A}$ satisfies the $(\eta,k)$-ANSP.
\end{proposition}

\begin{proof}
	Let $\bm{v}\in\mathcal{N}(\bm{A})\setminus\{\bm{0}\}$ and let
	$\mathcal{S}\subset\{1,\ldots,n\}$ satisfy $|\mathcal{S}|\le k$. By the assumed $\ell_1$-NSP, we have
	\begin{equation}
		\label{eq:tail-controls-full-l1}
		\|\bm{v}\|_1
		=
		\|\bm{v}_{\mathcal{S}}\|_1+\|\bm{v}_{\mathcal{S}^c}\|_1
		\le
		(1+\gamma)\|\bm{v}_{\mathcal{S}^c}\|_1 .
	\end{equation}
	Combining \eqref{eq:null-space-l1-l2-ratio} with \eqref{eq:tail-controls-full-l1}, we obtain
	\begin{equation}
		\label{eq:l2-tail-bound-spreadness}
		\|\bm{v}\|_2
		\le
		\frac{\omega}{\sqrt{k}}\|\bm{v}\|_1
		\le
		\frac{\omega(1+\gamma)}{\sqrt{k}}
		\|\bm{v}_{\mathcal{S}^c}\|_1 .
	\end{equation}
	It follows from the assumed $\ell_1$-NSP and \eqref{eq:l2-tail-bound-spreadness} that
	\begin{equation}
		\label{eq:ansp-upper-from-spreadness}
		\|\bm{v}_{\mathcal{S}}\|_1+\eta\|\bm{v}\|_2
		\le
		\gamma\|\bm{v}_{\mathcal{S}^c}\|_1
		+
		\frac{\eta\omega(1+\gamma)}{\sqrt{k}}\|\bm{v}_{\mathcal{S}^c}\|_1
		=
		\left(\gamma+\frac{\eta\omega(1+\gamma)}{\sqrt{k}}\right)
		\|\bm{v}_{\mathcal{S}^c}\|_1 .
	\end{equation}
	By the condition $\gamma+\eta\omega(1+\gamma)/\sqrt{k}<1$, the coefficient in \eqref{eq:ansp-upper-from-spreadness} is strictly less than one. Hence
	$\|\bm{v}_{\mathcal{S}}\|_1+\eta\|\bm{v}\|_2 < \|\bm{v}_{\mathcal{S}^c}\|_1$.
	Since $\bm{v}\in\mathcal{N}(\bm{A})\setminus\{\bm{0}\}$ and $\mathcal{S}$ with $|\mathcal{S}|\le k$ were arbitrary, $\bm{A}$ satisfies the $(\eta,k)$-ANSP.
\end{proof}

\begin{remark}[]
	\label{rem:null-space-spreadness-bound}
	Condition \eqref{eq:null-space-l1-l2-ratio} requires nonzero null-space vectors to be sufficiently spread in the $\ell_1/\ell_2$ sense. Equivalently, it imposes $\frac{\|\bm{v}\|_1}{\|\bm{v}\|_2}\ge\frac{\sqrt{k}}{\omega}$ for every $\bm{v}\in\mathcal{N}(\bm{A})\setminus\{\bm{0}\}$. Thus smaller values of $\omega$ exclude null-space vectors that are too concentrated. In Proposition~\ref{prop:nsp-spreadness-to-ansp}, this assumption is useful only when $\omega$ is small enough to satisfy $\omega<\frac{\sqrt{k}(1-\gamma)}{\eta(1+\gamma)}$ with $\eta>0$. By contrast, the trivial inequality $\|\bm{v}\|_2\le\|\bm{v}\|_1$ corresponds to the choice $\omega=\sqrt{k}$, and therefore does not by itself provide a nontrivial spreadness guarantee.
\end{remark}

\begin{corollary}[]
	\label{cor:nsp-ratio-uniform-exact-recovery}
	Assume that $\bm{A}$ satisfies the $\ell_1$-NSP of order $k$ with constant $\gamma\in(0,1)$ and the null-space $\ell_1/\ell_2$ spreadness bound of order $k$ with constant $\omega>0$. If $\gamma+\eta\omega(1+\gamma)/\sqrt{k}<1$, then, for every $k$-sparse vector $\bm{x}^*\in\mathbb{R}^n$ with $\bm{b}=\bm{A}\bm{x}^*$, the vector $\bm{x}^*$ is the unique global minimizer of $\mathcal{P}_{\eta}$.
\end{corollary}

\begin{proof}
	By Proposition \ref{prop:nsp-spreadness-to-ansp}, the $\ell_1$-NSP, the null-space $\ell_1/\ell_2$ spreadness bound, and condition $\gamma+\eta\omega(1+\gamma)/\sqrt{k}<1$ imply that $\bm{A}$ satisfies the $(\eta,k)$-ANSP. Therefore, Theorem~\ref{thm:uniform-exact-recovery-ansp} applies and gives the stated uniform exact recovery result.
\end{proof}

\begin{remark}[]\label{rem:eta-effect-ansp}
	Condition $\gamma+\eta\omega(1+\gamma)/\sqrt{k}<1$ shows how the parameter $\eta$ affects the null-space requirement. When $\eta=0$, the condition reduces to $\gamma<1$, which is the standard admissible range for the classical $\ell_1$-NSP. For $\eta>0$, the additional term $\frac{\eta\omega(1+\gamma)}{\sqrt{k}}$ reflects the contribution of the negative quadratic part $-\eta\|\bm{x}\|_2^2$ in the WDSN functional. Thus, for fixed $\gamma$, $\omega$, and $k$, a larger value of $\eta$ imposes a stronger null-space requirement. Conversely, a smaller spreadness-bound constant $\omega$ or a larger sparsity level $k$ makes the condition easier to satisfy.
\end{remark}

\section{Stable Recovery under a $\delta_{2k}$-RIP Condition}
\label{sec:stable-recovery-delta-2k}

The preceding section provided a null-space characterization of exact recovery for the noiseless model. We now turn to the noisy measurements and study stable recovery for $\mathcal{P}_{\eta}^{\varepsilon}$. Let the measurements satisfy $\bm{b}=\bm{A}\bm{x}^*+\bm{e}$ with $\|\bm{e}\|_2\le \varepsilon$, where $\bm{x}^*\in\mathbb{R}^n$ denotes the ground-truth signal. Throughout this section, we assume that $\bm{x}^*$ is $k$-sparse, namely $\|\bm{x}^*\|_0\le k$. We first recall the restricted isometry property (RIP).

\begin{definition}[RIP and Restricted isometry constant]
	\label{def:rip}
	For a positive integer $s$, the restricted isometry constant (RIC) $\delta_s$ of 
	$\bm{A}\in\mathbb{R}^{m\times n}$ is defined as the smallest number $\delta\ge0$ 
	such that
	\begin{equation}
		\label{eq:rip}
		(1-\delta)\|\bm{z}\|_2^2
		\le
		\|\bm{A}\bm{z}\|_2^2
		\le
		(1+\delta)\|\bm{z}\|_2^2
	\end{equation}
	for every $\bm{z}\in\mathbb{R}^n$ satisfying $\|\bm{z}\|_0\le s$. 
	The matrix $\bm{A}$ is said to satisfy the RIP of order $s$ if $\delta_s<1$.
\end{definition}

%We shall also use the following standard consequence of the RIP to control interactions between disjoint error blocks within the $2k$-sparse regime.

\begin{lemma}[Restricted Orthogonality]
	\label{lem:restricted-orthogonality}
	Suppose that $\bm{A}$ satisfies the RIP of order $s$ with constant $\delta_s$. Let $\bm{u},\bm{v}\in\mathbb{R}^n$ have disjoint supports and assume that $|\operatorname{supp}(\bm{u})\cup \operatorname{supp}(\bm{v})| \le s$. Then $|\langle \bm{A}\bm{u},\bm{A}\bm{v}\rangle| \le \delta_s\|\bm{u}\|_2\|\bm{v}\|_2$.
\end{lemma}

\begin{proof}
	If $\bm{u}=\bm{0}$ or $\bm{v}=\bm{0}$, then the asserted inequality follows immediately. It remains to consider the case $\bm{u}\neq\bm{0}$ and $\bm{v}\neq\bm{0}$. Set $\widetilde{\bm{u}}=\bm{u}/\|\bm{u}\|_2$ and $\widetilde{\bm{v}}=\bm{v}/\|\bm{v}\|_2$. Since $\bm{u}$ and $\bm{v}$ have disjoint supports, we have $\langle \widetilde{\bm{u}},\widetilde{\bm{v}}\rangle=0$, and hence $\|\widetilde{\bm{u}}+\widetilde{\bm{v}}\|_2^2 = \|\widetilde{\bm{u}}-\widetilde{\bm{v}}\|_2^2 = 2$. Moreover, both $\widetilde{\bm{u}}+\widetilde{\bm{v}}$ and $\widetilde{\bm{u}}-\widetilde{\bm{v}}$ are supported on $\operatorname{supp}(\bm{u})\cup\operatorname{supp}(\bm{v})$, whose cardinality is at most $s$. Therefore, the defining inequality for the RIC $\delta_s$ gives
	\begin{equation}
		\label{eq:rip-plus-minus}
		2(1-\delta_s)
		\le
		\|\bm{A}(\widetilde{\bm{u}}\pm\widetilde{\bm{v}})\|_2^2
		\le
		2(1+\delta_s).
	\end{equation}
	It follows from the polarization identity that
	\begin{equation}
		\label{eq:polarization-restricted-orthogonality}
		\langle \bm{A}\widetilde{\bm{u}},\bm{A}\widetilde{\bm{v}}\rangle
		=
		\frac{1}{4}
		\left(
		\|\bm{A}(\widetilde{\bm{u}}+\widetilde{\bm{v}})\|_2^2
		-
		\|\bm{A}(\widetilde{\bm{u}}-\widetilde{\bm{v}})\|_2^2
		\right).
	\end{equation}
	Using \eqref{eq:rip-plus-minus}, the difference in parentheses in \eqref{eq:polarization-restricted-orthogonality} is bounded in absolute value by $4\delta_s$. Consequently, $\left|\langle \bm{A}\widetilde{\bm{u}},\bm{A}\widetilde{\bm{v}}\rangle\right|\le \delta_s$. Multiplying this inequality by $\|\bm{u}\|_2\|\bm{v}\|_2$ yields $|\langle \bm{A}\bm{u},\bm{A}\bm{v}\rangle| \le \delta_s\|\bm{u}\|_2\|\bm{v}\|_2$. The proof is complete.
\end{proof}

Applying Lemma~\ref{lem:objective-induced-cone} to the noise-constrained feasible set $\mathcal{C}_{\varepsilon}:=\{\bm{x}\in\mathbb{R}^n:\|\bm{A}\bm{x}-\bm{b}\|_2\le\varepsilon\}$, we obtain an $\ell_1$-norm upper bound for the error $\bm{h}$ on $\mathcal{S}^c$. We next partition $\bm{h}_{\mathcal{S}^c}$ into magnitude-ordered blocks and derive upper bounds for their $\ell_2$-norms. These bounds will be used in the RIP analysis below.

\begin{lemma}[Blockwise tail bound]
	\label{lem:tail-estimate}
	Let $\mathcal{S}:=\operatorname{supp}(\bm{x}^*)$ with $|\mathcal{S}|\le k$, and let
	$\bm{\hat{x}}$ be a global minimizer of $\mathcal{P}_{\eta}^{\varepsilon}$. 
	Set $\bm{h}:=\bm{\hat{x}}-\bm{x}^*$. Let $\{\mathcal{S}_j\}_{j=1}^{J}$ be a magnitude-ordered partition of $\mathcal{S}^c$ associated with $\bm{h}$; that is,  $|\mathcal{S}_j|=k$ for $1\le j<J$, $|\mathcal{S}_J|\le k$, and $|h_i|\ge |h_\ell|$ whenever $i\in\mathcal{S}_j$, $\ell\in\mathcal{S}_{j+1}$, and $1\le j<J$. Set $\mathcal{T}_{01}:=\mathcal{S}\cup\mathcal{S}_1$. Then, we have
	\begin{equation}
		\label{eq:tail-bound}
		\sum_{j=2}^{J}\|\bm{h}_{\mathcal{S}_j}\|_2
		\le
		\|\bm{h}_{\mathcal{S}}\|_2
		+
		\frac{\eta}{\sqrt{k}}\|\bm{h}\|_2 ,
	\end{equation}
	and
	\begin{equation}
		\label{eq:tail-bound-T}
		\sum_{j=2}^{J}\|\bm{h}_{\mathcal{S}_j}\|_2
		\le
		\|\bm{h}_{\mathcal{T}_{01}}\|_2
		+
		\frac{\eta}{\sqrt{k}}\|\bm{h}\|_2 .
	\end{equation}
\end{lemma}

\begin{proof}
	If $J=1$, then the sum in \eqref{eq:tail-bound} is empty, and the assertion is immediate. Hence we assume $J\ge2$. By the magnitude-ordered construction of $\{\mathcal{S}_j\}_{j=1}^{J}$, for each $j=2,\ldots,J$, every entry of $\bm{h}_{\mathcal{S}_j}$ is no larger in magnitude than every entry of $\bm{h}_{\mathcal{S}_{j-1}}$. Since $|\mathcal{S}_{j-1}|=k$, it follows that
	\begin{equation}\label{eq:block-entry-bound}
		|h_i|
		\le
		\frac{1}{k}\|\bm{h}_{\mathcal{S}_{j-1}}\|_1,
		\quad
		i\in\mathcal{S}_j,\quad j=2,\ldots,J .
	\end{equation}
	Applying the pointwise bound \eqref{eq:block-entry-bound} to the squared
	$\ell_2$-norm of $\bm{h}_{\mathcal{S}_j}$ gives
	\begin{equation}\label{eq:block-l2-bound}
		\|\bm{h}_{\mathcal{S}_j}\|_2
		=
		\left(\sum_{i\in\mathcal{S}_j}|h_i|^2\right)^{1/2}
		\le
		\left(
		\frac{|\mathcal{S}_j|}{k^2}
		\right)^{1/2}
		\|\bm{h}_{\mathcal{S}_{j-1}}\|_1
		\le
		\frac{1}{\sqrt{k}}\|\bm{h}_{\mathcal{S}_{j-1}}\|_1 ,
		\qquad
		j=2,\ldots,J ,
	\end{equation}
	where the last inequality uses $|\mathcal{S}_j|\le k$. Summing \eqref{eq:block-l2-bound} over $j=2,\ldots,J$, one has
	\begin{equation}\label{eq:tail-first}
		\sum_{j=2}^{J}\|\bm{h}_{\mathcal{S}_j}\|_2
		\le
		\frac{1}{\sqrt{k}}
		\sum_{j=1}^{J-1}\|\bm{h}_{\mathcal{S}_j}\|_1
		\le
		\frac{1}{\sqrt{k}}\|\bm{h}_{\mathcal{S}^c}\|_1 .
	\end{equation}
	Since $\bm{x}^*$ is feasible for $\mathcal{C}_{\varepsilon}:=\{\bm{x}\in\mathbb{R}^n:\|\bm{A}\bm{x}-\bm{b}\|_2\le\varepsilon\}$ and $\bm{\hat{x}}$ is a global minimizer of $\mathcal{R}_{\eta}$ over this feasible set, Lemma~\ref{lem:objective-induced-cone} applied with $\mathcal{S}$ yields
	\begin{equation}
		\label{eq:objective-induced-cone-applied}
		\|\bm{h}_{\mathcal{S}^c}\|_1
		\le
		\|\bm{h}_{\mathcal{S}}\|_1+\eta\|\bm{h}\|_2 .
	\end{equation}
	Combining \eqref{eq:tail-first} and \eqref{eq:objective-induced-cone-applied}, we obtain
	\begin{equation}
		\label{eq:tail-after-objective-cone}
		\sum_{j=2}^{J}\|\bm{h}_{\mathcal{S}_j}\|_2
		\le
		\frac{1}{\sqrt{k}}\|\bm{h}_{\mathcal{S}}\|_1
		+
		\frac{\eta}{\sqrt{k}}\|\bm{h}\|_2 .
	\end{equation}
	Because $|\mathcal{S}|\le k$, the Cauchy--Schwarz inequality gives
	\begin{equation}
		\label{eq:support-l1-l2}
		\|\bm{h}_{\mathcal{S}}\|_1
		\le
		\sqrt{|\mathcal{S}|}\|\bm{h}_{\mathcal{S}}\|_2
		\le
		\sqrt{k}\|\bm{h}_{\mathcal{S}}\|_2 .
	\end{equation}
	Substituting \eqref{eq:support-l1-l2} into \eqref{eq:tail-after-objective-cone} gives $\sum_{j=2}^{J}\|\bm{h}_{\mathcal{S}_j}\|_2\le\|\bm{h}_{\mathcal{S}}\|_2+\frac{\eta}{\sqrt{k}}\|\bm{h}\|_2$, which proves \eqref{eq:tail-bound}. Finally, since $\mathcal{S}\subseteq\mathcal{T}_{01}$, we have $\|\bm{h}_{\mathcal{S}}\|_2\le\|\bm{h}_{\mathcal{T}_{01}}\|_2$. Therefore, \eqref{eq:tail-bound} immediately implies \eqref{eq:tail-bound-T}.
\end{proof}

Since \eqref{eq:tail-bound} and \eqref{eq:tail-bound-T} still contain $\|\bm{h}\|_2$ on the right-hand side, we need to relate the full error norm to the leading block.  The next lemma provides such a relation under the condition $\eta<\sqrt{k}$.

\begin{lemma}[Leading-block control of the error]
	\label{lem:local-global}
	Under the notation and assumptions of Lemma~\ref{lem:tail-estimate}, the following bound holds:
	\begin{equation}
		\label{eq:local-global}
		\|\bm{h}_{\mathcal{T}_{01}}\|_2
		\ge
		\frac{1-\eta/\sqrt{k}}{2}\|\bm{h}\|_2 .
	\end{equation}
\end{lemma}

\begin{proof}
	The disjoint decomposition
	$\bm{h}=\bm{h}_{\mathcal{T}_{01}}+\sum_{j=2}^{J}\bm{h}_{\mathcal{S}_j}$
	and the triangle inequality give
	\begin{equation}
		\label{eq:local-global-triangle}
		\|\bm{h}\|_2
		\le
		\|\bm{h}_{\mathcal{T}_{01}}\|_2
		+
		\sum_{j=2}^{J}\|\bm{h}_{\mathcal{S}_j}\|_2 .
	\end{equation}
	Using the blockwise tail bound \eqref{eq:tail-bound-T} in 
	\eqref{eq:local-global-triangle}, we obtain $\|\bm{h}\|_2\le2\|\bm{h}_{\mathcal{T}_{01}}\|_2+\frac{\eta}{\sqrt{k}}\|\bm{h}\|_2$ . Hence
	\begin{equation}
		\label{eq:local-global-rearranged}
		\left(1-\frac{\eta}{\sqrt{k}}\right)\|\bm{h}\|_2
		\le
		2\|\bm{h}_{\mathcal{T}_{01}}\|_2 .
	\end{equation}
	Since $k\ge1$ and $\eta\in[0,1]$, we have $1-\eta/\sqrt{k}\ge0$. Dividing
	\eqref{eq:local-global-rearranged} by $2$ yields \eqref{eq:local-global}.
\end{proof}

\begin{theorem}[Stable recovery under a $\delta_{2k}$-RIP condition]
	\label{thm:stable-recovery-delta-2k}
	Let $\bm{x}^*\in\mathbb{R}^n$ be $k$-sparse, and suppose that $\bm{b}=\bm{A}\bm{x}^*+\bm{e}$ with $\|\bm{e}\|_2\le \varepsilon$. Let $\bm{\hat{x}}$ be a global minimizer of $\mathcal{P}_{\eta}^{\varepsilon}$, and set $\bm{h}:=\bm{\hat{x}}-\bm{x}^*$. If $\bm{A}$ satisfies the RIP of order $2k$ with RIC $\delta_{2k}$ and
	\begin{equation}
		\label{eq:delta-2k-condition}
		\delta_{2k}
		<
		\frac{
			1-\frac{\eta}{\sqrt{k}}
		}{
			(1+\sqrt{2})+(\sqrt{2}-1)\frac{\eta}{\sqrt{k}}
		},
	\end{equation}
	then
	\begin{equation}
		\label{eq:stable-error-bound}
		\|\bm{\hat{x}}-\bm{x}^*\|_2
		\le
		\frac{
			4\sqrt{1+\delta_{2k}}
		}{
			\left(1-\frac{\eta}{\sqrt{k}}\right)
			-
			\delta_{2k}
			\left[
			1+\sqrt{2}
			+
			(\sqrt{2}-1)\frac{\eta}{\sqrt{k}}
			\right]
		}
		\varepsilon .
	\end{equation}
\end{theorem}

\begin{proof}
	Set $\bm{h}:=\bm{\hat{x}}-\bm{x}^*$ and $\theta:=\eta/\sqrt{k}$. Under the standing assumption $\eta\in[0,1]$, we have $\theta\ge0$. Moreover, since $\delta_{2k}\ge0$ and the denominator in \eqref{eq:delta-2k-condition} is positive, condition \eqref{eq:delta-2k-condition} implies $1-\theta>0$, and hence $\theta<1$.
	
	Since both $\bm{\hat{x}}$ and $\bm{x}^*$ are feasible for $\mathcal{P}_{\eta}^{\varepsilon}$, one has $\|\bm{A}\bm{\hat{x}}-\bm{b}\|_2\le\varepsilon$ and $\|\bm{A}\bm{x}^*-\bm{b}\|_2=\|\bm{e}\|_2\le\varepsilon$. Therefore, by the triangle inequality, one has
	\begin{equation}
		\label{eq:Ah-bound}
		\|\bm{A}\bm{h}\|_2
		=
		\|\bm{A}\bm{\hat{x}}-\bm{A}\bm{x}^*\|_2
		\le
		\|\bm{A}\bm{\hat{x}}-\bm{b}\|_2
		+
		\|\bm{b}-\bm{A}\bm{x}^*\|_2
		\le
		2\varepsilon .
	\end{equation}
	Let $\mathcal{S}:=\operatorname{supp}(\bm{x}^*)$ and $\mathcal{T}_{01}:=\mathcal{S}\cup\mathcal{S}_1$, where $\mathcal{S}_1$ contains the indices of the largest $k$ entries of $\bm{h}_{\mathcal{S}^c}$ in magnitude. Then $|\mathcal{T}_{01}|\le|\mathcal{S}|+|\mathcal{S}_1|\le2k$. Using the disjoint decomposition $\bm{h}=\bm{h}_{\mathcal{T}_{01}}+\sum_{j=2}^{J}\bm{h}_{\mathcal{S}_j}$, we obtain
	\begin{equation*}
		\label{eq:inner-expansion}
		\langle \bm{A}\bm{h}_{\mathcal{T}_{01}},\bm{A}\bm{h}\rangle
		=
		\|\bm{A}\bm{h}_{\mathcal{T}_{01}}\|_2^2
		+
		\sum_{j=2}^{J}
		\langle \bm{A}\bm{h}_{\mathcal{T}_{01}},
		\bm{A}\bm{h}_{\mathcal{S}_j}\rangle .
	\end{equation*}
	Since $\bm{h}_{\mathcal{T}_{01}}$ is supported on a set of cardinality at most $2k$, the RIP of order $2k$ gives
	\begin{equation*}
		\label{eq:main-lower}
		(1-\delta_{2k})\|\bm{h}_{\mathcal{T}_{01}}\|_2^2
		\le
		\|\bm{A}\bm{h}_{\mathcal{T}_{01}}\|_2^2
		=
		\langle \bm{A}\bm{h}_{\mathcal{T}_{01}},\bm{A}\bm{h}\rangle
		-
		\sum_{j=2}^{J}
		\langle \bm{A}\bm{h}_{\mathcal{T}_{01}},
		\bm{A}\bm{h}_{\mathcal{S}_j}\rangle .
	\end{equation*}
	Taking absolute values on the right-hand side gives
	\begin{equation}
		\label{eq:main-abs}
		(1-\delta_{2k})\|\bm{h}_{\mathcal{T}_{01}}\|_2^2
		\le
		|\langle \bm{A}\bm{h}_{\mathcal{T}_{01}},\bm{A}\bm{h}\rangle|
		+
		\sum_{j=2}^{J}
		|\langle \bm{A}\bm{h}_{\mathcal{T}_{01}},
		\bm{A}\bm{h}_{\mathcal{S}_j}\rangle| .
	\end{equation}
	By the Cauchy--Schwarz inequality, one has $|\langle \bm{A}\bm{h}_{\mathcal{T}_{01}},\bm{A}\bm{h}\rangle| \le \|\bm{A}\bm{h}_{\mathcal{T}_{01}}\|_2\|\bm{A}\bm{h}\|_2$. Using the RIP upper bound for $\bm{h}_{\mathcal{T}_{01}}$ and \eqref{eq:Ah-bound}, we get
	\begin{equation}
		\label{eq:first-term-bound}
		|\langle \bm{A}\bm{h}_{\mathcal{T}_{01}},\bm{A}\bm{h}\rangle|
		\le
		\|\bm{A}\bm{h}_{\mathcal{T}_{01}}\|_2\|\bm{A}\bm{h}\|_2
		\le
		2\varepsilon\sqrt{1+\delta_{2k}}\|\bm{h}_{\mathcal{T}_{01}}\|_2 .
	\end{equation}
	
	It remains to estimate the cross terms in \eqref{eq:main-abs} without invoking an RIP constant of order larger than $2k$. For this purpose, we split $\bm{h}_{\mathcal{T}_{01}}=\bm{h}_{\mathcal{S}}+\bm{h}_{\mathcal{S}_1}$. For every $j=2,\ldots,J$, we have
	\begin{equation}
		\label{eq:split-cross-term}
		\langle \bm{A}\bm{h}_{\mathcal{T}_{01}},
		\bm{A}\bm{h}_{\mathcal{S}_j}\rangle
		=
		\langle \bm{A}\bm{h}_{\mathcal{S}},
		\bm{A}\bm{h}_{\mathcal{S}_j}\rangle
		+
		\langle \bm{A}\bm{h}_{\mathcal{S}_1},
		\bm{A}\bm{h}_{\mathcal{S}_j}\rangle .
	\end{equation}
	The sets $\mathcal{S}$ and $\mathcal{S}_j$ are disjoint and satisfy $|\mathcal{S}\cup\mathcal{S}_j|\le|\mathcal{S}|+|\mathcal{S}_j|\le2k$. Similarly, $\mathcal{S}_1$ and $\mathcal{S}_j$ are disjoint and $|\mathcal{S}_1\cup\mathcal{S}_j|\le2k$. Therefore, Lemma~\ref{lem:restricted-orthogonality} with $s=2k$ gives
	\begin{equation}
		\label{eq:two-cross-bounds}
		|\langle \bm{A}\bm{h}_{\mathcal{S}},
		\bm{A}\bm{h}_{\mathcal{S}_j}\rangle|
		\le
		\delta_{2k}\|\bm{h}_{\mathcal{S}}\|_2\|\bm{h}_{\mathcal{S}_j}\|_2,
		\quad
		|\langle \bm{A}\bm{h}_{\mathcal{S}_1},
		\bm{A}\bm{h}_{\mathcal{S}_j}\rangle|
		\le
		\delta_{2k}\|\bm{h}_{\mathcal{S}_1}\|_2\|\bm{h}_{\mathcal{S}_j}\|_2 .
	\end{equation}
	It follows from $a+b\le\sqrt{2(a^2+b^2)}$ for $a,b\ge0$ that
	\begin{equation*}
		\|\bm{h}_{\mathcal{S}}\|_2+\|\bm{h}_{\mathcal{S}_1}\|_2
		\le
		\sqrt{2}
		\left(
		\|\bm{h}_{\mathcal{S}}\|_2^2+\|\bm{h}_{\mathcal{S}_1}\|_2^2
		\right)^{1/2}.
	\end{equation*}
	Since $\mathcal{S}$ and $\mathcal{S}_1$ are disjoint, one has $\|\bm{h}_{\mathcal{S}}\|_2^2+\|\bm{h}_{\mathcal{S}_1}\|_2^2=\|\bm{h}_{\mathcal{T}_{01}}\|_2^2$. Combining \eqref{eq:split-cross-term} and \eqref{eq:two-cross-bounds}, we obtain
	\begin{equation}
		\label{eq:sqrt2-trick}
		|\langle \bm{A}\bm{h}_{\mathcal{T}_{01}},
		\bm{A}\bm{h}_{\mathcal{S}_j}\rangle|
		\le
		\delta_{2k}
		\left(
		\|\bm{h}_{\mathcal{S}}\|_2+\|\bm{h}_{\mathcal{S}_1}\|_2
		\right)
		\|\bm{h}_{\mathcal{S}_j}\|_2
		\le
		\sqrt{2}\delta_{2k}
		\|\bm{h}_{\mathcal{T}_{01}}\|_2
		\|\bm{h}_{\mathcal{S}_j}\|_2 .
	\end{equation}
	Summing \eqref{eq:sqrt2-trick} over $j=2,\ldots,J$ yields
	\begin{equation}
		\label{eq:cross-sum-bound}
		\sum_{j=2}^{J}
		|\langle \bm{A}\bm{h}_{\mathcal{T}_{01}},
		\bm{A}\bm{h}_{\mathcal{S}_j}\rangle|
		\le
		\sqrt{2}\delta_{2k}
		\|\bm{h}_{\mathcal{T}_{01}}\|_2
		\sum_{j=2}^{J}\|\bm{h}_{\mathcal{S}_j}\|_2 .
	\end{equation}
	Substituting \eqref{eq:first-term-bound} and \eqref{eq:cross-sum-bound} into \eqref{eq:main-abs} gives
	\begin{equation}
		\label{eq:before-dividing}
		(1-\delta_{2k})\|\bm{h}_{\mathcal{T}_{01}}\|_2^2
		\le
		2\varepsilon\sqrt{1+\delta_{2k}}\|\bm{h}_{\mathcal{T}_{01}}\|_2
		+
		\sqrt{2}\delta_{2k}
		\|\bm{h}_{\mathcal{T}_{01}}\|_2
		\sum_{j=2}^{J}\|\bm{h}_{\mathcal{S}_j}\|_2 .
	\end{equation}
	If $\|\bm{h}_{\mathcal{T}_{01}}\|_2=0$, then Lemma~\ref{lem:local-global} and $\theta<1$ imply $\|\bm{h}\|_2=0$, and the desired result follows immediately. Hence we may assume that $\|\bm{h}_{\mathcal{T}_{01}}\|_2>0$. Dividing \eqref{eq:before-dividing} by $\|\bm{h}_{\mathcal{T}_{01}}\|_2$ gives
	\begin{equation}
		\label{eq:after-dividing}
		(1-\delta_{2k})\|\bm{h}_{\mathcal{T}_{01}}\|_2
		\le
		2\varepsilon\sqrt{1+\delta_{2k}}
		+
		\sqrt{2}\delta_{2k}
		\sum_{j=2}^{J}\|\bm{h}_{\mathcal{S}_j}\|_2 .
	\end{equation}
	It follows from the upper bound \eqref{eq:tail-bound-T} that
	\begin{equation}
		\label{eq:tail-bound-used}
		\sum_{j=2}^{J}\|\bm{h}_{\mathcal{S}_j}\|_2
		\le
		\|\bm{h}_{\mathcal{T}_{01}}\|_2+\theta\|\bm{h}\|_2 .
	\end{equation}
	Combining \eqref{eq:after-dividing} and \eqref{eq:tail-bound-used}, we get
	\begin{equation}
		\label{eq:pre-global}
		\left(1-(1+\sqrt{2})\delta_{2k}\right)
		\|\bm{h}_{\mathcal{T}_{01}}\|_2
		\le
		2\varepsilon\sqrt{1+\delta_{2k}}
		+
		\sqrt{2}\delta_{2k}\theta\|\bm{h}\|_2 .
	\end{equation}
	Condition \eqref{eq:delta-2k-condition} also implies $\delta_{2k}<1/(1+\sqrt{2})$. Indeed, since $\theta\ge0$, we have
	\begin{equation*}
		\label{eq:delta-smaller-basic}
		\frac{1-\theta}{(1+\sqrt{2})+(\sqrt{2}-1)\theta}
		\le
		\frac{1}{1+\sqrt{2}} .
	\end{equation*}
	Thus $1-(1+\sqrt{2})\delta_{2k}>0$. Applying Lemma~\ref{lem:local-global} to the left-hand side of \eqref{eq:pre-global} gives
	\begin{equation}
		\label{eq:coefficient-form}
		\left[
		\frac{
			\left(1-(1+\sqrt{2})\delta_{2k}\right)(1-\theta)
		}{2}
		-
		\sqrt{2}\delta_{2k}\theta
		\right]
		\|\bm{h}\|_2
		\le
		2\varepsilon\sqrt{1+\delta_{2k}} .
	\end{equation}
	
	It remains to simplify the coefficient in \eqref{eq:coefficient-form}. With $\theta=\eta/\sqrt{k}$, we have
	\begin{equation*}\label{eq:coefficient-simplification}
		\frac{\left(1-(1+\sqrt{2})\delta_{2k}\right)(1-\theta)}{2}-\sqrt{2}\delta_{2k}\theta =\frac{(1-\theta)-\delta_{2k}\left[1+\sqrt{2}+(\sqrt{2}-1)\theta\right]}{2}.
	\end{equation*}
	Therefore, \eqref{eq:coefficient-form} becomes
	\begin{equation}
		\label{eq:almost-final-bound}
		\frac{
			\left(1-\frac{\eta}{\sqrt{k}}\right)
			-
			\delta_{2k}
			\left[
			1+\sqrt{2}
			+
			(\sqrt{2}-1)\frac{\eta}{\sqrt{k}}
			\right]
		}{2}
		\|\bm{h}\|_2
		\le
		2\varepsilon\sqrt{1+\delta_{2k}} .
	\end{equation}
	By \eqref{eq:delta-2k-condition}, the coefficient in \eqref{eq:almost-final-bound} is positive. Dividing by this positive coefficient yields
	\begin{equation*}
		\|\bm{h}\|_2
		\le
		\frac{
			4\sqrt{1+\delta_{2k}}
		}{
			\left(1-\frac{\eta}{\sqrt{k}}\right)
			-
			\delta_{2k}
			\left[
			1+\sqrt{2}
			+
			(\sqrt{2}-1)\frac{\eta}{\sqrt{k}}
			\right]
		}
		\varepsilon .
	\end{equation*}
	Since $\bm{h}=\bm{\hat{x}}-\bm{x}^*$, this proves \eqref{eq:stable-error-bound}.
\end{proof}

\begin{remark}[]\label{rem:noiseless}
	In the noiseless case, namely $\varepsilon=0$, the $\ell_2$-stable recovery bound \eqref{eq:stable-error-bound} reduces to $\|\bm{\hat{x}}-\bm{x}^*\|_2=0$. Hence $\bm{\hat{x}}=\bm{x}^*$. Therefore, under the RIP condition \eqref{eq:delta-2k-condition}, NC-WDSN model $\mathcal{P}_{\eta}^{\varepsilon}$ guarantees exact recovery in the absence of measurement noise.
\end{remark}

\begin{remark}[]\label{rem:l1-reduction}
	When $\eta=0$, NC-WDSN model $\mathcal{P}_{\eta}^{\varepsilon}$ reduces to $\min_{\bm{x}\in\mathbb{R}^n}\left\{\|\bm{x}\|_1^2:\|\bm{A}\bm{x}-\bm{b}\|_2\le \varepsilon\right\}$.	Since $t\mapsto t^2$ is strictly increasing on $[0,\infty)$, this problem has the same set of global minimizers as $\min_{\bm{x}\in\mathbb{R}^n}\left\{\|\bm{x}\|_1:\|\bm{A}\bm{x}-\bm{b}\|_2\le \varepsilon\right\}$. Moreover, the RIP condition \eqref{eq:delta-2k-condition} becomes $\delta_{2k}<\frac{1}{1+\sqrt{2}}=\sqrt{2}-1$. Thus the $\eta=0$ case is consistent with the classical Cand\`es-type sufficient RIP condition for stable recovery by constrained $\ell_1$ minimization.
\end{remark}

\section{Stable Recovery for General Signals}
\label{sec:stable-recovery-general-delta-2k}

We continue to consider the NC-WDSN model $\mathcal{P}_{\eta}^{\varepsilon}$ under noisy measurements 
$\bm{b}=\bm{A}\bm{x}^*+\bm{e}$ with $\|\bm{e}\|_2\le \varepsilon$. 
Throughout this section, the target signal $\bm{x}^*\in\mathbb{R}^n$ is arbitrary and is not assumed to be exactly $k$-sparse. The sufficient $\delta_{2k}$-RIP condition remains the same as in the exactly sparse case, while the error bound gains an additional term measuring the best $k$-term approximation error of $\bm{x}^*$. 

To make this approximation term precise, we first fix a best $k$-term support of $\bm{x}^*$ and record the corresponding $\ell_1$ tail error. This support will be used throughout the block decomposition and RIP estimates below.

\begin{definition}[]\label{def:best-k-term-support}
	Let $\bm{x}^*\in\mathbb{R}^n$ be an arbitrary signal, and let $1\le k\le n$. A set $\mathcal{S}\subset\{1,\ldots,n\}$ is called a best $k$-term index set of $\bm{x}^*$ if
	$
		\mathcal{S}\in
		\argmax_{\substack{\mathcal{T}\subset\{1,\ldots,n\},\;|\mathcal{T}|=k}}
		\|\bm{x}^*_{\mathcal{T}}\|_1 .
	$
	The best $k$-term approximation error of $\bm{x}^*$ in the $\ell_1$ norm is defined by
	\begin{equation*}\label{eq:sigma-k}
		\sigma_k(\bm{x}^*)_1
		:=
		\inf_{\|\bm{z}\|_0\le k}
		\|\bm{x}^*-\bm{z}\|_1 .
	\end{equation*}
	For any best $k$-term index set $\mathcal{S}$ of $\bm{x}^*$, we have $\sigma_k(\bm{x}^*)_1=\|\bm{x}^*_{\mathcal{S}^c}\|_1$.
\end{definition}

In contrast to the exactly $k$-sparse case, $\bm{x}^*$ may have nonzero entries on $\mathcal{S}^c$. Consequently, the recovery bound contains an additional term depending on $\|\bm{x}^*_{\mathcal{S}^c}\|_1$. The following lemma makes this observation precise.

\begin{lemma}[]
	\label{lem:generalized-modified-cone}
	Let $\bm{x}^*\in\mathbb{R}^n$ be arbitrary, and let $\mathcal{S}$ be a best $k$-term index set of $\bm{x}^*$ in the sense of Definition~\ref{def:best-k-term-support}. 
	Assume that $\bm{b}=\bm{A}\bm{x}^*+\bm{e}$ with $\|\bm{e}\|_2\le \varepsilon$. 
	Let $\bm{\hat{x}}$ be a global minimizer of $\mathcal{P}_{\eta}^{\varepsilon}$ and set $\bm{h}:=\bm{\hat{x}}-\bm{x}^*$. 
	Then
	\begin{equation}
		\label{eq:generalized-cone}
		\|\bm{h}_{\mathcal{S}^c}\|_1
		\le
		\|\bm{h}_{\mathcal{S}}\|_1
		+
		\eta\|\bm{h}\|_2
		+
		2\sigma_k(\bm{x}^*)_1.
	\end{equation}
\end{lemma}

\begin{proof}
	Since $\bm{b}=\bm{A}\bm{x}^*+\bm{e}$ and $\|\bm{e}\|_2\le \varepsilon$, one has $\|\bm{A}\bm{x}^*-\bm{b}\|_2=\|\bm{e}\|_2\le \varepsilon$. Hence $\bm{x}^*$ is feasible for $\mathcal{P}_{\eta}^{\varepsilon}$. By the global minimality of $\bm{\hat{x}}$, one has
	\begin{equation*}
		\label{eq:compressible-global-optimality}
		\|\bm{\hat{x}}\|_1^2-\eta\|\bm{\hat{x}}\|_2^2
		\le
		\|\bm{x}^*\|_1^2-\eta\|\bm{x}^*\|_2^2 .
	\end{equation*}
	It follows that
	\begin{equation}
		\label{eq:compressible-l1-l2-square}
		\|\bm{\hat{x}}\|_1^2-\|\bm{x}^*\|_1^2
		\le
		\eta\bigl(\|\bm{\hat{x}}\|_2^2-\|\bm{x}^*\|_2^2\bigr).
	\end{equation}
	If $\|\bm{\hat{x}}\|_1\le \|\bm{x}^*\|_1$, then $\|\bm{\hat{x}}\|_1-\|\bm{x}^*\|_1 \le \eta\|\bm{h}\|_2 $ is immediate. It remains to treat the case $\|\bm{\hat{x}}\|_1-\|\bm{x}^*\|_1>0$. Dividing \eqref{eq:compressible-l1-l2-square} by $\|\bm{\hat{x}}\|_1+\|\bm{x}^*\|_1$ gives
	\begin{align}
		\|\bm{\hat{x}}\|_1-\|\bm{x}^*\|_1 &\le \eta \frac{ \|\bm{\hat{x}}\|_2^2-\|\bm{x}^*\|_2^2 }{ \|\bm{\hat{x}}\|_1+\|\bm{x}^*\|_1 } \le \eta \frac{ \bigl|\|\bm{\hat{x}}\|_2-\|\bm{x}^*\|_2\bigr| \bigl(\|\bm{\hat{x}}\|_2+\|\bm{x}^*\|_2\bigr) }{ \|\bm{\hat{x}}\|_1+\|\bm{x}^*\|_1 } \notag\\
		&\le \eta \|\bm{h}\|_2 \frac{ \|\bm{\hat{x}}\|_2+\|\bm{x}^*\|_2 }{ \|\bm{\hat{x}}\|_1+\|\bm{x}^*\|_1 } \le \eta\|\bm{h}\|_2 , \label{eq:compressible-l1-difference-upper-proof}
	\end{align}
	where we used the reverse triangle inequality and the norm relation $\|\bm{z}\|_2\le \|\bm{z}\|_1$. Since $\bm{\hat{x}}=\bm{x}^*+\bm{h}$, the decomposition with respect to $\mathcal{S}$ gives
	\begin{align}
		\|\bm{\hat{x}}\|_1-\|\bm{x}^*\|_1 &= \|\bm{x}^*_{\mathcal{S}}+\bm{h}_{\mathcal{S}}\|_1 + \|\bm{x}^*_{\mathcal{S}^c}+\bm{h}_{\mathcal{S}^c}\|_1 - \|\bm{x}^*_{\mathcal{S}}\|_1 - \|\bm{x}^*_{\mathcal{S}^c}\|_1 \notag\\
		&\ge \|\bm{x}^*_{\mathcal{S}}\|_1-\|\bm{h}_{\mathcal{S}}\|_1 + \|\bm{h}_{\mathcal{S}^c}\|_1-\|\bm{x}^*_{\mathcal{S}^c}\|_1 - \|\bm{x}^*_{\mathcal{S}}\|_1 - \|\bm{x}^*_{\mathcal{S}^c}\|_1 \notag\\
		&= \|\bm{h}_{\mathcal{S}^c}\|_1 - \|\bm{h}_{\mathcal{S}}\|_1 - 2\|\bm{x}^*_{\mathcal{S}^c}\|_1 . \label{eq:compressible-l1-difference-lower}
	\end{align}
	Combining $\|\bm{\hat{x}}\|_1-\|\bm{x}^*\|_1 \le \eta\|\bm{h}\|_2$ and \eqref{eq:compressible-l1-difference-lower} yields $\|\bm{h}_{\mathcal{S}^c}\|_1-\|\bm{h}_{\mathcal{S}}\|_1-2\|\bm{x}^*_{\mathcal{S}^c}\|_1 \le \eta\|\bm{h}\|_2$. It follows that
	\begin{equation*}
		\|\bm{h}_{\mathcal{S}^c}\|_1 \le \|\bm{h}_{\mathcal{S}}\|_1 + \eta\|\bm{h}\|_2 + 2\|\bm{x}^*_{\mathcal{S}^c}\|_1 .
	\end{equation*}
	Since $\mathcal{S}$ is a best $k$-term index set of $\bm{x}^*$, Definition~\ref{def:best-k-term-support} gives $\|\bm{x}^*_{\mathcal{S}^c}\|_1=\sigma_k(\bm{x}^*)_1$, and hence \eqref{eq:generalized-cone} follows, which completes proof.
\end{proof}

The next lemma uses the previous bound on $\|\bm{h}_{\mathcal{S}^c}\|_1$ to control the tail part of $\bm{h}$ in the $\ell_2$ norm. This is the form needed for applying the RIP.

\begin{lemma}[]\label{lem:general-tail-estimate}
	Let $\bm{x}^*\in\mathbb{R}^n$ be arbitrary, and let $\mathcal{S}$ be a best $k$-term index set of $\bm{x}^*$ in the sense of Definition~\ref{def:best-k-term-support}. Let $\bm{\hat{x}}$ be a global minimizer of $\mathcal{P}_{\eta}^{\varepsilon}$, and set $\bm{h}:=\bm{\hat{x}}-\bm{x}^*$. Arrange the entries of $\bm{h}_{\mathcal{S}^c}$ in nonincreasing order of magnitude, and let $\mathcal{S}_1,\mathcal{S}_2,\ldots$ be the corresponding disjoint blocks, each of cardinality $k$ except possibly the last one. Define $\mathcal{T}_{01}:=\mathcal{S}\cup\mathcal{S}_1$. Then
	\begin{equation*}
		\sum_{j\ge 2}\|\bm{h}_{\mathcal{S}_j}\|_2 \le \|\bm{h}_{\mathcal{T}_{01}}\|_2 + \frac{\eta}{\sqrt{k}}\|\bm{h}\|_2 + \frac{2}{\sqrt{k}}\sigma_k(\bm{x}^*)_1 .
	\end{equation*}
\end{lemma}

\begin{proof}
	By the ordering of the blocks, for any $j\ge2$ with $\mathcal{S}_j\neq\varnothing$, one has
	\begin{equation*}
		|h_i|\le |h_\ell|, \quad i\in\mathcal{S}_j,\ \ell\in\mathcal{S}_{j-1}.
	\end{equation*}
	It follows from $|\mathcal{S}_{j-1}|=k$ that $\|\bm{h}_{\mathcal{S}_{j-1}}\|_1 \ge k\|\bm{h}_{\mathcal{S}_j}\|_\infty$. Hence $\|\bm{h}_{\mathcal{S}_j}\|_{\infty} \le \frac{1}{k}\|\bm{h}_{\mathcal{S}_{j-1}}\|_1$ for $j\ge 2$. Consequently, for every $j\ge2$, we have
	\begin{equation}\label{eq:block-l2-l1-step}
		\|\bm{h}_{\mathcal{S}_j}\|_2 \le \sqrt{|\mathcal{S}_j|}\, \|\bm{h}_{\mathcal{S}_j}\|_{\infty} \le \sqrt{k}\, \frac{1}{k}\|\bm{h}_{\mathcal{S}_{j-1}}\|_1 = \frac{1}{\sqrt{k}}\|\bm{h}_{\mathcal{S}_{j-1}}\|_1 .
	\end{equation}
	Summing \eqref{eq:block-l2-l1-step} over $j\ge2$ yields
	\begin{equation}\label{eq:tail-first-general}
		\sum_{j\ge 2}\|\bm{h}_{\mathcal{S}_j}\|_2 \le \frac{1}{\sqrt{k}} \sum_{j\ge 2}\|\bm{h}_{\mathcal{S}_{j-1}}\|_1 \le \frac{1}{\sqrt{k}}\|\bm{h}_{\mathcal{S}^c}\|_1 .
	\end{equation}
	By Lemma~\ref{lem:generalized-modified-cone}, one has
	\begin{equation}\label{eq:cone-used-tail}
		\|\bm{h}_{\mathcal{S}^c}\|_1 \le \|\bm{h}_{\mathcal{S}}\|_1 + \eta\|\bm{h}\|_2 + 2\sigma_k(\bm{x}^*)_1 .
	\end{equation}
	Combining \eqref{eq:tail-first-general} and \eqref{eq:cone-used-tail}, we obtain
	\begin{equation}
		\label{eq:tail-before-cs}
		\sum_{j\ge 2}\|\bm{h}_{\mathcal{S}_j}\|_2 \le \frac{1}{\sqrt{k}}\|\bm{h}_{\mathcal{S}}\|_1 + \frac{\eta}{\sqrt{k}}\|\bm{h}\|_2 + \frac{2}{\sqrt{k}}\sigma_k(\bm{x}^*)_1 .
	\end{equation}
	Since $|\mathcal{S}|=k$, the Cauchy--Schwarz inequality gives $\frac{1}{\sqrt{k}}\|\bm{h}_{\mathcal{S}}\|_1 \le \|\bm{h}_{\mathcal{S}}\|_2$. Moreover, $\mathcal{S}\subseteq\mathcal{T}_{01}$, and therefore $\|\bm{h}_{\mathcal{S}}\|_2\le \|\bm{h}_{\mathcal{T}_{01}}\|_2$. Substituting this into \eqref{eq:tail-before-cs} gives
	\begin{equation*}
		\sum_{j\ge 2}\|\bm{h}_{\mathcal{S}_j}\|_2 \le \|\bm{h}_{\mathcal{T}_{01}}\|_2 + \frac{\eta}{\sqrt{k}}\|\bm{h}\|_2 + \frac{2}{\sqrt{k}}\sigma_k(\bm{x}^*)_1 ,
	\end{equation*}
	which proves the assertion.
\end{proof}

The tail estimate also shows that, to bound the whole error $\bm{h}$, it is enough to control the main block $\bm{h}_{\mathcal{T}_{01}}$ and the approximation error term.

\begin{lemma}[]\label{lem:general-local-global}
	Let $\bm{x}^*\in\mathbb{R}^n$ be arbitrary, and let $\mathcal{S}$ be a best $k$-term index set of $\bm{x}^*$ in the sense of Definition~\ref{def:best-k-term-support}. 
	Let $\bm{\hat{x}}$ be a global minimizer of $\mathcal{P}_{\eta}^{\varepsilon}$, and set $\bm{h}:=\bm{\hat{x}}-\bm{x}^*$. 
	Let $\mathcal{S}_1,\mathcal{S}_2,\ldots$ be the block decomposition of $\mathcal{S}^c$ induced by the nonincreasing rearrangement of the magnitudes of $\bm{h}_{\mathcal{S}^c}$, as in Lemma~\ref{lem:general-tail-estimate}, and define $ \mathcal{T}_{01}:=\mathcal{S}\cup\mathcal{S}_1$. For $0\le \eta\le 1$, one has
	\begin{equation}\label{eq:general-local-global}
		\|\bm{h}_{\mathcal{T}_{01}}\|_2 \ge \frac{1-\eta/\sqrt{k}}{2}\|\bm{h}\|_2 - \frac{1}{\sqrt{k}}\sigma_k(\bm{x}^*)_1 .
	\end{equation}
\end{lemma}

\begin{proof}
	By the definition of $\mathcal{T}_{01}$ and the block decomposition of $\mathcal{S}^c$, the vector $\bm{h}$ can be written as $\bm{h}=\bm{h}_{\mathcal{T}_{01}}+\sum_{j\ge2}\bm{h}_{\mathcal{S}_j}$. Consequently, the triangle inequality yields
	\begin{equation}\label{eq:local-global-triangle1}
		\|\bm{h}\|_2 \le \|\bm{h}_{\mathcal{T}_{01}}\|_2 + \sum_{j\ge2}\|\bm{h}_{\mathcal{S}_j}\|_2 .
	\end{equation}
	Applying Lemma~\ref{lem:general-tail-estimate} to the second term on the right-hand side of \eqref{eq:local-global-triangle1}, we obtain
	\begin{align}\label{eq:local-global-before-rearrange1}
		\|\bm{h}\|_2 &\le \|\bm{h}_{\mathcal{T}_{01}}\|_2 + \|\bm{h}_{\mathcal{T}_{01}}\|_2 + \frac{\eta}{\sqrt{k}}\|\bm{h}\|_2 + \frac{2}{\sqrt{k}}\sigma_k(\bm{x}^*)_1 \notag\\
		&= 2\|\bm{h}_{\mathcal{T}_{01}}\|_2 + \frac{\eta}{\sqrt{k}}\|\bm{h}\|_2 + \frac{2}{\sqrt{k}}\sigma_k(\bm{x}^*)_1 .
	\end{align}
	It follows from \eqref{eq:local-global-before-rearrange1} that
	\begin{equation}\label{eq:local-global-rearranged2}
		\left(1-\frac{\eta}{\sqrt{k}}\right)\|\bm{h}\|_2 \le 2\|\bm{h}_{\mathcal{T}_{01}}\|_2 + \frac{2}{\sqrt{k}}\sigma_k(\bm{x}^*)_1 .
	\end{equation}
	Rearranging \eqref{eq:local-global-rearranged2} and then dividing by $2$ gives
	\begin{equation*}
		\|\bm{h}_{\mathcal{T}_{01}}\|_2 \ge \frac{1-\eta/\sqrt{k}}{2}\|\bm{h}\|_2
		- \frac{1}{\sqrt{k}}\sigma_k(\bm{x}^*)_1 ,
	\end{equation*}
	which proves \eqref{eq:general-local-global}.
\end{proof}

We now give the main recovery result for general signals. The RIP condition is the same as in the exactly $k$-sparse case, while the error bound contains the additional best $k$-term approximation error.

\begin{theorem}[]\label{thm:general-stable-recovery}
	Let $\bm{x}^*\in\mathbb{R}^n$ be arbitrary, and suppose that the measurements satisfy $\bm{b}=\bm{A}\bm{x}^*+\bm{e}$ with $\|\bm{e}\|_2\le \varepsilon$. Let $\bm{\hat{x}}$ be a global minimizer of $\mathcal{P}_{\eta}^{\varepsilon}$, and set $\bm{h}:=\bm{\hat{x}}-\bm{x}^*$. 
	Let $\mathcal{S}$ be a best $k$-term index set of $\bm{x}^*$ in the sense of Definition~\ref{def:best-k-term-support}. 
	Let $\mathcal{S}_1,\mathcal{S}_2,\ldots$ be the block decomposition of $\mathcal{S}^c$ induced by the nonincreasing rearrangement of the magnitudes of $\bm{h}_{\mathcal{S}^c}$, and define $\mathcal{T}_{01}:=\mathcal{S}\cup\mathcal{S}_1$. Assume that $\bm{A}$ satisfies the RIP of order $2k$ with constant $\delta_{2k}$ and that
	\begin{equation}\label{eq:general-delta-condition}
		\delta_{2k} < \frac{1-\frac{\eta}{\sqrt{k}}}{1+\sqrt{2}+(\sqrt{2}-1)\frac{\eta}{\sqrt{k}}}.
	\end{equation}
	Then
	\begin{equation}\label{eq:general-oracle-bound}
		\|\bm{\hat{x}}-\bm{x}^*\|_2 \le C_{\varepsilon}\varepsilon + C_{\mathrm{app}} \frac{\sigma_k(\bm{x}^*)_1}{\sqrt{k}},
	\end{equation}
	where
	\begin{equation}\label{eq:general-C}
		C_{\varepsilon} := \frac{4\sqrt{1+\delta_{2k}}}{ \left(1-\frac{\eta}{\sqrt{k}}\right) - \delta_{2k} \left[ 1+\sqrt{2} + (\sqrt{2}-1)\frac{\eta}{\sqrt{k}} \right]},
	\end{equation}
	and
	\begin{equation}\label{eq:general-C-app}
		C_{\mathrm{app}} := \frac{2\left[1+(\sqrt{2}-1)\delta_{2k}\right]}{ \left(1-\frac{\eta}{\sqrt{k}}\right) - \delta_{2k} \left[ 1+\sqrt{2} + (\sqrt{2}-1)\frac{\eta}{\sqrt{k}} \right] }.
	\end{equation}
\end{theorem}

\begin{proof}
	Since both $\bm{x}^*$ and $\bm{\hat{x}}$ are feasible for $\mathcal{P}_{\eta}^{\varepsilon}$, the measurement residual satisfies
	\begin{equation}
		\label{eq:general-Ah-bound}
		\|\bm{A}\bm{h}\|_2
		=
		\|\bm{A}\bm{\hat{x}}-\bm{A}\bm{x}^*\|_2
		\le
		\|\bm{A}\bm{\hat{x}}-\bm{b}\|_2
		+
		\|\bm{A}\bm{x}^*-\bm{b}\|_2
		\le
		2\varepsilon .
	\end{equation}
	For notational brevity, write $\delta:=\delta_{2k}$ and $\theta:=\eta/\sqrt{k}$. It follows from $0\le \eta\le1$ that $0\le \theta\le1$. The RIP condition \eqref{eq:general-delta-condition} implies
	\begin{equation}
		\label{eq:general-positive-denominator}
		D:=
		1-\theta-\delta\left[1+\sqrt{2}+(\sqrt{2}-1)\theta\right]>0.
	\end{equation}
	Since $|\mathcal{T}_{01}|\le2k$, the RIP lower bound and the decomposition $\bm{h}=\bm{h}_{\mathcal{T}_{01}}+\sum_{j\ge2}\bm{h}_{\mathcal{S}_j}$ give
	\begin{align}
		(1-\delta)\|\bm{h}_{\mathcal{T}_{01}}\|_2^2 &\le \|\bm{A}\bm{h}_{\mathcal{T}_{01}}\|_2^2 = \left\langle \bm{A}\bm{h}_{\mathcal{T}_{01}}, \bm{A}\bm{h} \right\rangle - \sum_{j\ge2} \left\langle \bm{A}\bm{h}_{\mathcal{T}_{01}}, \bm{A}\bm{h}_{\mathcal{S}_j} \right\rangle \notag\\
		&\le \left| \left\langle \bm{A}\bm{h}_{\mathcal{T}_{01}}, \bm{A}\bm{h} \right\rangle \right| + \sum_{j\ge2} \left| \left\langle \bm{A}\bm{h}_{\mathcal{T}_{01}}, \bm{A}\bm{h}_{\mathcal{S}_j} \right\rangle \right|. \label{eq:general-main-inner-bound}
	\end{align}
	The first term on the right-hand side is controlled by the RIP upper bound and \eqref{eq:general-Ah-bound}, that is,
	\begin{equation}
		\label{eq:general-first-inner-term}
		\left|
		\left\langle
		\bm{A}\bm{h}_{\mathcal{T}_{01}},
		\bm{A}\bm{h}
		\right\rangle
		\right|
		\le
		\|\bm{A}\bm{h}_{\mathcal{T}_{01}}\|_2
		\|\bm{A}\bm{h}\|_2
		\le
		2\varepsilon\sqrt{1+\delta}\,
		\|\bm{h}_{\mathcal{T}_{01}}\|_2 .
	\end{equation}
	It remains to bound the cross terms. Since $\bm{h}_{\mathcal{T}_{01}}=\bm{h}_{\mathcal{S}}+\bm{h}_{\mathcal{S}_1}$, for each $j\ge2$, one has
	\begin{align}
		\left|
		\left\langle
		\bm{A}\bm{h}_{\mathcal{T}_{01}},
		\bm{A}\bm{h}_{\mathcal{S}_j}
		\right\rangle
		\right|
		&\le
		\left|
		\left\langle
		\bm{A}\bm{h}_{\mathcal{S}},
		\bm{A}\bm{h}_{\mathcal{S}_j}
		\right\rangle
		\right|
		+
		\left|
		\left\langle
		\bm{A}\bm{h}_{\mathcal{S}_1},
		\bm{A}\bm{h}_{\mathcal{S}_j}
		\right\rangle
		\right|                                      \notag\\
		&\le
		\delta
		\left(
		\|\bm{h}_{\mathcal{S}}\|_2
		+
		\|\bm{h}_{\mathcal{S}_1}\|_2
		\right)
		\|\bm{h}_{\mathcal{S}_j}\|_2
		\le
		\sqrt{2}\delta
		\|\bm{h}_{\mathcal{T}_{01}}\|_2
		\|\bm{h}_{\mathcal{S}_j}\|_2 ,
		\label{eq:general-cross-term-bound}
	\end{align}
	where the second inequality follows from the restricted orthogonality consequence of the $\delta_{2k}$-RIP applied to the disjoint pairs $(\mathcal{S},\mathcal{S}_j)$ and $(\mathcal{S}_1,\mathcal{S}_j)$, and the last inequality uses
	\begin{equation*}
		\|\bm{h}_{\mathcal{S}}\|_2+\|\bm{h}_{\mathcal{S}_1}\|_2
		\le
		\sqrt{2}
		\left(
		\|\bm{h}_{\mathcal{S}}\|_2^2+\|\bm{h}_{\mathcal{S}_1}\|_2^2
		\right)^{1/2}
		=
		\sqrt{2}\|\bm{h}_{\mathcal{T}_{01}}\|_2 .
	\end{equation*}
	Summing \eqref{eq:general-cross-term-bound} over $j\ge2$ yields
	\begin{equation}
		\label{eq:general-cross-sum}
		\sum_{j\ge2}
		\left|
		\left\langle
		\bm{A}\bm{h}_{\mathcal{T}_{01}},
		\bm{A}\bm{h}_{\mathcal{S}_j}
		\right\rangle
		\right|
		\le
		\sqrt{2}\delta
		\|\bm{h}_{\mathcal{T}_{01}}\|_2
		\sum_{j\ge2}\|\bm{h}_{\mathcal{S}_j}\|_2 .
	\end{equation}
	Combining \eqref{eq:general-main-inner-bound}, \eqref{eq:general-first-inner-term}, and \eqref{eq:general-cross-sum}, we obtain
	\begin{equation}
		\label{eq:general-before-division}
		(1-\delta)\|\bm{h}_{\mathcal{T}_{01}}\|_2^2
		\le
		2\varepsilon\sqrt{1+\delta}\,
		\|\bm{h}_{\mathcal{T}_{01}}\|_2
		+
		\sqrt{2}\delta
		\|\bm{h}_{\mathcal{T}_{01}}\|_2
		\sum_{j\ge2}\|\bm{h}_{\mathcal{S}_j}\|_2 .
	\end{equation}
	
	If $\|\bm{h}_{\mathcal{T}_{01}}\|_2=0$, Lemma~\ref{lem:general-local-global} gives $\|\bm{h}\|_2 \le \frac{2}{1-\theta} \frac{\sigma_k(\bm{x}^*)_1}{\sqrt{k}}$.	Moreover, \eqref{eq:general-positive-denominator} implies
	\begin{equation}
		C_{\mathrm{app}} - \frac{2}{1-\theta} = \frac{2\left[1+(\sqrt{2}-1)\delta\right]}{D} - \frac{2}{1-\theta} = \frac{4\sqrt{2}\delta}{D(1-\theta)} \ge 0 .
	\end{equation}
	Thus \eqref{eq:general-oracle-bound} holds in this case.
	
	Assume henceforth that $\|\bm{h}_{\mathcal{T}_{01}}\|_2>0$. Dividing \eqref{eq:general-before-division} by $\|\bm{h}_{\mathcal{T}_{01}}\|_2$ and applying Lemma~\ref{lem:general-tail-estimate} give
	\begin{align}
		(1-\delta)\|\bm{h}_{\mathcal{T}_{01}}\|_2
		&\le
		2\varepsilon\sqrt{1+\delta}
		+
		\sqrt{2}\delta
		\sum_{j\ge2}\|\bm{h}_{\mathcal{S}_j}\|_2                                      \notag\\
		&\le
		2\varepsilon\sqrt{1+\delta}
		+
		\sqrt{2}\delta
		\left(
		\|\bm{h}_{\mathcal{T}_{01}}\|_2
		+
		\theta\|\bm{h}\|_2
		+
		2\frac{\sigma_k(\bm{x}^*)_1}{\sqrt{k}}
		\right).
		\label{eq:general-after-tail}
	\end{align}
	Consequently,
	\begin{equation}
		\label{eq:general-C1-ineq}
		\left[1-(1+\sqrt{2})\delta\right]
		\|\bm{h}_{\mathcal{T}_{01}}\|_2
		\le
		2\varepsilon\sqrt{1+\delta}
		+
		\sqrt{2}\delta\theta\|\bm{h}\|_2
		+
		2\sqrt{2}\delta
		\frac{\sigma_k(\bm{x}^*)_1}{\sqrt{k}} .
	\end{equation}
	Set $C_1:=1-(1+\sqrt{2})\delta$. Since \eqref{eq:general-delta-condition} implies $\delta<1/(1+\sqrt{2})$, one has $C_1>0$. By Lemma~\ref{lem:general-local-global},
	\begin{equation}
		\label{eq:general-local-lower-used}
		\|\bm{h}_{\mathcal{T}_{01}}\|_2
		\ge
		\frac{1-\theta}{2}\|\bm{h}\|_2
		-
		\frac{\sigma_k(\bm{x}^*)_1}{\sqrt{k}} .
	\end{equation}
	Substituting \eqref{eq:general-local-lower-used} into the left-hand side of \eqref{eq:general-C1-ineq} yields
	\begin{align}
		C_1
		\left(
		\frac{1-\theta}{2}\|\bm{h}\|_2
		-
		\frac{\sigma_k(\bm{x}^*)_1}{\sqrt{k}}
		\right)
		&\le
		2\varepsilon\sqrt{1+\delta}
		+
		\sqrt{2}\delta\theta\|\bm{h}\|_2
		+
		2\sqrt{2}\delta
		\frac{\sigma_k(\bm{x}^*)_1}{\sqrt{k}} .
	\end{align}
	After collecting the terms involving $\|\bm{h}\|_2$, we obtain
	\begin{equation}
		\label{eq:general-D-ineq}
		\left[
		C_1\frac{1-\theta}{2}
		-
		\sqrt{2}\delta\theta
		\right]\|\bm{h}\|_2
		\le
		2\varepsilon\sqrt{1+\delta}
		+
		\left(C_1+2\sqrt{2}\delta\right)
		\frac{\sigma_k(\bm{x}^*)_1}{\sqrt{k}} .
	\end{equation}
	The coefficient on the left-hand side of \eqref{eq:general-D-ineq} satisfies
	\begin{equation}\label{eq:general-denominator-half}
		C_1\frac{1-\theta}{2}-\sqrt{2}\delta\theta=\frac{1-\theta-\delta\left[1+\sqrt{2}+(\sqrt{2}-1)\theta\right]}{2}=\frac{D}{2}.
	\end{equation}
	Moreover,
	\begin{equation}
		\label{eq:general-app-coeff-simplify}
		C_1+2\sqrt{2}\delta
		=
		1-(1+\sqrt{2})\delta+2\sqrt{2}\delta
		=
		1+(\sqrt{2}-1)\delta .
	\end{equation}
	Using \eqref{eq:general-denominator-half} and \eqref{eq:general-app-coeff-simplify} in \eqref{eq:general-D-ineq}, and then multiplying by $2$, gives
	\begin{equation*}
		D\|\bm{h}\|_2
		\le
		4\varepsilon\sqrt{1+\delta}
		+
		2\left[1+(\sqrt{2}-1)\delta\right]
		\frac{\sigma_k(\bm{x}^*)_1}{\sqrt{k}} .
	\end{equation*}
	Since $D>0$, division by $D$ gives
	\begin{equation*}
		\|\bm{h}\|_2
		\le
		\frac{
			4\sqrt{1+\delta}
		}{
			D
		}\varepsilon
		+
		\frac{
			2\left[1+(\sqrt{2}-1)\delta\right]
		}{
			D
		}
		\frac{\sigma_k(\bm{x}^*)_1}{\sqrt{k}} .
	\end{equation*}
	Restoring $\delta=\delta_{2k}$ and the definition of $D$ gives exactly \eqref{eq:general-oracle-bound}--\eqref{eq:general-C-app}. Since $\bm{h}=\bm{\hat{x}}-\bm{x}^*$, the proof is complete.
\end{proof}

\begin{remark}[]\label{rem:general-l1-reduction}
	When $\eta=0$, the RIP condition becomes $\delta_{2k}<\frac{1}{1+\sqrt{2}}=\sqrt{2}-1$. Moreover, the recovery bound becomes
	\begin{equation*}
		\|\bm{\hat{x}}-\bm{x}^*\|_2
		\le
		\frac{4\sqrt{1+\delta_{2k}}}{1-(1+\sqrt{2})\delta_{2k}}\varepsilon
		+
		\frac{2\left[1+(\sqrt{2}-1)\delta_{2k}\right]}
		{1-(1+\sqrt{2})\delta_{2k}}
		\frac{\|\bm{x}^*_{\mathcal{S}^c}\|_1}{\sqrt{k}}.
	\end{equation*}
	This coincides with the classical RIP-based stable recovery estimate for constrained $\ell_1$ minimization obtained by Cand\`es~\cite{CandesRombergTao2006Stable, Candes2008RIP}.
\end{remark}

\begin{remark}[]\label{rem:general-sparse-reduction}
	If $\bm{x}^*$ is exactly $k$-sparse, then choosing $\mathcal{S}=\operatorname{supp}(\bm{x}^*)$ gives $\bm{x}^*_{\mathcal{S}^c}=0$. Therefore, $\|\bm{x}^*_{\mathcal{S}^c}\|_1=0$. Thus the approximation term vanishes, and the oracle inequality \eqref{eq:general-oracle-bound} reduces to $\|\bm{\hat{x}}-\bm{x}^*\|_2\le C_{\varepsilon}\varepsilon$,  which recovers the stable recovery estimate for $k$-sparse signals.
\end{remark}

\begin{remark}[]
	The recovery constants in Theorems \ref{thm:stable-recovery-delta-2k} and \ref{thm:general-stable-recovery} contain the denominator
	\begin{equation*}
		D_{\eta,k}
		:=
		\left(1-\frac{\eta}{\sqrt{k}}\right)
		-
		\delta_{2k}
		\left[
		1+\sqrt{2}
		+
		(\sqrt{2}-1)\frac{\eta}{\sqrt{k}}
		\right].
	\end{equation*}
	The RIP condition imposed in these theorems is precisely the positivity condition $D_{\eta,k}>0$. Hence, as $\eta/\sqrt{k}$ approaches 1, the admissible upper bound on $\delta_{2k}$ tends to zero, and the recovery constants become unbounded whenever the noise level or the approximation error is nonzero. This degeneration should not be regarded merely as a byproduct of loose estimates in the proof. It reflects a structural feature of the WDSN penalty. Indeed, for any one-sparse vector $\bm{v}$, $\|\bm{v}\|_1^2-\eta\|\bm{v}\|_2^2=(1-\eta)\|\bm{v}\|_2^2 $. Thus, when $\eta\to1$, the penalty becomes increasingly weak along one-sparse directions; in the limiting case $\eta=1$, it vanishes identically on all one-sparse vectors. Consequently, the WDSN penalty loses coercivity along these directions, and the present RIP-based recovery argument cannot yield a uniform stable recovery bound near this boundary.
	
	Under the assumption $0\le\eta\le1$, this degeneracy can occur only in the nearly one-sparse case $k=1$ with $\eta$ close to one. For $k\ge2$, one has $\eta/\sqrt{k}\le1/\sqrt{2}$, and therefore the factor $1-\eta/\sqrt{k}$ remains uniformly separated from zero.
\end{remark}

The RIP condition \eqref{eq:general-delta-condition} is identical to the condition obtained in the exactly $k$-sparse case. Removing exact sparsity does not require a stronger bound on $\delta_{2k}$. Instead, the lack of exact sparsity appears in the recovery estimate only through $\|\bm{x}^*_{\mathcal{S}^c}\|_1/\sqrt{k}$. This term is small when $\bm{x}^*$ is well approximated by a $k$-sparse vector.

\section{Proximal Operator of \texorpdfstring{$\mathcal{R}_{\eta}$}{l1 squared minus eta l2 squared}}\label{sec:prox-l1sq-l2sq}

In this section, we derive the proximal operator of the function $\mathcal{R}_{\eta}$. For a given stepsize $t>0$ and a vector $\bm{v}\in\mathbb{R}^n$, the proximal operator of $t\mathcal{R}_{\eta}$ is defined by
\begin{equation}
	\label{eq:prox-def}
	\operatorname{Prox}_{t\mathcal{R}_{\eta}}(\bm{v})
	:=
	\arg\min_{\bm{x}\in\mathbb{R}^n}
	\left\{
	\mathcal{R}_{\eta}(\bm{x})+\frac{1}{2t}\|\bm{x}-\bm{v}\|_2^2
	\right\}.
\end{equation}
Multiplying the objective function in \eqref{eq:prox-def} by the positive scalar $t$, we obtain the equivalent minimization problem
\begin{equation}
	\label{eq:prox-equivalent}
	\operatorname{Prox}_{t\mathcal{R}_{\eta}}(\bm{v})
	=
	\arg\min_{\bm{x}\in\mathbb{R}^n}
	\left\{
	t\|\bm{x}\|_1^2
	-
	t\eta\|\bm{x}\|_2^2
	+
	\frac{1}{2}\|\bm{x}-\bm{v}\|_2^2
	\right\}.
\end{equation}
Expanding the last term gives $\frac{1}{2}\|\bm{x}-\bm{v}\|_2^2=\frac{1}{2}\|\bm{x}\|_2^2-\bm{v}^\top \bm{x}+\frac{1}{2}\|\bm{v}\|_2^2$. Since the constant $\frac{1}{2}\|\bm{v}\|_2^2$ does not affect the minimizers, the problem is equivalent to
\begin{equation}
	\label{eq:prox-expanded}
	\min_{\bm{x}\in\mathbb{R}^n}
	\left\{
	t\|\bm{x}\|_1^2
	+
	\frac{1-2t\eta}{2}\|\bm{x}\|_2^2
	-
	\bm{v}^\top \bm{x}
	\right\}.
\end{equation}

We first reduce \eqref{eq:prox-expanded} to a nonnegative optimization problem.

\begin{lemma}[Reduction to the nonnegative orthant]
	\label{lem:sign-reduction}
	Let $\bm{v}\in\mathbb{R}^n$, and set $\bm{q}:=|\bm{v}|$ componentwise. 
	Let $c:=1-2t\eta$ and define $F(\bm{x}):=t\|\bm{x}\|_1^2+\frac{c}{2}\|\bm{x}\|_2^2-\bm{v}^{\top}\bm{x}$. Under the standing assumption $0\le\eta\le1$, define
	\begin{equation}
		\label{eq:G-u}
		G(\bm{u})
		:=
		t(\mathbf{1}^{\top}\bm{u})^2
		+
		\frac{c}{2}\|\bm{u}\|_2^2
		-
		\bm{q}^{\top}\bm{u},
		\qquad
		\bm{u}\in\mathbb{R}_+^n .
	\end{equation}
	Then
	\begin{equation*}
		\label{eq:argmin-sign-reduction}
		\operatorname*{arg\,min}_{\bm{x}\in\mathbb{R}^n} F(\bm{x})
		=
		\left\{
		\operatorname{sign}(\bm{v})\odot\bm{u}:
		\bm{u}\in
		\operatorname*{arg\,min}_{\bm{u}\in\mathbb{R}_+^n}G(\bm{u})
		\right\},
	\end{equation*}
	where $\operatorname{sign}(0)=0$. In particular, every global minimizer $\bar{\bm{x}}$ of $F$ satisfies $\bar{x}_i v_i\ge0$ for all $i$, and $\bar{x}_i=0$ whenever $v_i=0$.
\end{lemma}

\begin{proof}
	We first note that $G$ attains its minimum on $\mathbb{R}_+^n$. For every $\bm{u}\in\mathbb{R}_+^n$, we have $\mathbf{1}^{\top}\bm{u}=\|\bm{u}\|_1$ and $\|\bm{u}\|_1^2\ge\|\bm{u}\|_2^2$. Hence
	\begin{equation*}
		t(\mathbf{1}^{\top}\bm{u})^2+\frac{c}{2}\|\bm{u}\|_2^2
		=
		t\|\bm{u}\|_1^2+\frac{1-2t\eta}{2}\|\bm{u}\|_2^2
		\ge
		\left(\frac{1}{2}+t(1-\eta)\right)\|\bm{u}\|_2^2 .
	\end{equation*}
	Under $0\le\eta\le1$, the coefficient $\frac{1}{2}+t(1-\eta)$ is positive. Since $-\bm{q}^{\top}\bm{u}\ge-\|\bm{q}\|_2\|\bm{u}\|_2$, it follows that $G$ is coercive on the closed set $\mathbb{R}_+^n$. By continuity, $G$ has a global minimizer over $\mathbb{R}_+^n$.
	
	For any $\bm{x}\in\mathbb{R}^n$, set $\bm{u}:=|\bm{x}|$. Then
	$\|\bm{x}\|_1=\mathbf{1}^{\top}\bm{u}$ and $\|\bm{x}\|_2=\|\bm{u}\|_2$. Moreover, $\bm{v}^{\top}\bm{x} \le \sum_{i=1}^n |v_i||x_i|= \bm{q}^{\top}\bm{u}$. It follows that $F(\bm{x}) \ge G(|\bm{x}|)$. Let $\bar{\bm{u}}$ be a global minimizer of $G$ over $\mathbb{R}_+^n$. 
	If $v_i=0$ and $\bar u_i>0$ for some index $i$, let $\tilde{\bm{u}}$ be obtained from $\bar{\bm{u}}$ by setting the $i$th component to zero and leaving all other components unchanged. Put $r:=\bar u_i$ and $a:=\mathbf{1}^{\top}\tilde{\bm{u}}$. Since $q_i=0$, we have
	\begin{equation*}
		G(\bar{\bm{u}})-G(\tilde{\bm{u}})=t\left[(a+r)^2-a^2\right]+\frac{c}{2}r^2=2tar+\left(t+\frac{c}{2}\right)r^2=2tar+\left(\frac{1}{2}+t(1-\eta)\right)r^2>0,
	\end{equation*}
	where we used $0\le\eta\le1$. This contradicts the global minimality of $\bar{\bm{u}}$. Therefore, $\bar u_i=0$ whenever $v_i=0$.
	
	Let $\bar{\bm{x}}$ be a global minimizer of $F$. Suppose that $\bar{x}_i v_i<0$ for some index $i$. Let $\tilde{\bm{x}}$ be obtained from $\bar{\bm{x}}$ by replacing $\bar{x}_i$ with $-\bar{x}_i$ and leaving all other components unchanged. Since this operation does not change the componentwise magnitudes, one has $\|\tilde{\bm{x}}\|_1=\|\bar{\bm{x}}\|_1$ and $\|\tilde{\bm{x}}\|_2=\|\bar{\bm{x}}\|_2$. Thus the only change in the objective value comes from the linear term. More precisely, $-\bm{v}^{\top}\tilde{\bm{x}}-\left(-\bm{v}^{\top}\bar{\bm{x}}\right)=2v_i\bar{x}_i<0$. It follows that $F(\tilde{\bm{x}})<F(\bar{\bm{x}})$, contradicting the global minimality of $\bar{\bm{x}}$. Therefore, $\bar{x}_i v_i\ge0$ for every $i$.
	
	If $v_i=0$ and $\bar x_i\ne0$, let $\tilde{\bm{x}}$ be obtained from $\bar{\bm{x}}$ by setting the $i$th component to zero and keeping all other components unchanged. Put $r:=|\bar x_i|$ and $a:=\|\tilde{\bm{x}}\|_1$. Since the linear term is unchanged, the same computation gives
	\begin{equation*}
		F(\bar{\bm{x}})-F(\tilde{\bm{x}})=t\left[(a+r)^2-a^2\right]+\frac{c}{2}r^2=2tar+\left(\frac{1}{2}+t(1-\eta)\right)r^2>0,
	\end{equation*}
	again contradicting the global minimality of $\bar{\bm{x}}$. Thus $\bar x_i=0$ whenever $v_i=0$. Together with the sign alignment, this gives $\bar{\bm{x}}=\operatorname{sign}(\bm{v})\odot|\bar{\bm{x}}|$.
	
	By $F(\bm{x}) \ge G(|\bm{x}|)$, if $|\bar{\bm{x}}|$ were not a global minimizer of $G$ over $\mathbb{R}_+^n$, then there would exist $\bm{u}\in\mathbb{R}_+^n$ such that $G(\bm{u})<G(|\bar{\bm{x}}|)$. Choosing a vector $\bm{z}$ with $|z_i|=u_i$ and $v_i z_i=|v_i|u_i$ for every $i$, we obtain $F(\bm{z})=G(\bm{u})$, and hence $F(\bm{z})=G(\bm{u})<G(|\bar{\bm{x}}|)=F(\bar{\bm{x}})$, which contradicts the global minimality of $\bar{\bm{x}}$. Therefore, $|\bar{\bm{x}}|$ is a global minimizer of $G$.
	
	Conversely, let $\bar{\bm{u}}$ be a global minimizer of $G$ over $\mathbb{R}_+^n$. It follows that $\bar u_i=0$ whenever $v_i=0$. Hence $F(\operatorname{sign}(\bm{v})\odot\bar{\bm{u}})=G(\bar{\bm{u}})$. For any $\bm{x}\in\mathbb{R}^n$, $F(\bm{x}) \ge G(|\bm{x}|)$ gives $F(\bm{x})\ge G(|\bm{x}|)\ge G(\bar{\bm{u}})=F(\operatorname{sign}(\bm{v})\odot\bar{\bm{u}})$. Thus $\operatorname{sign}(\bm{v})\odot\bar{\bm{u}}$ is a global minimizer of $F$. This proves the claimed reduction.
\end{proof}

After the sign reduction in Lemma~\ref{lem:sign-reduction}, the computation of $\operatorname{Prox}_{t\mathcal{R}_{\eta}}(\bm{v})$ reduces to minimizing $G$ over $\mathbb{R}_+^n$. The curvature parameter $c=1-2t\eta$ determines the structure of this reduced problem: $c>0$ gives a strictly convex problem, $c=0$ gives a degenerate convex problem, and $c<0$ gives a nonconvex problem. We analyze these three cases separately.

We first consider the case $c>0$, equivalently $\eta<1/(2t)$. In this case, the function $G$ is strictly convex. Indeed, $\nabla^2 G(\bm{u})=2t\mathbf{1}\mathbf{1}^\top+c\bm{I}$, and hence, for every nonzero $\bm{z}\in\mathbb{R}^n$, $\bm{z}^\top \nabla^2 G(\bm{u})\bm{z}=2t(\mathbf{1}^\top \bm{z})^2+c\|\bm{z}\|_2^2>0$. Thus $G$ is strictly convex on $\mathbb{R}_+^n$, so the reduced problem has at most one minimizer. Since $G$ is also coercive when $c>0$, this minimizer exists and is unique. We now characterize it by the KKT conditions.

\begin{lemma}[KKT characterization in the case $c>0$]
	\label{lem:kkt-c-positive}
	Assume that $c>0$. A vector $\bm{u}^*\in\mathbb{R}_+^n$ is the unique global minimizer of \eqref{eq:G-u} if and only if, with $S^*:=\mathbf{1}^{\top}\bm{u}^*$, it satisfies
	\begin{equation}
		\label{eq:u-soft-threshold-form}
		u_i^* = \max\left\{0, \frac{q_i-2tS^*}{c}\right\}, \qquad i=1,\ldots,n.
	\end{equation}
	Consequently, at the minimizer, $S^*$ satisfies the scalar fixed-point equation
	\begin{equation}\label{eq:S-fixed-point}
		S^* = \sum_{i=1}^n \max\left\{0, \frac{q_i-2tS^*}{c}\right\}.
	\end{equation}
\end{lemma}

\begin{proof}
	Since $G$ is differentiable and convex on the closed convex set $\mathbb{R}_+^n$, the KKT conditions are necessary and sufficient for global optimality. Its gradient is $\nabla G(\bm{u}) = 2t(\mathbf{1}^{\top}\bm{u})\mathbf{1} + c\bm{u} - \bm{q}$. Therefore, a vector $\bm{u}^*\in\mathbb{R}_+^n$ is a global minimizer if and only if the KKT conditions for the constraint $\bm{u}\ge0$ hold:
	\begin{equation*}
		\bm{u}^*\ge0, \quad \nabla G(\bm{u}^*)\ge0, \quad u_i^*[\nabla G(\bm{u}^*)]_i=0, \quad i=1,\ldots,n.
	\end{equation*}
	Since $S^*=\mathbf{1}^{\top}\bm{u}^*$, the $i$th component of the gradient is $[\nabla G(\bm{u}^*)]_i = 2tS^* + cu_i^* - q_i$. Hence, for each $i=1,\ldots,n$, we have $u_i^*\ge0, \quad 2tS^*+cu_i^*-q_i\ge0$ and $u_i^*(2tS^*+cu_i^*-q_i)=0$. If $u_i^*>0$, the complementarity condition gives $2tS^*+cu_i^*-q_i=0$, and hence $u_i^* = \frac{q_i-2tS^*}{c}$. In this case, necessarily $q_i>2tS^*$. If $u_i^*=0$, the dual feasibility condition gives $2tS^*-q_i\ge0$, that is, $q_i\le2tS^*$. Combining the two alternatives yields \eqref{eq:u-soft-threshold-form}. Conversely, if the displayed identity holds with $S^*=\mathbf{1}^{\top}\bm{u}^*$, then $\bm{u}^*\ge0$. Moreover, if $q_i>2tS^*$, then $u_i^*=(q_i-2tS^*)/c$ and $2tS^*+cu_i^*-q_i=0$; if $q_i\le2tS^*$, then $u_i^*=0$ and $2tS^*+cu_i^*-q_i=2tS^*-q_i\ge0$. Thus the KKT conditions hold. Summing this identity over $i=1,\ldots,n$ and using $S^*=\mathbf{1}^{\top}\bm{u}^*$ gives \eqref{eq:S-fixed-point}. Since $G$ is strictly convex when $c>0$, any global minimizer is unique. This proves the lemma.
\end{proof}

By Lemma~\ref{lem:kkt-c-positive}, the minimizer is determined once the scalar $S^*$, or equivalently the threshold $2tS^*$, is known. Since the active indices are those for which $q_i>2tS^*$, we sort the entries of $\bm{q}$ in nonincreasing order. Let
\begin{equation}
	q_{(1)}\ge q_{(2)}\ge \cdots \ge q_{(n)}\ge 0
\end{equation}
be such a rearrangement, and define
\begin{equation}
	Q_k:=\sum_{i=1}^k q_{(i)}, \qquad k=1,\ldots,n.
\end{equation}
The following theorem gives an explicit construction of $S^*$ from these sorted entries.

\begin{theorem}[]\label{thm:prox-c-positive}
	Assume that $c=1-2t\eta>0$. If $\bm{v}=\bm{0}$, then $\operatorname{Prox}_{t\mathcal{R}_{\eta}}(\bm{0})=\{\bm{0}\}$. If $\bm{v}\ne\bm{0}$, define
	\begin{equation}
		\label{eq:k-star-c-positive}
		k^* := \max\left\{k\in\{1,\ldots,n\}: q_{(k)} > \frac{2tQ_k}{c+2tk}\right\}.
	\end{equation}
	Set $S^* := \frac{Q_{k^*}}{c+2tk^*}$. Then $\operatorname{Prox}_{t\mathcal{R}_{\eta}}(\bm{v})$ is given by
	\begin{equation}\label{eq:prox-c-positive-final}
		x_i^* = \operatorname{sign}(v_i) \max\left\{0, \frac{|v_i|-2tS^*}{c}\right\}, \quad i=1,\ldots,n.
	\end{equation}
\end{theorem}

\begin{proof}
	If $\bm{v}=\bm{0}$, then $\bm{q}=\bm{0}$ and the function $G$ becomes $G(\bm{u}) = t(\mathbf{1}^{\top}\bm{u})^2 + \frac{c}{2}\|\bm{u}\|_2^2$ with $\bm{u}\in\mathbb{R}_+^n$. Since $t>0$ and $c>0$, one has $G(\bm{u})\ge0$, and equality holds if and only if $\bm{u}=\bm{0}$. Hence $\bm{0}$ is the unique minimizer of $G$ over $\mathbb{R}_+^n$. Lemma~\ref{lem:sign-reduction} therefore gives $\operatorname{Prox}_{t\mathcal{R}_{\eta}}(\bm{0}) = \{\bm{0}\}$.
	
	If $\bm{v}\ne\bm{0}$, then $q_{(1)}>0$. For $k=1$, it follows from $c>0$ that
$
		\frac{2tQ_1}{c+2t} = \frac{2tq_{(1)}}{c+2t} < q_{(1)},
$
	Therefore, the index set in \eqref{eq:k-star-c-positive} is nonempty, and $k^*$ is well defined. Let
$
		\tau^* := 2tS^* = \frac{2tQ_{k^*}}{c+2tk^*}.
$
	By the definition of $k^*$, one has $q_{(k^*)} > \tau^* $.
	It follows from the monotonicity of the sorted sequence that $q_{(i)} > \tau^*$ for $i=1,\ldots,k^* $. It remains to identify the entries beyond the first $k^*$. If $k^*=n$, then there is no such entry. Suppose that $k^*<n$. By the maximality of $k^*$, the index $k^*+1$ does not satisfy the strict inequality in \eqref{eq:k-star-c-positive}. Hence $q_{(k^*+1)} \le \frac{2tQ_{k^*+1}}{c+2t(k^*+1)}$. Since $Q_{k^*+1}=Q_{k^*}+q_{(k^*+1)}$, this implies $q_{(k^*+1)} \left(c+2t(k^*+1)\right) \le 2tQ_{k^*} + 2tq_{(k^*+1)}$, and therefore $q_{(k^*+1)}(c+2tk^*) \le 2tQ_{k^*}$. Since $c+2tk^*>0$, we obtain $q_{(k^*+1)} \le \frac{2tQ_{k^*}}{c+2tk^*} = \tau^* $. Using the monotonicity of the sorted sequence again, we obtain $q_{(i)} \le \tau^*$ for $i=k^*+1,\ldots,n$. It follows that the indices satisfying $q_i>\tau^*$ are exactly the first $k^*$ indices in the sorted order. Thus their cardinality is $k^*$, and $\sum_{i:q_i>\tau^*} q_i = Q_{k^*}$. Define $\bm{u}^*\in\mathbb{R}_+^n$ by
	\begin{equation}
		\label{eq:u-star-c-positive}
		u_i^* := \max\left\{0, \frac{q_i-\tau^*}{c}\right\}, \quad i=1,\ldots,n.
	\end{equation}
	Then
	\begin{equation*}
		\mathbf{1}^{\top}\bm{u}^* = \sum_{i:q_i>\tau^*} \frac{q_i-\tau^*}{c} = \frac{Q_{k^*}-k^*\tau^*}{c} = \frac{Q_{k^*}-2tk^*S^*}{c}.
	\end{equation*}
	Using $S^*=Q_{k^*}/(c+2tk^*)$, equivalently $Q_{k^*}=cS^*+2tk^*S^*$, we obtain $\mathbf{1}^{\top}\bm{u}^* = S^*$. Together with $\tau^*=2tS^*$, \eqref{eq:u-star-c-positive} becomes \eqref{eq:u-soft-threshold-form}, and $S^*=\mathbf{1}^{\top}\bm{u}^*$. Hence $\bm{u}^*$ satisfies the KKT condition in Lemma~\ref{lem:kkt-c-positive}. Since $G$ is strictly convex for $c>0$, $\bm{u}^*$ is the unique global minimizer of the reduced problem.
	
	Finally, Lemma~\ref{lem:sign-reduction} yields the corresponding unique minimizer of the original proximal subproblem:
	\begin{equation}
		x_i^* = \operatorname{sign}(v_i)u_i^* = \operatorname{sign}(v_i) \max\left\{0, \frac{|v_i|-2tS^*}{c}\right\}, \qquad i=1,\ldots,n.
	\end{equation}
	This proves \eqref{eq:prox-c-positive-final}.
\end{proof}

When $c=0$, the term $\frac{c}{2}\|\bm{u}\|_2^2$ vanishes from the function $G$. Thus the quadratic part of $G$ depends only on the total mass $\mathbf{1}^{\top}\bm{u}$, while the linear term determines where this mass should be placed. In particular, for a fixed value of $\mathbf{1}^{\top}\bm{u}$, the objective is minimized by assigning all mass to the indices where $q_i$ is maximal. This leads to a set-valued proximal formula.

Let $q_\infty:=\|\bm{q}\|_{\infty}=\|\bm{v}\|_{\infty}$, and define the maximal-index set $\mathcal{M}:=\{i:\ q_i=q_\infty\}$.

\begin{theorem}[]\label{thm:prox-c-zero}
	Assume that $c=1-2t\eta=0$. If $\bm{v}=\bm{0}$, then $\operatorname{Prox}_{t\mathcal{R}_{\eta}}(\bm{0}) = \{\bm{0}\}$. If $\bm{v}\ne\bm{0}$, then
	\begin{equation}
		\label{eq:prox-c-zero-final}
		\operatorname{Prox}_{t\mathcal{R}_{\eta}}(\bm{v}) = \left\{\operatorname{sign}(\bm{v})\odot\bm{u}: \bm{u}\in\mathbb{R}_+^n,\ \operatorname{supp}(\bm{u})\subseteq\mathcal{M},\ \|\bm{u}\|_1=\frac{\|\bm{v}\|_{\infty}}{2t}\right\}.
	\end{equation}
\end{theorem}

\begin{proof}
	When $c=0$, the function $G$ is $G(\bm{u}) = t(\mathbf{1}^{\top}\bm{u})^2 - \bm{q}^{\top}\bm{u}, \quad \bm{u}\in\mathbb{R}_+^n$. For any $\bm{u}\in\mathbb{R}_+^n$, set $S:=\mathbf{1}^{\top}\bm{u}=\|\bm{u}\|_1$. Since $q_i\le q_\infty$ and $u_i\ge0$ for all $i$, one has
	\begin{equation*}
		\bm{q}^{\top}\bm{u} = \sum_{i=1}^n q_i u_i \le q_\infty\sum_{i=1}^n u_i = q_\infty S .
	\end{equation*}
	It follows that
	\begin{equation}
		\label{eq:G-lower-c-zero}
		G(\bm{u}) \ge tS^2-q_\infty S .
	\end{equation}
	Moreover,
	\begin{equation}
		\label{eq:equality-condition-c-zero}
		q_\infty S-\bm{q}^{\top}\bm{u} = \sum_{i=1}^n (q_\infty-q_i)u_i .
	\end{equation}
	Each term on the right-hand side of \eqref{eq:equality-condition-c-zero} is nonnegative. Hence equality in $\bm{q}^{\top}\bm{u}\le q_\infty S$ holds if and only if $(q_\infty-q_i)u_i=0$ for $i=1,\ldots,n$. Equivalently, $\operatorname{supp}(\bm{u})\subseteq\mathcal{M}$, where $\mathcal{M}:=\{i:\ q_i=q_\infty\}$.
	
	Consider the scalar function $\phi(S):=tS^2-q_\infty S$, $S\ge0$. Since $t>0$, $\phi$ is strictly convex on $[0,\infty)$, and $\phi'(S)=2tS-q_\infty$. Thus its unique minimizer is $S^* = \frac{q_\infty}{2t}$. The lower bound \eqref{eq:G-lower-c-zero} is sharp: for any $S\ge0$ and any $j\in\mathcal{M}$, the vector $S\bm{e}_j$ satisfies $\mathbf{1}^{\top}(S\bm{e}_j)=S$ and $\bm{q}^{\top}(S\bm{e}_j)=q_\infty S$, and therefore $G(S\bm{e}_j)=\phi(S)$. Hence the global minimum of $G$ over $\mathbb{R}_+^n$ equals $\min_{S\ge0}\phi(S)$. It follows that every global minimizer of $G$ must satisfy both $\mathbf{1}^{\top}\bm{u}=S^*$ and equality in $\bm{q}^{\top}\bm{u}\le q_\infty\mathbf{1}^{\top}\bm{u}$. By the equality characterization above, this is equivalent to $\operatorname{supp}(\bm{u})\subseteq\mathcal{M}$ and $\|\bm{u}\|_1=S^*$.
	
	If $\bm{v}=\bm{0}$, then $\bm{q}=\bm{0}$ and $q_\infty=0$. Hence $S^*=0$, and the only vector $\bm{u}\in\mathbb{R}_+^n$ satisfying $\|\bm{u}\|_1=S^*$ is $\bm{u}=\bm{0}$. Thus $\bm{0}$ is the unique global minimizer of $G$ over $\mathbb{R}_+^n$. Lemma~\ref{lem:sign-reduction} gives $\operatorname{Prox}_{t\mathcal{R}_{\eta}}(\bm{0}) = \{\bm{0}\}$.
	
	If $\bm{v}\ne\bm{0}$, then $q_\infty>0$ and $S^*>0$. From the preceding argument, every global minimizer of $G$ over $\mathbb{R}_+^n$ must satisfy $\operatorname{supp}(\bm{u})\subseteq\mathcal{M}$ and $\|\bm{u}\|_1=q_\infty/(2t)$. Conversely, if $\bm{u}\in\mathbb{R}_+^n$ satisfies these two conditions, then $\mathbf{1}^{\top}\bm{u}=S^*$ and $\bm{q}^{\top}\bm{u}=q_\infty S^*$. Hence $G(\bm{u}) = t(S^*)^2-q_\infty S^* = \phi(S^*) = \min_{S\ge0}\phi(S)$. Together with the lower bound \eqref{eq:G-lower-c-zero}, this shows that such $\bm{u}$ is a global minimizer of $G$. Therefore,
	\begin{equation}
		\operatorname*{arg\,min}_{\bm{u}\in\mathbb{R}_+^n}G(\bm{u}) = \left\{\bm{u}\in\mathbb{R}_+^n: \operatorname{supp}(\bm{u})\subseteq\mathcal{M},\ \|\bm{u}\|_1=\frac{q_\infty}{2t}\right\}.
	\end{equation}
	Since $q_\infty=\|\bm{q}\|_\infty=\|\bm{v}\|_\infty$, Lemma~\ref{lem:sign-reduction} yields
	\begin{equation}
		\operatorname{Prox}_{t\mathcal{R}_{\eta}}(\bm{v}) = \left\{\operatorname{sign}(\bm{v})\odot\bm{u}: \bm{u}\in\mathbb{R}_+^n,\ \operatorname{supp}(\bm{u})\subseteq\mathcal{M},\ \|\bm{u}\|_1=\frac{\|\bm{v}\|_\infty}{2t}\right\}.
	\end{equation}
	This proves \eqref{eq:prox-c-zero-final}.
\end{proof}

We finally consider the case $c=1-2t\eta<0$. In contrast to the case $c>0$, the reduced objective is no longer convex in general. Nevertheless, its global minimizers can still be characterized by separating the total mass
$S:=\mathbf{1}^{\top}\bm{u}$ from the distribution of this mass among the coordinates. For a fixed $S\ge0$, define $\Delta_S:=\left\{\bm{u}\in\mathbb{R}_+^n:\mathbf{1}^{\top}\bm{u}=S\right\}$. The term $t(\mathbf{1}^{\top}\bm{u})^2$ is fixed. Since $c<0$, the term $\frac{c}{2}\|\bm{u}\|_2^2$ favors concentration of the mass, while the linear term $-\bm{q}^{\top}\bm{u}$ favors placing mass on coordinates where $q_i$ is largest. The following lemma identifies the minimizers of $G$ on each fixed-mass simplex.

\begin{lemma}[Minimization over fixed-mass simplexes for $c<0$]
	\label{lem:simplex-reduction-c-negative}
	Assume that $c<0$. For $S\ge0$, one has
	\begin{equation}
		\label{eq:simplex-min-c-negative}
		\min_{\bm{u}\in\Delta_S}G(\bm{u}) = \left(t+\frac{c}{2}\right)S^2-q_\infty S.
	\end{equation}
	Moreover,
	\begin{equation}
		\label{eq:simplex-argmin-c-negative}
		\operatorname*{arg\,min}_{\bm{u}\in\Delta_S}G(\bm{u}) = \left\{S\bm{e}_j:\ j\in\mathcal{M}\right\},
	\end{equation}
	where $\bm{e}_j$ denotes the $j$th canonical basis vector. In particular, when $S=0$, the set in \eqref{eq:simplex-argmin-c-negative} reduces to $\{\bm{0}\}$.
\end{lemma}

\begin{proof}
	Fix $S\ge0$ and let $\bm{u}\in\Delta_S$. Then $\bm{u}\ge0$ and $\mathbf{1}^{\top}\bm{u}=S$. Since
	\begin{equation*}
		S^2-\|\bm{u}\|_2^2 = \left(\sum_{i=1}^n u_i\right)^2-\sum_{i=1}^n u_i^2 = 2\sum_{1\le i<j\le n}u_i u_j \ge0,
	\end{equation*}
	we have $\|\bm{u}\|_2^2\le S^2$. Moreover, since $q_i\le q_\infty$ and $u_i\ge0$ for all $i$,
	\begin{equation*}
		q_\infty S-\bm{q}^{\top}\bm{u} = \sum_{i=1}^n (q_\infty-q_i)u_i \ge0.
	\end{equation*}
	Using $G(\bm{u})=tS^2+\frac{c}{2}\|\bm{u}\|_2^2-\bm{q}^{\top}\bm{u}$, we obtain
	\begin{align}
		G(\bm{u}) &-\left[\left(t+\frac{c}{2}\right)S^2-q_\infty S\right] = \frac{c}{2}\bigl(\|\bm{u}\|_2^2-S^2\bigr) + \bigl(q_\infty S-\bm{q}^{\top}\bm{u}\bigr) \notag \\
		&= -\frac{c}{2}\bigl(S^2-\|\bm{u}\|_2^2\bigr) + \sum_{i=1}^n (q_\infty-q_i)u_i \ge0, \label{eq:G-gap-c-negative}
	\end{align}
	where the last inequality follows from $c<0$, $S^2-\|\bm{u}\|_2^2\ge0$, and $q_\infty-q_i\ge0$. Hence
	\begin{equation*}
		G(\bm{u}) \ge \left(t+\frac{c}{2}\right)S^2-q_\infty S, \qquad \bm{u}\in\Delta_S.
	\end{equation*}
	
	It remains to characterize the equality case in \eqref{eq:G-gap-c-negative}. Since both terms on the right-hand side of \eqref{eq:G-gap-c-negative} are nonnegative, equality holds if and only if $S^2-\|\bm{u}\|_2^2=0$ and
	\begin{equation}
		\label{eq:equality-linear-simplex}
		(q_\infty-q_i)u_i=0, \quad i=1,\ldots,n.
	\end{equation}
	The identity $S^2-\|\bm{u}\|_2^2 = 2\sum_{1\le i<j\le n}u_i u_j$ shows that $S^2-\|\bm{u}\|_2^2=0$ is equivalent to $u_i u_j=0$ for all $i\ne j$, that is, $\bm{u}$ has at most one nonzero component. On the other hand, \eqref{eq:equality-linear-simplex} is equivalent to $\operatorname{supp}(\bm{u})\subseteq \mathcal{M}$, where $\mathcal{M}=\{i:\ q_i=q_\infty\}$.
	
	Therefore, equality in the lower bound is attained precisely when $\bm{u}$ has at most one nonzero component and this component, if present, is indexed by an element of $\mathcal{M}$. Together with the constraint $\mathbf{1}^{\top}\bm{u}=S$, this gives $\bm{u}=S\bm{e}_j$ for some $j\in\mathcal{M}$. For every such $j$, direct substitution yields
	\begin{equation*}
		G(S\bm{e}_j) = tS^2+\frac{c}{2}S^2-q_\infty S = \left(t+\frac{c}{2}\right)S^2-q_\infty S.
	\end{equation*}
	Thus the lower bound is attained exactly at the vectors $S\bm{e}_j$, $j\in\mathcal{M}$, which proves \eqref{eq:simplex-min-c-negative} and \eqref{eq:simplex-argmin-c-negative}. When $S=0$, all vectors $S\bm{e}_j$ coincide with $\bm{0}$, and hence the argmin set reduces to $\{\bm{0}\}$.
\end{proof}

By Lemma~\ref{lem:simplex-reduction-c-negative}, once the total mass $S=\mathbf{1}^{\top}\bm{u}$ is fixed, the minimum of $G$ over $\Delta_S$ is $\left(t+\frac{c}{2}\right)S^2-q_\infty S$. Since $c=1-2t\eta$, we have $t+\frac{c}{2} = \frac{1+2t(1-\eta)}{2}$. Under the standing assumption $0\le\eta\le1$, this coefficient is positive. Thus the remaining minimization over $S\ge0$ is a one-dimensional strictly convex quadratic problem. The following theorem gives the proximal formula.

\begin{theorem}[Proximal formula in the case $c<0$]
	\label{thm:prox-c-negative}
	Assume that $0\le\eta\le1$ and $c=1-2t\eta<0$. Let $q_\infty:=\|\bm{q}\|_\infty=\|\bm{v}\|_\infty$ and $\mathcal{M}:=\{i:\ |v_i|=q_\infty\}$. Define $S^* := \frac{q_\infty}{1+2t(1-\eta)}$. Then
	\begin{equation*}
		\label{eq:prox-c-negative-final}
		\operatorname{Prox}_{t\mathcal{R}_{\eta}}(\bm{v}) = \left\{\operatorname{sign}(v_j)S^*\bm{e}_j: j\in\mathcal{M}\right\}.
	\end{equation*}
	In particular, $\operatorname{Prox}_{t\mathcal{R}_{\eta}}(\bm{v})=\{\bm{0}\}$ if $\bm{v}=\bm{0}$, and it is single-valued whenever $\bm{v}\ne\bm{0}$ and the largest entry of $|\bm{v}|$ is unique.
\end{theorem}

\begin{proof}
	By Lemma~\ref{lem:simplex-reduction-c-negative}, for every fixed $S\ge0$, one has $\min_{\bm{u}\in\Delta_S}G(\bm{u}) = \left(t+\frac{c}{2}\right)S^2-q_\infty S$, and the minimizers over $\Delta_S$ are precisely $S\bm{e}_j$ with $j\in\mathcal{M}$. Since $\mathbb{R}_+^n=\bigcup_{S\ge0}\Delta_S$, the global minimization of $G$ over $\mathbb{R}_+^n$ reduces to minimizing the scalar function $\psi(S) := \left(t+\frac{c}{2}\right)S^2-q_\infty S$ with $S\ge0$.	Using $c=1-2t\eta$, we have $t+\frac{c}{2} = \frac{1+2t(1-\eta)}{2}$. Since $0\le\eta\le1$, the coefficient $1+2t(1-\eta)$ is positive. Therefore, $\psi$ is strictly convex on $[0,\infty)$. Its derivative is $\psi'(S) = \left(1+2t(1-\eta)\right)S-q_\infty$,	and hence the unique minimizer is $S^* = \frac{q_\infty}{1+2t(1-\eta)}$. It follows again from Lemma~\ref{lem:simplex-reduction-c-negative} that the global minimizers of $G$ are exactly $\operatorname*{arg\,min}_{\bm{u}\in\mathbb{R}_+^n}G(\bm{u}) = \left\{S^*\bm{e}_j: j\in\mathcal{M}\right\}$. Applying the sign reduction in Lemma~\ref{lem:sign-reduction}, we obtain $\operatorname{sign}(\bm{v})\odot(S^*\bm{e}_j) = \operatorname{sign}(v_j)S^*\bm{e}_j$ for $j\in\mathcal{M}$. Thus $\operatorname{Prox}_{t\mathcal{R}_{\eta}}(\bm{v}) = \left\{\operatorname{sign}(v_j)S^*\bm{e}_j: j\in\mathcal{M}\right\}$.
\end{proof}

\begin{remark}[]\label{rem:set-valued-prox}
	When $c>0$, the function $G$ is strictly convex, and hence $\operatorname{Prox}_{t\mathcal{R}_{\eta}}(\bm{v})$ is single-valued. In contrast, when $c=0$ or $c<0$, the proximal operator may be set-valued. In the case $c=0$, $\bm{u}$ is allowed to have support inside the maximal-index set $\mathcal{M}:=\{i:\ |v_i|=\|\bm{v}\|_{\infty}\}$, while its $\ell_1$ norm is fixed. Thus, if $\mathcal{M}$ contains more than one index and $\bm{v}\ne\bm{0}$, the proximal set contains infinitely many points. In the case $c<0$, if $\bm{v}\ne\bm{0}$, every global minimizer of $G$ is one-sparse and is supported on an index in $\mathcal{M}$. Hence, for $\bm{v}\ne\bm{0}$, the proximal operator is set-valued precisely when the largest magnitude of $\bm{v}$ is attained at more than one index. If $\bm{v}=\bm{0}$, then $S^*=0$ and the proximal set reduces to $\{\bm{0}\}$ under the standing assumption $0\le\eta\le1$.
\end{remark}

The restriction $\eta\in[0,1]$ guarantees that the proximal subproblem is well-defined for every stepsize $t>0$. Indeed, in the nonconvex case $c<0$, the relevant scalar quadratic coefficient is $\frac{1+2t(1-\eta)}{2}$. For the proximal subproblem to be bounded below for every $\bm{v}$, one needs $1+2t(1-\eta)>0$, or equivalently, $\eta<1+\frac{1}{2t}$. Thus, if one allows $\eta>1$, the proximal operator may still exist for sufficiently small $t$, provided $\eta<1+\frac{1}{2t}$. However, if $\eta\ge 1+\frac{1}{2t}$, then along a one-sparse ray $\bm{u}=S\bm{e}_j$ with $j\in\mathcal{M}$, one obtains
$
	G(S\bm{e}_j) = \frac{1+2t(1-\eta)}{2}S^2-q_\infty S.
$
If the quadratic coefficient is negative, then $G(S\bm{e}_j)\to-\infty$ as $S\to+\infty$. If the coefficient is zero and $q_\infty>0$, then again $G(S\bm{e}_j)\to-\infty$ linearly. At the boundary $1+2t(1-\eta)=0$ with $\bm{v}=\bm{0}$, the reduced problem is bounded below but has nonunique one-sparse minimizers. Hence the strict condition $1+2t(1-\eta)>0$ is the correct condition for existence for every $\bm{v}$.
	
For this reason, the range $\eta\in[0,1]$ is natural: it guarantees the nonnegativity of the penalty $\|\bm{x}\|_1^2-\eta\|\bm{x}\|_2^2\ge 0$, because $\|\bm{x}\|_2\le \|\bm{x}\|_1$, and it also ensures the unconditional existence of the proximal operator for arbitrary $t>0$.

The resulting proximal computation is summarized in Algorithm~\ref{alg:prox_l1sq_l2sq_compact}. In the branch $c>0$, the active-set search in Algorithm~\ref{alg:prox_l1sq_l2sq_compact} is dominated by sorting the magnitudes $|v_i|$. Thus the sorting-based proximal step costs $O(n\log n)$ arithmetic operations, while the cost can be reduced to $O(n)$ if a selection-based active-set search is used.

\begin{algorithm}[t]
	\caption{Compact computation of $\operatorname{Prox}_{t\mathcal{R}_{\eta}}(\bm{v})$}
	\label{alg:prox_l1sq_l2sq_compact}
	\LinesNumbered
	\KwIn{$\bm{v}\in\mathbb{R}^n$, $\eta\in[0,1]$, stepsize $t>0$}
	\KwOut{One element $\bm{x}^*\in\operatorname{Prox}_{t\mathcal{R}_{\eta}}(\bm{v})$}
	Compute magnitudes and curvature: $\bm{q}\leftarrow|\bm{v}|$, $c\leftarrow1-2t\eta$\;
	\If{$\bm{v}=\bm{0}$}{
		\Return $\bm{x}^*=\bm{0}$\;
	}
	\If{$c>0$}{
		Sort magnitudes: $q_{(1)}\ge\cdots\ge q_{(n)}$\;
		Accumulate sorted magnitudes: $Q_k\leftarrow\sum_{i=1}^k q_{(i)}$, $k=1,\ldots,n$\;
		Determine active-set size: $k^*\leftarrow\max\{k\in\{1,\ldots,n\}: q_{(k)}>2tQ_k/(c+2tk)\}$\;
		Compute total mass: $S^*\leftarrow Q_{k^*}/(c+2tk^*)$\;
		Recover signs: $x_i^*\leftarrow\operatorname{sign}(v_i)\max\{0,(|v_i|-2tS^*)/c\}$, $i=1,\ldots,n$\;
		\Return $\bm{x}^*$\;
	}
	Identify maximal entries: $q_\infty\leftarrow\|\bm{q}\|_\infty$, $\mathcal{M}\leftarrow\{i: q_i=q_\infty\}$\;
	\If{$c=0$}{
		Select critical-regime magnitude: choose $\bm{u}\in\mathbb{R}_+^n$ with $\operatorname{supp}(\bm{u})\subseteq\mathcal{M}$ and $\|\bm{u}\|_1=q_\infty/(2t)$\;
		Recover signs: $\bm{x}^*\leftarrow\operatorname{sign}(\bm{v})\odot\bm{u}$\;
		\Return $\bm{x}^*$\;
	}
	\If{$c<0$}{
		Select maximal coordinate: choose $j\in\mathcal{M}$\;
		Form one-sparse minimizer: $\bm{x}^*\leftarrow\operatorname{sign}(v_j)q_\infty\bm{e}_j/[1+2t(1-\eta)]$\;
		\Return $\bm{x}^*$\;
	}
\end{algorithm}

\section{ADMM Algorithm}\label{sec:admm-prox}

The preceding recovery results are formulated for the constrained WDSN models $\mathcal{P}_{\eta}$ and $\mathcal{P}_{\eta}^{\varepsilon}$ in \eqref{eq:WDSN_models}. These formulations are convenient for recovery analysis, because the data-fidelity constraint is stated explicitly. In the noiseless case, the constraint is $\bm{A}\bm{x}=\bm{b}$, whereas in the noisy case the residual is required to satisfy $\|\bm{A}\bm{x}-\bm{b}\|_2\le\varepsilon$.

To write the two constrained models in a unified form, we use the indicator function of a set $\mathcal{D}$, defined by $\iota_{\mathcal{D}}(\bm{y}) := 0$ if $\bm{y}\in\mathcal{D}$ and $+\infty$ otherwise. Let $\mathbb{B}_{\delta}:=\{\bm{r}\in\mathbb{R}^m:\|\bm{r}\|_2\le\delta\}$. Then both models can be written as
\begin{equation}
	\label{eq:indicator-wdsn-model}
	\min_{\bm{x}\in\mathbb{R}^n} \left\{\mathcal{R}_{\eta}(\bm{x}) + \iota_{\mathbb{B}_{\delta}}(\bm{A}\bm{x}-\bm{b})\right\}, \qquad \delta\in\{0,\varepsilon\}.
\end{equation}
When $\delta=0$, since $\mathbb{B}_{0}=\{\bm{0}\}$, the condition $\bm{A}\bm{x}-\bm{b}\in\mathbb{B}_{0}$ is equivalent to $\bm{A}\bm{x}=\bm{b}$, and \eqref{eq:indicator-wdsn-model} reduces to $\mathcal{P}_{\eta}$. When $\delta=\varepsilon$, \eqref{eq:indicator-wdsn-model} reduces to $\mathcal{P}_{\eta}^{\varepsilon}$. Thus \eqref{eq:indicator-wdsn-model} gives a common indicator-function formulation of the noiseless and noisy constrained WDSN models; the two cases differ only in the choice of the set $\mathbb{B}_{\delta}$.

For computation, we replace the indicator function in 
\eqref{eq:indicator-wdsn-model} by the standard quadratic data-fidelity term. This leads to the regularized WDSN (RWDSN) model
\begin{equation}
	\label{eq:regularized-wdsn-model}
	\min_{\bm{x}\in\mathbb{R}^n} F_\lambda(\bm{x}) = \frac12\|\bm{A}\bm{x}-\bm{b}\|_2^2+\lambda\mathcal{R}_{\eta}(\bm{x}), \quad \lambda>0 .
\end{equation}
The regularization parameter $\lambda$ determines the relative weight of the WDSN penalty with respect to the data-fidelity term.  The smaller values of $\lambda$ emphasize data fidelity, whereas larger values assign more weight to the WDSN regularization. Thus \eqref{eq:regularized-wdsn-model} should be regarded as a penalized computational formulation associated with the constrained models.

The connection with the noiseless model $\mathcal{P}_{\eta}$ can still be made precise at the level of global minimizers. Let $\bm{x}_{\lambda}$ be a global minimizer of \eqref{eq:regularized-wdsn-model}, and let $\bm{x}^{\star}$ be a global minimizer of $\mathcal{P}_{\eta}$. Since $\bm{A}\bm{x}^{\star}=\bm{b}$, the global optimality of $\bm{x}_{\lambda}$ gives
\begin{equation}
	\label{eq:penalty-path-bound}
	\frac12\|\bm{A}\bm{x}_{\lambda}-\bm{b}\|_2^2 + \lambda\mathcal{R}_{\eta}(\bm{x}_{\lambda}) \le \lambda\mathcal{R}_{\eta}(\bm{x}^{\star}).
\end{equation}
For $0\le\eta\le1$, one has $\mathcal{R}_{\eta}(\bm{x})\ge0$ for all $\bm{x}\in\mathbb{R}^n$. Hence \eqref{eq:penalty-path-bound} implies
\begin{equation}
	\label{eq:penalty-residual-bound}
	\|\bm{A}\bm{x}_{\lambda}-\bm{b}\|_2 \le \sqrt{2\lambda\mathcal{R}_{\eta}(\bm{x}^{\star})}.
\end{equation}
Thus the residual of any global minimizer of \eqref{eq:regularized-wdsn-model} tends to zero as $\lambda\downarrow0$. Moreover, \eqref{eq:penalty-path-bound} also yields $\mathcal{R}_{\eta}(\bm{x}_{\lambda}) \le \mathcal{R}_{\eta}(\bm{x}^{\star})$. Let $\lambda_j\downarrow0$ and suppose that the sequence $\{\bm{x}_{\lambda_j}\}$ has a cluster point $\bar{\bm{x}}$. Passing to a subsequence if necessary, we may assume that $\bm{x}_{\lambda_j}\to\bar{\bm{x}}$. By \eqref{eq:penalty-residual-bound}, we have $\bm{A}\bar{\bm{x}}=\bm{b}$, and hence $\bar{\bm{x}}$ is feasible for $\mathcal{P}_{\eta}$. Since $\mathcal{R}_{\eta}$ is continuous, $\mathcal{R}_{\eta}(\bm{x}_{\lambda}) \le \mathcal{R}_{\eta}(\bm{x}^{\star})$ gives $\mathcal{R}_{\eta}(\bar{\bm{x}}) \le \mathcal{R}_{\eta}(\bm{x}^{\star})$. On the other hand, since $\bm{x}^{\star}$ is a global minimizer of $\mathcal{P}_{\eta}$ and $\bar{\bm{x}}$ is equality-feasible, one also has $\mathcal{R}_{\eta}(\bm{x}^{\star}) \le \mathcal{R}_{\eta}(\bar{\bm{x}})$. Therefore, $\mathcal{R}_{\eta}(\bar{\bm{x}}) =\mathcal{R}_{\eta}(\bm{x}^{\star})$, and $\bar{\bm{x}}$ is also a global minimizer of $\mathcal{P}_{\eta}$. This argument gives a relation between the constrained model and the penalized model at the level of global minimizers.

For noisy data, let $\bm{x}_{\lambda}$ be a global minimizer of \eqref{eq:regularized-wdsn-model}, and define $ \varepsilon_{\lambda} := \|\bm{A}\bm{x}_{\lambda}-\bm{b}\|_2$. Then $\bm{x}_{\lambda}$ is a global minimizer of
\begin{equation}
	\label{eq:residual-level-model}
	\min_{\bm{x}\in\mathbb{R}^n} \left\{\mathcal{R}_{\eta}(\bm{x}): \|\bm{A}\bm{x}-\bm{b}\|_2\le\varepsilon_{\lambda}\right\}.
\end{equation}
Indeed, let $\bm{y}$ be feasible for \eqref{eq:residual-level-model}. If $\mathcal{R}_{\eta}(\bm{y})<\mathcal{R}_{\eta}(\bm{x}_{\lambda})$, then, since $\|\bm{A}\bm{y}-\bm{b}\|_2\le\varepsilon_{\lambda} =\|\bm{A}\bm{x}_{\lambda}-\bm{b}\|_2$, we would have $\frac12\|\bm{A}\bm{y}-\bm{b}\|_2^2 + \lambda\mathcal{R}_{\eta}(\bm{y}) < \frac12\|\bm{A}\bm{x}_{\lambda}-\bm{b}\|_2^2 + \lambda\mathcal{R}_{\eta}(\bm{x}_{\lambda})$, which contradicts the global minimality of $\bm{x}_{\lambda}$ for \eqref{eq:regularized-wdsn-model}. Therefore no feasible point of \eqref{eq:residual-level-model} has a smaller WDSN value than $\bm{x}_{\lambda}$, and hence $\bm{x}_{\lambda}$ solves \eqref{eq:residual-level-model}.

This observation gives the usual relationship between the penalized and constrained formulations. We now derive an ADMM scheme for \eqref{eq:regularized-wdsn-model}. Let $f(\bm{x}):=\frac12\|\bm{A}\bm{x}-\bm{b}\|_2^2$. Introducing a consensus variable separates the smooth data-fidelity term from the nonsmooth WDSN penalty:
\begin{equation}
	\label{eq:regularized-split}
	\begin{aligned}
		\min_{\bm{x},\bm{z}\in\mathbb{R}^n}\quad & f(\bm{x})+\lambda \mathcal{R}_{\eta}(\bm{z}) \\
		\mathrm{s.t.}\quad & \bm{x}-\bm{z}=\bm{0}.
	\end{aligned}
\end{equation}
With scaled dual variable $\bm{u}$ and penalty parameter $\rho>0$, the scaled augmented Lagrangian is
\begin{equation}
	\label{eq:regularized-scaled-lagrangian}
	\mathcal{L}_{\rho}(\bm{x},\bm{z},\bm{u}) = f(\bm{x})+\lambda\mathcal{R}_{\eta}(\bm{z}) +\frac{\rho}{2}\|\bm{x}-\bm{z}+\bm{u}\|_2^2 -\frac{\rho}{2}\|\bm{u}\|_2^2 .
\end{equation}
The corresponding ADMM iteration is written in the $z$-first Gauss--Seidel order:
\begin{subequations}
	\label{eq:regularized-admm}
	\begin{empheq}[left=\empheqlbrace]{align}
		\bm{z}^{(k+1)} &\in \arg\min_{\bm{z}\in\mathbb{R}^n} \left\{\lambda \mathcal{R}_{\eta}(\bm{z}) +\frac{\rho}{2}\|\bm{x}^{(k)}-\bm{z}+\bm{u}^{(k)}\|_2^2\right\}, \label{eq:regularized-z-update}\\
		\bm{x}^{(k+1)} &= \arg\min_{\bm{x}\in\mathbb{R}^n} \left\{\frac12\|\bm{A}\bm{x}-\bm{b}\|_2^2 +\frac{\rho}{2}\|\bm{x}-\bm{z}^{(k+1)}+\bm{u}^{(k)}\|_2^2\right\}, \label{eq:regularized-x-update}\\
		\bm{u}^{(k+1)} &= \bm{u}^{(k)}+\bm{x}^{(k+1)}-\bm{z}^{(k+1)}. \label{eq:regularized-u-update}
	\end{empheq}
\end{subequations}

The $\bm{z}$-subproblem is exactly a WDSN proximal step. Define $\bm{p}^{(k)}:=\bm{x}^{(k)}+\bm{u}^{(k)}$. After dividing \eqref{eq:regularized-z-update} by $\lambda>0$, we obtain
\begin{equation*}
	\bm{z}^{(k+1)} \in \arg\min_{\bm{z}\in\mathbb{R}^n} \left\{\mathcal{R}_{\eta}(\bm{z}) +\frac{\rho}{2\lambda}\|\bm{z}-\bm{p}^{(k)}\|_2^2\right\}.
\end{equation*}
It follows that $\bm{z}^{(k+1)} \in \operatorname{Prox}_{(\lambda/\rho)\mathcal{R}_{\eta}}(\bm{p}^{(k)})$.

The $\bm{x}$-subproblem is a strongly convex quadratic problem. Its first-order optimality condition is
$
	\bm{A}^{\top}(\bm{A}\bm{x}^{(k+1)}-\bm{b}) + \rho(\bm{x}^{(k+1)}-\bm{z}^{(k+1)}+\bm{u}^{(k)}) =0.
$
Equivalently, $\bm{x}^{(k+1)}$ is obtained by solving the fixed-coefficient linear system
\begin{equation}
	\label{eq:regularized-x-linear-system}
	(\bm{A}^{\top}\bm{A}+\rho\bm{I})\bm{x}^{(k+1)} = \bm{A}^{\top}\bm{b}+\rho(\bm{z}^{(k+1)}-\bm{u}^{(k)}).
\end{equation}
Since $\rho>0$, the matrix $\bm{A}^{\top}\bm{A}+\rho\bm{I}$ is positive definite, so this subproblem has a unique minimizer. In implementation, the coefficient matrix is unchanged across iterations, and its factorization can be reused.

Algorithm~\ref{alg:admm_regularized} summarizes the ADMM scheme for the penalized model \eqref{eq:regularized-wdsn-model}. The same penalized solver can be applied to data generated from either noiseless or noisy measurement models; the noise level affects the practical choice of the regularization parameter $\lambda$.

\begin{algorithm}[t]
	\caption{$z$-first ADMM for the regularized WDSN model}
	\label{alg:admm_regularized}
	\LinesNumbered
	\KwIn{$\bm{A}$, $\bm{b}$, $\eta\in[0,1]$, $\lambda>0$, $\rho>0$, $\bm{x}^{(0)}$, $\bm{u}^{(0)}$, $\varepsilon_{\mathrm{pri}}>0$, $\varepsilon_{\mathrm{dual}}>0$, $K$}
	\KwOut{An approximate solution of \eqref{eq:regularized-wdsn-model}}
	
	Precompute a factorization of $\bm{A}^{\top}\bm{A}+\rho\bm{I}$\;
	
	\For{$k=0,1,2,\ldots,K-1$}{
		Form proximal input: $\bm{p}^{(k)}\leftarrow \bm{x}^{(k)}+\bm{u}^{(k)}$\;
		Choose $\bm{z}^{(k+1)}\in\operatorname{Prox}_{(\lambda/\rho)\mathcal{R}_{\eta}}(\bm{p}^{(k)})$\;
		Assemble right-hand side: $\bm{h}^{(k+1)}\leftarrow \bm{A}^{\top}\bm{b}+\rho(\bm{z}^{(k+1)}-\bm{u}^{(k)})$\;
		Solve quadratic subproblem: $(\bm{A}^{\top}\bm{A}+\rho\bm{I})\bm{x}^{(k+1)}=\bm{h}^{(k+1)}$\;
		Update scaled dual variable: $\bm{u}^{(k+1)}\leftarrow\bm{u}^{(k)}+\bm{x}^{(k+1)}-\bm{z}^{(k+1)}$\;
		Compute primal residual: $\bm{r}^{(k+1)}\leftarrow\bm{x}^{(k+1)}-\bm{z}^{(k+1)}$\;
		Compute dual residual: $\bm{s}^{(k+1)}\leftarrow\rho(\bm{x}^{(k+1)}-\bm{x}^{(k)})$\;
		
		\If{$\|\bm{r}^{(k+1)}\|_2\le\varepsilon_{\mathrm{pri}}$ and $\|\bm{s}^{(k+1)}\|_2\le\varepsilon_{\mathrm{dual}}$}{
			\Return $\bm{z}^{(k+1)}$\;
		}
	}
	
	\Return $\bm{z}^{(K)}$\;
\end{algorithm}

\section{Convergence Analysis}
\label{sec:admm-convergence}

In this section, we establish the global convergence of the proposed ADMM scheme (Algorithm \ref{alg:admm_regularized}). Throughout the rest of this section, let
\begin{equation*}
	g(\bm{z}):=\lambda\mathcal{R}_\eta(\bm{z}),\quad F_\lambda:=f+\lambda\mathcal{R}_\eta,\quad L_f:=\|\bm{A}^{\top}\bm{A}\|_2,
\end{equation*}
and write $\bm{w}^{(k)}:=(\bm{x}^{(k)},\bm{z}^{(k)},\bm{u}^{(k)})$ and $\mathcal{L}_\rho^k:=\mathcal{L}_\rho(\bm{x}^{(k)},\bm{z}^{(k)},\bm{u}^{(k)})$. For $\bm{s}\in\mathbb{R}^n$, we use $\operatorname{SGN}(\bm{s})$ for the set-valued sign vector, that is,
$
	\operatorname{SGN}(\bm{s}) :=\{\bm{\xi}\in\mathbb{R}^n:\ \xi_i=\operatorname{sign}(s_i)\ \text{if }s_i\ne0,\ \xi_i\in[-1,1]\ \text{if }s_i=0\}.
$

\begin{theorem}[Existence of global minimizers]
	\label{thm:existence-regularized-wdsn}
	Let $\bm{A}\in\mathbb{R}^{m\times n}$, $\bm{b}\in\mathbb{R}^m$, $\lambda>0$, and $\eta\in[0,1]$.  Assume that either $0\le\eta<1$, or $\eta=1$ and
	\begin{equation}
		\label{eq:no-degenerate-direction}
		\left\{\bm{d}\in\mathbb{R}^n\setminus\{\bm{0}\}: \bm{A}\bm{d}=\bm{0},\ \mathcal{R}_1(\bm{d})=0\right\} = \varnothing .
	\end{equation}
	Then $F_\lambda(\bm{x})$ is coercive. Consequently, the solution set of \eqref{eq:regularized-wdsn-model} is nonempty and compact.
\end{theorem}

\begin{proof}
	Since $\|\bm{x}\|_1\ge \|\bm{x}\|_2$, for every $\bm{x}\in\mathbb{R}^n$ one has $\mathcal{R}_\eta(\bm{x})=\|\bm{x}\|_1^2-\eta\|\bm{x}\|_2^2 \ge (1-\eta)\|\bm{x}\|_2^2\ge0$. If $0\le\eta<1$, then $F_\lambda(\bm{x})\ge \lambda(1-\eta)\|\bm{x}\|_2^2$, and hence $F_\lambda(\bm{x})\to+\infty$ as $\|\bm{x}\|_2\to+\infty$. Thus $F_\lambda$ is coercive in this case.
	
	It remains to consider $\eta=1$. In this case, $\mathcal{R}_1(\bm{x})=\|\bm{x}\|_1^2-\|\bm{x}\|_2^2 =2\sum_{1\le i<j\le n}|x_i||x_j|\ge0$. For every $\bm{x}\neq\bm{0}$, write $\bm{x}=t\bm{d}$, where $t:=\|\bm{x}\|_2>0$ and $\bm{d}:=\bm{x}/\|\bm{x}\|_2$, so that $\|\bm{d}\|_2=1$. Define $\Theta(\bm{d}):=\frac12\|\bm{A}\bm{d}\|_2^2+\lambda\mathcal{R}_1(\bm{d})$ for $\|\bm{d}\|_2=1$. By \eqref{eq:no-degenerate-direction} and the nonnegativity of both terms in $\Theta$, we have $\Theta(\bm{d})>0$ on the unit sphere $\mathbb{S}^{n-1}:=\{\bm{d}\in\mathbb{R}^n:\|\bm{d}\|_2=1\}$. Since $\Theta$ is continuous and $\mathbb{S}^{n-1}$ is compact, there exists $c>0$ such that $\Theta(\bm{d})\ge c$ for all $\bm{d}\in\mathbb{S}^{n-1}$. Using the quadratic homogeneity of $\mathcal{R}_1$, we get
	\begin{equation*}
		F_\lambda(t\bm{d})= \frac12\|t\bm{A}\bm{d}-\bm{b}\|_2^2 + \lambda t^2\mathcal{R}_1(\bm{d}) = t^2\Theta(\bm{d}) - t\langle \bm{A}\bm{d},\bm{b}\rangle + \frac12\|\bm{b}\|_2^2 \ge ct^2 - t\|\bm{A}\|_2\|\bm{b}\|_2 + \frac12\|\bm{b}\|_2^2 .
	\end{equation*}
	The right-hand side tends to $+\infty$ as $t\to+\infty$. Since $t=\|\bm{x}\|_2$, $F_\lambda(\bm{x})\to+\infty$ as $\|\bm{x}\|_2\to+\infty$. Hence $F_\lambda$ is coercive.
	
	Thus $F_\lambda$ is continuous and coercive under the stated assumptions. Therefore, its lower level sets are compact. Let $\{\bm{x}^k\}$ be a minimizing sequence. Since $F_\lambda(\bm{0})<+\infty$, this sequence is eventually contained in a compact lower level set and hence admits a convergent subsequence $\bm{x}^{k_j}\to\bm{x}^\star$. By continuity, $F_\lambda(\bm{x}^\star)=\inf_{\bm{x}\in\mathbb{R}^n}F_\lambda(\bm{x})$. Thus $\bm{x}^\star$ is a global minimizer. The solution set is closed by continuity and bounded by coercivity, and is therefore compact.
\end{proof}

\begin{lemma}[Bregman-Lagrangian identity and lower bound]
	\label{lem:bregman_identity}
	Let $\{\bm{w}^{(k)}\}$ be generated by the exact ADMM iteration \eqref{eq:regularized-admm}. For every $k\ge1$,
	\begin{equation}
		\label{eq:dual_mapping}
		\rho\bm{u}^{(k)}=-\nabla f(\bm{x}^{(k)}).
	\end{equation}
	Moreover,
	\begin{equation}
		\label{eq:bregman_identity}
		\mathcal{L}_\rho^k
		= F_\lambda(\bm{z}^{(k)})
		-\frac{1}{2}\|\bm{A}(\bm{z}^{(k)}-\bm{x}^{(k)})\|_2^2
		+\frac{\rho}{2}\|\bm{x}^{(k)}-\bm{z}^{(k)}\|_2^2,
	\end{equation}
	where $F_\lambda=f+\lambda\mathcal{R}_\eta$. Consequently,
	\begin{equation}
		\label{eq:al_geometric_bound1}
		\mathcal{L}_\rho^k
		\ge F_\lambda(\bm{z}^{(k)})
		+\frac{\rho-L_f}{2}\|\bm{x}^{(k)}-\bm{z}^{(k)}\|_2^2 .
	\end{equation}
\end{lemma}
\begin{proof}
	For the quadratic function $f(\bm{x})=\frac12\|\bm{A}\bm{x}-\bm{b}\|_2^2$, the Bregman divergence satisfies
	\begin{equation} \label{eq:bregman_def}
		\mathcal{D}_f(\bm{z}, \bm{x}) := f(\bm{z}) - f(\bm{x}) - \langle \nabla f(\bm{x}), \bm{z} - \bm{x} \rangle = \frac{1}{2}\|\bm{A}(\bm{z} - \bm{x})\|_2^2.
	\end{equation}
	For $k\ge1$, the first-order optimality condition of the $\bm{x}$-subproblem is
	\begin{equation*}
		\nabla f(\bm{x}^{(k)})+\rho(\bm{x}^{(k)}-\bm{z}^{(k)}+\bm{u}^{(k-1)})=\bm{0}.
	\end{equation*}
	Combining this relation with the dual update $\bm{u}^{(k)}=\bm{u}^{(k-1)}+\bm{x}^{(k)}-\bm{z}^{(k)}$ gives $\rho\bm{u}^{(k)}=-\nabla f(\bm{x}^{(k)})$. Moreover, substituting $\rho\bm{u}^{(k)}=-\nabla f(\bm{x}^{(k)})$ into the augmented Lagrangian function gives
	\begin{align}
		\mathcal{L}_\rho^k &= f(\bm{x}^{(k)}) + \lambda \mathcal{R}_\eta(\bm{z}^{(k)}) + \langle \rho\bm{u}^{(k)}, \bm{x}^{(k)} - \bm{z}^{(k)} \rangle + \frac{\rho}{2}\|\bm{x}^{(k)} - \bm{z}^{(k)}\|_2^2 \nonumber \\
		&= f(\bm{x}^{(k)}) - \langle \nabla f(\bm{x}^{(k)}), \bm{x}^{(k)} - \bm{z}^{(k)} \rangle + \lambda \mathcal{R}_\eta(\bm{z}^{(k)}) + \frac{\rho}{2}\|\bm{x}^{(k)} - \bm{z}^{(k)}\|_2^2. \label{eq:al_expansion}
	\end{align}
	It follows from the $\mathcal{D}_f(\bm{z}, \bm{x})$ that
	\begin{equation*}
		f(\bm{x}^{(k)})-\langle\nabla f(\bm{x}^{(k)}),\bm{x}^{(k)}-\bm{z}^{(k)}\rangle
		=f(\bm{z}^{(k)})-\mathcal{D}_f(\bm{z}^{(k)},\bm{x}^{(k)}).
	\end{equation*}
	This proves \eqref{eq:bregman_identity}. Finally, $\|\bm{A}(\bm{z}-\bm{x})\|_2^2\le L_f\|\bm{z}-\bm{x}\|_2^2$ gives \eqref{eq:al_geometric_bound1}.
\end{proof}

%\subsection{Energy Dissipation and Global Convergence}

\begin{lemma}[Sufficient descent]
	\label{lem:sufficient_descent}
	Let $\{\bm{w}^{(k)}\}$ be generated by Algorithm~\ref{alg:admm_regularized}. If $\rho>\sqrt{2}L_f$, then, for every $k\ge1$, one has
	\begin{equation}
		\label{eq:descent_ineq}
		\mathcal{L}_\rho^{k+1} \le \mathcal{L}_\rho^k - \left(\frac{\rho}{2} - \frac{L_f^2}{\rho}\right) \|\bm{x}^{(k+1)}-\bm{x}^{(k)}\|_2^2 .
	\end{equation}
\end{lemma}

\begin{proof}
	By the optimality of $\bm{z}^{(k+1)}$ in the $\bm{z}$-subproblem, we have
	\begin{equation*}
		\mathcal{L}_\rho(\bm{x}^{(k)},\bm{z}^{(k+1)},\bm{u}^{(k)}) \le \mathcal{L}_\rho(\bm{x}^{(k)},\bm{z}^{(k)},\bm{u}^{(k)}).
	\end{equation*}
	For fixed $(\bm{z}^{(k+1)},\bm{u}^{(k)})$, the $\bm{x}$-subproblem has Hessian $\bm{A}^\top\bm{A}+\rho\bm{I}\succeq \rho\bm{I}$. Therefore,
	\begin{equation*}
		\mathcal{L}_\rho(\bm{x}^{(k+1)},\bm{z}^{(k+1)},\bm{u}^{(k)}) \le \mathcal{L}_\rho(\bm{x}^{(k)},\bm{z}^{(k+1)},\bm{u}^{(k)}) - \frac{\rho}{2} \|\bm{x}^{(k+1)}-\bm{x}^{(k)}\|_2^2 .
	\end{equation*}
	The dual update gives $\bm{u}^{(k+1)}-\bm{u}^{(k)} = \bm{x}^{(k+1)}-\bm{z}^{(k+1)}$, and hence
	\begin{equation*}
		\begin{aligned}
			&\mathcal{L}_\rho(\bm{x}^{(k+1)},\bm{z}^{(k+1)},\bm{u}^{(k+1)}) - \mathcal{L}_\rho(\bm{x}^{(k+1)},\bm{z}^{(k+1)},\bm{u}^{(k)}) \\
			&\qquad = \rho \langle \bm{u}^{(k+1)}-\bm{u}^{(k)}, \bm{x}^{(k+1)}-\bm{z}^{(k+1)} \rangle = \rho \|\bm{u}^{(k+1)}-\bm{u}^{(k)}\|_2^2 .
		\end{aligned}
	\end{equation*}
	The first-order optimality condition of the $\bm{x}$-subproblem is $\nabla f(\bm{x}^{(k+1)}) + \rho (\bm{x}^{(k+1)}-\bm{z}^{(k+1)}+\bm{u}^{(k)}) = \bm{0}$. Using the dual update, this becomes $\rho\bm{u}^{(k+1)} = -\nabla f(\bm{x}^{(k+1)})$. For $k\ge1$, the same identity holds at index $k$. It follows from the $L_f$-Lipschitz continuous of $\nabla f$ that
	\begin{equation}
		\label{eq:dual_bound}
		\rho \|\bm{u}^{(k+1)}-\bm{u}^{(k)}\|_2 = \|\nabla f(\bm{x}^{(k+1)})-\nabla f(\bm{x}^{(k)})\|_2 \le L_f \|\bm{x}^{(k+1)}-\bm{x}^{(k)}\|_2 .
	\end{equation}
	Combining the preceding estimates gives $\mathcal{L}_\rho^{k+1} \le \mathcal{L}_\rho^k - \frac{\rho}{2} \|\bm{x}^{(k+1)}-\bm{x}^{(k)}\|_2^2 + \frac{L_f^2}{\rho} \|\bm{x}^{(k+1)}-\bm{x}^{(k)}\|_2^2$, which proves \eqref{eq:descent_ineq}.
\end{proof}

\begin{lemma}[Boundedness of the ADMM sequence]
	\label{lem:bregman_boundedness}
	Assume that the hypotheses of Theorem~\ref{thm:existence-regularized-wdsn} hold and that $\rho>\sqrt{2}L_f$. Then $\{\bm{w}^{(k)}\}$ is bounded.
\end{lemma}
\begin{proof}
	By Theorem~\ref{thm:existence-regularized-wdsn}, $F_\lambda=f+\lambda\mathcal{R}_\eta$ is coercive.
	By Lemma~\ref{lem:sufficient_descent}, $\{\mathcal{L}_\rho^k\}_{k\ge1}$ is nonincreasing. Since $0\le\eta\le1$, one has $\mathcal{R}_\eta(\bm{x})\ge0$ for every $\bm{x}$, and hence $F_\lambda\ge0$. Combining this with \eqref{eq:al_geometric_bound1} and $\rho>\sqrt{2}L_f>L_f$ gives
	\begin{equation*}
		0\le F_\lambda(\bm{z}^{(k)})+\frac{\rho-L_f}{2}\|\bm{x}^{(k)}-\bm{z}^{(k)}\|_2^2
		\le \mathcal{L}_\rho^k
		\le \mathcal{L}_\rho^1,\qquad k\ge1 .
	\end{equation*}
	Therefore $\{F_\lambda(\bm{z}^{(k)})\}$ is bounded above, and the coercivity of $F_\lambda$ implies that $\{\bm{z}^{(k)}\}$ is bounded. The same estimate bounds $\{\bm{x}^{(k)}-\bm{z}^{(k)}\}$, so $\{\bm{x}^{(k)}\}$ is bounded. Finally, $\rho\bm{u}^{(k)}=-\nabla f(\bm{x}^{(k)})$ and the affine form of $\nabla f$ imply that $\{\bm{u}^{(k)}\}_{k\ge1}$ is bounded. Adding the finite initial term, $\{\bm{w}^{(k)}\}$ is bounded.
\end{proof}

\begin{theorem}[Subsequential convergence]
	\label{thm:subsequential_convergence_admm}
	Suppose that the sequence $\{\bm{w}^{(k)}\}$ is generated by Algorithm~\ref{alg:admm_regularized}. Assume that the hypotheses of Theorem~\ref{thm:existence-regularized-wdsn} hold and that $\rho>\sqrt{2}L_f$. Then the following assertions hold:
	\begin{enumerate}[(i)]
		\item The sequence $\{\bm{w}^{(k)}\}$ has at least one accumulation point.
		
		\item The iterates satisfy $\|\bm{x}^{(k+1)}-\bm{x}^{(k)}\|_2\to0$, $\|\bm{z}^{(k+1)}-\bm{z}^{(k)}\|_2\to0$, $\|\bm{u}^{(k+1)}-\bm{u}^{(k)}\|_2\to0$, and $\|\bm{x}^{(k)}-\bm{z}^{(k)}\|_2\to0$.
		
		\item For every accumulation point $(\bar{\bm{x}},\bar{\bm{z}},\bar{\bm{u}})$, we have $\bar{\bm{x}}=\bar{\bm{z}}$, $\nabla f(\bar{\bm{x}})+\rho\bar{\bm{u}}=\bm{0}$, and $\bm{0}\in\partial g(\bar{\bm{z}})-\rho\bar{\bm{u}}$. Consequently, $\bar{\bm{x}}=\bar{\bm{z}}$ is a stationary point of \eqref{eq:regularized-wdsn-model}, namely
		\begin{equation*}
			\bm{0} \in \bm{A}^\top(\bm{A}\bar{\bm{x}}-\bm{b}) +2\lambda\|\bar{\bm{x}}\|_1\operatorname{SGN}(\bar{\bm{x}}) -2\lambda\eta\bar{\bm{x}} .
		\end{equation*}
	\end{enumerate}
\end{theorem}

\begin{proof}
	(i) By Lemma~\ref{lem:bregman_boundedness}, the sequence $\{\bm{w}^{(k)}\}$ is bounded. Since the ambient space is finite-dimensional, the Bolzano--Weierstrass theorem gives at least one accumulation point.
	
	(ii) By Lemma~\ref{lem:sufficient_descent}, for every $k\ge1$, one has
	\begin{equation*}
		\mathcal{L}_\rho^{k+1}
		\le
		\mathcal{L}_\rho^k
		-\left(\frac{\rho}{2}-\frac{L_f^2}{\rho}\right)
		\|\bm{x}^{(k+1)}-\bm{x}^{(k)}\|_2^2 .
	\end{equation*}
	The coefficient is positive because $\rho>\sqrt{2}L_f$. Moreover, \eqref{eq:al_geometric_bound1}, $F_\lambda\ge0$, and $\rho>L_f$ imply that $\mathcal{L}_\rho^k$ is bounded from below for $k\ge1$. Summing the descent inequality over $k$ therefore gives
$
		\sum_{k=1}^{\infty}
		\|\bm{x}^{(k+1)}-\bm{x}^{(k)}\|_2^2
		<+\infty .
$
	Hence $\|\bm{x}^{(k+1)}-\bm{x}^{(k)}\|_2\to0$. The dual estimate \eqref{eq:dual_bound} gives $\|\bm{u}^{(k+1)}-\bm{u}^{(k)}\|_2\to0$. The dual update gives $\bm{x}^{(k+1)}-\bm{z}^{(k+1)}=\bm{u}^{(k+1)}-\bm{u}^{(k)}$, and hence $\|\bm{x}^{(k)}-\bm{z}^{(k)}\|_2\to0$. Finally, for $k\ge1$, we have
	\begin{equation*}
		\bm{z}^{(k+1)}-\bm{z}^{(k)}
		=
		\bm{x}^{(k+1)}-\bm{x}^{(k)}
		-
		\left[
		(\bm{u}^{(k+1)}-\bm{u}^{(k)})
		-
		(\bm{u}^{(k)}-\bm{u}^{(k-1)})
		\right].
	\end{equation*}
	Therefore $\|\bm{z}^{(k+1)}-\bm{z}^{(k)}\|_2\to0$.
	
	(iii) Let $(\bm{x}^{(k_j)},\bm{z}^{(k_j)},\bm{u}^{(k_j)})\to(\bar{\bm{x}},\bar{\bm{z}},\bar{\bm{u}})$. By (ii), $\bar{\bm{x}}=\bar{\bm{z}}$. Passing to the limit in \eqref{eq:dual_mapping} along this subsequence yields $\nabla f(\bar{\bm{x}})+\rho\bar{\bm{u}}=\bm{0}$. The optimality condition of the $\bm{z}$-subproblem is
	\begin{equation*}
		\bm{0}\in\partial g(\bm{z}^{(k+1)})-\rho(\bm{x}^{(k)}-\bm{z}^{(k+1)}+\bm{u}^{(k)}).
	\end{equation*}
	Using (ii), we have $\bm{z}^{(k_j+1)}\to\bar{\bm{z}}$. Since $g$ is continuous and the graph of the limiting subdifferential is closed, passing to the limit along the same subsequence gives
	\begin{equation*}
		\bm{0}\in\partial g(\bar{\bm{z}})-\rho(\bar{\bm{x}}-\bar{\bm{z}}+\bar{\bm{u}}) =\partial g(\bar{\bm{z}})-\rho\bar{\bm{u}}.
	\end{equation*}
	Combining the two limiting relations gives $\bm{0}\in\nabla f(\bar{\bm{x}})+\partial g(\bar{\bm{x}})$. Since $\nabla f(\bar{\bm{x}})=\bm{A}^\top(\bm{A}\bar{\bm{x}}-\bm{b})$ and $\partial g(\bar{\bm{x}})=2\lambda\|\bar{\bm{x}}\|_1\operatorname{SGN}(\bar{\bm{x}})-2\lambda\eta\bar{\bm{x}}$, the stated stationarity condition follows.
\end{proof}

\begin{lemma}[Subgradient bound]
	\label{lem:subgrad_bound}
	Let $\{\bm{w}^{(k)}\}$ be generated by Algorithm~\ref{alg:admm_regularized}. Then, for every $k\ge1$, one has
$
		\operatorname{dist}\bigl(\bm{0},\partial\mathcal{L}_\rho(\bm{w}^{(k+1)})\bigr) \le C_\rho\|\bm{x}^{(k+1)}-\bm{x}^{(k)}\|_2 ,
$
	where one may take $C_\rho:=\sqrt{2L_f^2+(\rho+L_f)^2}$. If $\rho\bm{u}^{(0)}=-\nabla f(\bm{x}^{(0)})$, then the same estimate also holds for $k=0$.
\end{lemma}

\begin{proof}
	The first-order optimality condition of the $\bm{x}$-subproblem is
	\begin{equation*}
		\nabla f(\bm{x}^{(k+1)}) +\rho(\bm{x}^{(k+1)}-\bm{z}^{(k+1)}+\bm{u}^{(k)}) = \bm{0} .
	\end{equation*}
	Together with the dual update $\bm{u}^{(k+1)}=\bm{u}^{(k)}+\bm{x}^{(k+1)}-\bm{z}^{(k+1)}$, this gives $\rho\bm{u}^{(k+1)}=-\nabla f(\bm{x}^{(k+1)})$. Hence, for $k\ge1$, one has
	\begin{equation*}
		\rho\|\bm{u}^{(k+1)}-\bm{u}^{(k)}\|_2 = \|\nabla f(\bm{x}^{(k+1)})-\nabla f(\bm{x}^{(k)})\|_2 \le L_f\|\bm{x}^{(k+1)}-\bm{x}^{(k)}\|_2 .
	\end{equation*}
	
	We now construct an element of $\partial\mathcal{L}_\rho(\bm{w}^{(k+1)})$. The $\bm{x}$- and $\bm{u}$-components satisfy
	\begin{equation*}
		\bm{d}_x^{(k+1)} := \nabla_{\bm{x}}\mathcal{L}_\rho(\bm{w}^{(k+1)}) = \rho(\bm{u}^{(k+1)}-\bm{u}^{(k)}),
	\end{equation*}
	and
$
		\bm{d}_u^{(k+1)} := \nabla_{\bm{u}}\mathcal{L}_\rho(\bm{w}^{(k+1)}) = \rho(\bm{x}^{(k+1)}-\bm{z}^{(k+1)}) = \rho(\bm{u}^{(k+1)}-\bm{u}^{(k)}).
$
	For the $\bm{z}$-component, the optimality condition of the $\bm{z}$-subproblem gives an element $\bm{v}_z^{(k+1)}\in\partial g(\bm{z}^{(k+1)})$ such that $\bm{v}_z^{(k+1)} + \rho(\bm{z}^{(k+1)}-\bm{x}^{(k)}-\bm{u}^{(k)}) = \bm{0}$. Therefore,
	\begin{equation*}
		\begin{aligned}
			\bm{d}_z^{(k+1)} &:= \bm{v}_z^{(k+1)} -\rho(\bm{x}^{(k+1)}-\bm{z}^{(k+1)}+\bm{u}^{(k+1)}) \\
			&= \rho(\bm{x}^{(k)}-\bm{x}^{(k+1)}+\bm{u}^{(k)}-\bm{u}^{(k+1)}) \in \partial_{\bm{z}}\mathcal{L}_\rho(\bm{w}^{(k+1)}).
		\end{aligned}
	\end{equation*}
	Thus $\bm{d}^{(k+1)}:=(\bm{d}_x^{(k+1)},\bm{d}_z^{(k+1)},\bm{d}_u^{(k+1)}) \in\partial\mathcal{L}_\rho(\bm{w}^{(k+1)})$. Using the preceding dual bound,
	\begin{equation*}
		\|\bm{d}_x^{(k+1)}\|_2\le L_f\|\bm{x}^{(k+1)}-\bm{x}^{(k)}\|_2,\qquad \|\bm{d}_u^{(k+1)}\|_2\le L_f\|\bm{x}^{(k+1)}-\bm{x}^{(k)}\|_2,
	\end{equation*}
	and
	\begin{equation*}
		\|\bm{d}_z^{(k+1)}\|_2 \le \rho\|\bm{x}^{(k+1)}-\bm{x}^{(k)}\|_2 + \rho\|\bm{u}^{(k+1)}-\bm{u}^{(k)}\|_2 \le (\rho+L_f)\|\bm{x}^{(k+1)}-\bm{x}^{(k)}\|_2 .
	\end{equation*}
	Consequently,
$
		\operatorname{dist}\bigl(\bm{0},\partial\mathcal{L}_\rho(\bm{w}^{(k+1)})\bigr) \le \|\bm{d}^{(k+1)}\|_2 \le \sqrt{2L_f^2+(\rho+L_f)^2}\, \|\bm{x}^{(k+1)}-\bm{x}^{(k)}\|_2 .
$
	This proves the claim.
\end{proof}

With the sufficient descent and subgradient bounds established, we are now ready to state the global convergence theorem.

\begin{theorem}[Global Convergence]
	\label{thm:global_convergence}
	Suppose that the sequence $\{\bm{w}^{(k)}\}$ is generated by Algorithm \ref{alg:admm_regularized}, with each $\bm{z}$-subproblem solved globally. Assume that the hypotheses of Theorem~\ref{thm:existence-regularized-wdsn} hold and that $\rho > \sqrt{2}L_f$. Then the sequence has finite length
	\begin{equation*}
		\sum_{k=0}^{\infty} \left\|(\bm{x}^{(k+1)}-\bm{x}^{(k)},\bm{z}^{(k+1)}-\bm{z}^{(k)},\bm{u}^{(k+1)}-\bm{u}^{(k)})\right\|_2<+\infty .
	\end{equation*}
	Consequently, the whole sequence converges to a stationary point of the split problem \eqref{eq:regularized-split}. In particular, the common primal limit is a stationary point of \eqref{eq:regularized-wdsn-model}.
\end{theorem}

\begin{proof}
	Set $a:=\frac{\rho}{2}-\frac{L_f^2}{\rho}>0$ and $s_k:=\|\bm{x}^{(k+1)}-\bm{x}^{(k)}\|_2$. By Lemma~\ref{lem:sufficient_descent}, for every $k\ge1$, one has
	\begin{equation}
		\label{eq:global-kl-descent}
		\mathcal{L}_\rho^k-\mathcal{L}_\rho^{k+1} \ge a s_k^2 .
	\end{equation}
	On the other hand, \eqref{eq:al_geometric_bound1}, $F_\lambda\ge0$, and $\rho>L_f$ show that $\{\mathcal{L}_\rho^k\}_{k\ge1}$ is bounded from below. Thus $\mathcal{L}_\rho^k$ converges to some finite value $\mathcal{L}_\rho^\infty$, and \eqref{eq:global-kl-descent} implies $\sum_{k=1}^{\infty}s_k^2<+\infty$. Hence $s_k^2\to0$, and therefore
	$s_k\to0$. Lemma~\ref{lem:subgrad_bound} gives the relative-error estimate
	\begin{equation}
		\label{eq:global-kl-relative-error}
		\operatorname{dist}\bigl(\bm{0},\partial\mathcal{L}_\rho(\bm{w}^{(k+1)})\bigr) \le C_\rho s_k, \qquad k\ge1 .
	\end{equation}
	The augmented Lagrangian function $\mathcal{L}_\rho$ is continuous and semi-algebraic, and therefore satisfies the Kurdyka--\L{}ojasiewicz property. Moreover, Lemma~\ref{lem:bregman_boundedness} gives the boundedness of $\{\bm{w}^{(k)}\}$. Let $\Omega$ denote its set of accumulation points. Then $\Omega$ is nonempty and compact, $\operatorname{dist}(\bm{w}^{(k)},\Omega)\to0$, and, by continuity of $\mathcal{L}_\rho$, $\mathcal{L}_\rho(\bar{\bm{w}}) = \mathcal{L}_\rho^\infty$ for $\bar{\bm{w}}\in\Omega$. The uniformized KL property on $\Omega$ yields constants $\tau>0$, $\nu>0$, and a concave function $\varphi\in C^1(0,\nu)$ with $\varphi'>0$, such that
	\begin{equation}
		\label{eq:global-uniform-kl}
		\varphi'\bigl(\mathcal{L}_\rho(\bm{w})-\mathcal{L}_\rho^\infty\bigr) \operatorname{dist}\bigl(\bm{0},\partial\mathcal{L}_\rho(\bm{w})\bigr) \ge1
	\end{equation}
	whenever $\operatorname{dist}(\bm{w},\Omega)<\tau$ and $0<\mathcal{L}_\rho(\bm{w})-\mathcal{L}_\rho^\infty<\nu$.

	If $\mathcal{L}_\rho^K=\mathcal{L}_\rho^\infty$ for some $K\ge1$, then the monotonicity of $\{\mathcal{L}_\rho^k\}$ implies $\mathcal{L}_\rho^k=\mathcal{L}_\rho^\infty$ for $k\ge K$. Consequently, by \eqref{eq:global-kl-descent},
	\begin{equation*}
		0 \le a s_k^2 \le \mathcal{L}_\rho^k-\mathcal{L}_\rho^{k+1} =0, \qquad k\ge K ,
	\end{equation*}
	and hence $\bm{x}^{(k+1)}=\bm{x}^{(k)}$ for $k\ge K$. Using \eqref{eq:dual_bound}, we further obtain $\bm{u}^{(k+1)}=\bm{u}^{(k)}$ for $k\ge K$. The dual update then gives $\bm{x}^{(k+1)}-\bm{z}^{(k+1)} = \bm{u}^{(k+1)}-\bm{u}^{(k)} = \bm{0}$ for $k\ge K$. Thus, with $K_0:=K+1$, $\bm{w}^{(k)}=\bm{w}^{(K_0)}$ for $k\ge K_0$. It follows that $\sum_{k=K_0}^{\infty} \|\bm{w}^{(k+1)}-\bm{w}^{(k)}\|_2 =0$, and the finite-length property follows. Moreover, \eqref{eq:global-kl-relative-error} yields $\operatorname{dist}\bigl(\bm{0},\partial\mathcal{L}_\rho(\bm{w}^{(k+1)})\bigr) \le C_\rho s_k=0$ for $k\ge K$, so $\bm{w}^{(K_0)}$ is a stationary point of the split problem.
	
	Assume now that, for some $N_0\ge1$, $\mathcal{L}_\rho^k>\mathcal{L}_\rho^\infty$ for $k\ge N_0$. Choose $N\ge N_0$ such that $\operatorname{dist}(\bm{w}^{(k+1)},\Omega)<\tau$, $0<\mathcal{L}_\rho^{k+1}-\mathcal{L}_\rho^\infty<\nu$ for $k\ge N$. Then \eqref{eq:global-uniform-kl} applies to $\bm{w}^{(k+1)}$, and together with \eqref{eq:global-kl-relative-error} gives
	\begin{equation}
		\label{eq:kl-relative-error-x}
		1 \le C_\rho \varphi'\bigl(\mathcal{L}_\rho^{k+1}-\mathcal{L}_\rho^\infty\bigr) s_k, \qquad k\ge N .
	\end{equation}
	In particular, $s_k>0$ for all $k\ge N$. Define $\Theta_k := \varphi\bigl(\mathcal{L}_\rho^{k+1}-\mathcal{L}_\rho^\infty\bigr) - \varphi\bigl(\mathcal{L}_\rho^{k+2}-\mathcal{L}_\rho^\infty\bigr)$ for $k\ge N$. By the concavity of $\varphi$, \eqref{eq:global-kl-descent}, and \eqref{eq:kl-relative-error-x}, one has
	\begin{equation*}
		\Theta_k \ge \varphi'\bigl(\mathcal{L}_\rho^{k+1}-\mathcal{L}_\rho^\infty\bigr) \bigl(\mathcal{L}_\rho^{k+1}-\mathcal{L}_\rho^{k+2}\bigr) \ge a \varphi'\bigl(\mathcal{L}_\rho^{k+1}-\mathcal{L}_\rho^\infty\bigr) s_{k+1}^2 \ge \frac{a s_{k+1}^2}{C_\rho s_k}, \qquad k\ge N .
	\end{equation*}
	Hence $s_{k+1}^2 \le \frac{C_\rho}{a}s_k\Theta_k$ for $k\ge N$. Using $2\sqrt{\alpha\beta}\le \alpha+\beta$, we obtain
	\begin{equation}
		\label{eq:kl-one-step-length}
		s_{k+1} \le \frac12s_k + \frac{C_\rho}{2a}\Theta_k, \qquad k\ge N .
	\end{equation}
	Summing \eqref{eq:kl-one-step-length} from $k=N$ to $M$ gives
	\begin{equation*}
		\sum_{j=N+1}^{M+1}s_j \le \frac12s_N + \frac12\sum_{j=N+1}^{M}s_j + \frac{C_\rho}{2a} \sum_{k=N}^{M}\Theta_k .
	\end{equation*}
	Therefore, $\sum_{j=N+1}^{M+1}s_j \le s_N + \frac{C_\rho}{a} \sum_{k=N}^{M}\Theta_k$. By the definition of $\Theta_k$, the sum can be written as
	\begin{equation*}
		\sum_{k=N}^{M}\Theta_k = \varphi\bigl(\mathcal{L}_\rho^{N+1}-\mathcal{L}_\rho^\infty\bigr) - \varphi\bigl(\mathcal{L}_\rho^{M+2}-\mathcal{L}_\rho^\infty\bigr) \le \varphi\bigl(\mathcal{L}_\rho^{N+1}-\mathcal{L}_\rho^\infty\bigr).
	\end{equation*}
	Letting $M\to\infty$ yields
	\begin{equation}
		\label{eq:x-finite-length}
		\sum_{k=0}^{\infty} \|\bm{x}^{(k+1)}-\bm{x}^{(k)}\|_2 <+\infty .
	\end{equation}
	
	It follows from \eqref{eq:dual_bound} and \eqref{eq:x-finite-length} that
	\begin{equation*}
		\sum_{k=1}^{\infty} \|\bm{u}^{(k+1)}-\bm{u}^{(k)}\|_2 \le \frac{L_f}{\rho} \sum_{k=1}^{\infty} \|\bm{x}^{(k+1)}-\bm{x}^{(k)}\|_2 <+\infty .
	\end{equation*}
	Moreover, for $k\ge1$, $\bm{z}^{(k+1)}-\bm{z}^{(k)} = \bm{x}^{(k+1)}-\bm{x}^{(k)} - (\bm{u}^{(k+1)}-\bm{u}^{(k)}) + (\bm{u}^{(k)}-\bm{u}^{(k-1)})$. Thus
	\begin{equation*}
		\begin{aligned}
			\sum_{k=1}^{\infty} \|\bm{z}^{(k+1)}-\bm{z}^{(k)}\|_2 &\le \sum_{k=1}^{\infty} \|\bm{x}^{(k+1)}-\bm{x}^{(k)}\|_2 \\
			&\quad+ \sum_{k=1}^{\infty} \|\bm{u}^{(k+1)}-\bm{u}^{(k)}\|_2 + \sum_{k=1}^{\infty} \|\bm{u}^{(k)}-\bm{u}^{(k-1)}\|_2 <+\infty .
		\end{aligned}
	\end{equation*}
	Consequently, $\sum_{k=0}^{\infty} \|\bm{w}^{(k+1)}-\bm{w}^{(k)}\|_2 <+\infty$. Hence $\{\bm{w}^{(k)}\}$ is a Cauchy sequence and hence converges. Denote its limit by
	$\bar{\bm{w}} := (\bar{\bm{x}},\bar{\bm{z}},\bar{\bm{u}})$. It remains to show that this limit is stationary.
	
	By Lemma~\ref{lem:subgrad_bound}, for every $k\ge1$ there exists $\bm{q}^{(k+1)} \in \partial\mathcal{L}_\rho(\bm{w}^{(k+1)})$ such that $\|\bm{q}^{(k+1)}\|_2 \le C_\rho s_k$. Since $s_k\to0$, $\bm{q}^{(k+1)}\to\bm{0}$. Together with $\bm{w}^{(k+1)}\to\bar{\bm{w}}$, $\mathcal{L}_\rho(\bm{w}^{(k+1)}) \to \mathcal{L}_\rho(\bar{\bm{w}})$, the closedness of the limiting-subdifferential graph gives $\bm{0} \in \partial\mathcal{L}_\rho(\bar{\bm{w}})$. Equivalently, $\nabla f(\bar{\bm{x}}) + \rho(\bar{\bm{x}}-\bar{\bm{z}}+\bar{\bm{u}}) = \bm{0}$, $\bm{0} \in \partial g(\bar{\bm{z}}) - \rho(\bar{\bm{x}}-\bar{\bm{z}}+\bar{\bm{u}})$, and $\bar{\bm{x}}-\bar{\bm{z}} = \bm{0}$. Therefore, $\bar{\bm{x}}=\bar{\bm{z}}$, $\nabla f(\bar{\bm{x}})+\rho\bar{\bm{u}}=\bm{0}$, $\rho\bar{\bm{u}}\in\partial g(\bar{\bm{x}})$. It follows that $\bm{0} \in \nabla f(\bar{\bm{x}})+\partial g(\bar{\bm{x}})$, which is the stationarity condition for \eqref{eq:regularized-wdsn-model}. Hence $\bar{\bm{w}}$ is a stationary point of \eqref{eq:regularized-split}, and $\bar{\bm{x}}$ is a stationary point of \eqref{eq:regularized-wdsn-model}.
\end{proof}

\section{Numerical Experiments}\label{sec-num}

This section evaluates the efficiency of the ADMM algorithm based on the closed-form proximal mapping derived in Section~\ref{sec:prox-l1sq-l2sq}, and compares it with the HV algorithm. The experiments cover both noiseless and noisy compressed sensing problems, with the aim of highlighting the computational advantage brought by exactly solving the full WDSN proximal subproblem.

\subsection{Noiseless Compressed Sensing}
\label{subsec:noiseless-experiments}

We first evaluate the two solvers in noiseless compressed sensing settings. For the problem sizes, we set $M=500,600,\ldots,1000$, $N=10M$, and $s=M/10$, and perform 20 random Monte Carlo trials. The two solvers use the same $\ell_1$ warm start, the same parameters $\eta=1$ and $\lambda=10^{-6}$, the same maximum number of iterations $5N$, and the same stopping tolerance $10^{-8}$; the penalty parameter in ADMM is set to $\rho=10$. We test two types of sensing matrices: oversampled DCT matrices with oversampling factor $F=20$, and Gaussian matrices with correlation parameter $r=0.2$. Figures~\ref{fig:runtime-scaling-summary} and~\ref{fig:runtime-scaling-gaussian-summary} report, respectively, the average number of iterations, the average CPU running time, the CPU running-time scatter plot, and the CPU running time speedup of the two solvers.
\begin{figure}[htbp]
	\centering
	\begin{subfigure}[t]{0.4\textwidth}
		\centering
		\vspace{0pt}
		\includegraphics[width=0.9\linewidth,keepaspectratio]
		{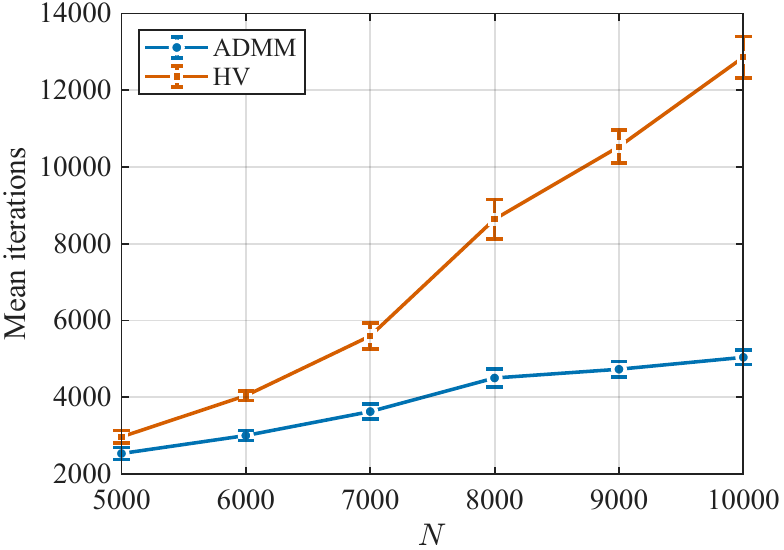}
		\caption{Mean iterations.}
		\label{fig:runtime-scaling-iterations}
	\end{subfigure}
%	\hfill
	\begin{subfigure}[t]{0.4\textwidth}
		\centering
		\vspace{0pt}
		\includegraphics[width=0.9\linewidth,keepaspectratio]{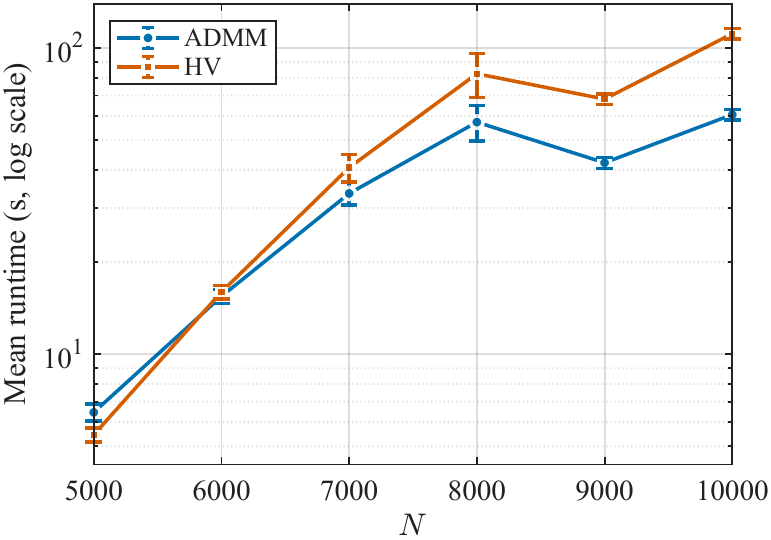}
		\caption{Mean CPU runtime.}
		\label{fig:runtime-scaling-runtime}
	\end{subfigure}

%	\vspace{0.6em}

	\begin{subfigure}[t]{0.4\textwidth}
		\centering
		\vspace{0pt}
		\includegraphics[width=0.9\linewidth,keepaspectratio]{\detokenize{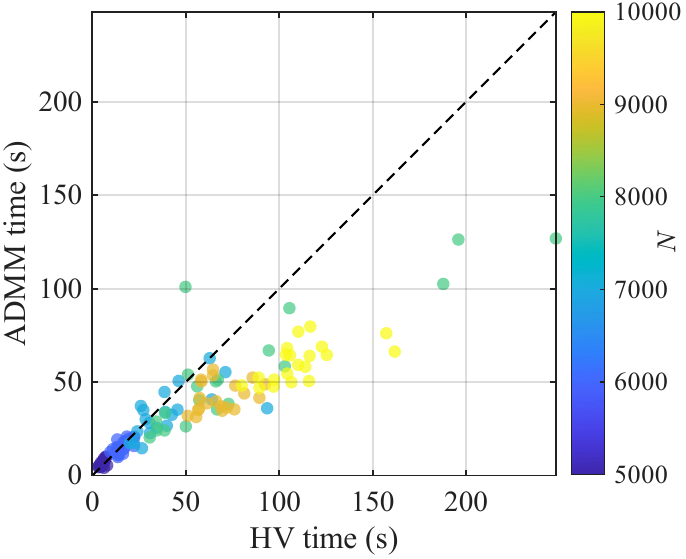}}
		\caption{CPU runtime.}
		\label{fig:runtime-scaling-scatter}
	\end{subfigure}
%	\hfill
	\begin{subfigure}[t]{0.4\textwidth}
		\centering
		\vspace{0pt}
		\includegraphics[width=0.9\linewidth,keepaspectratio]{\detokenize{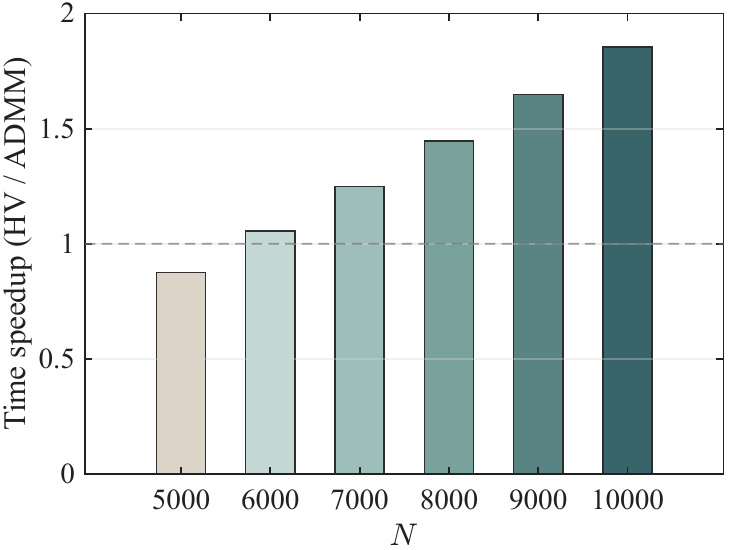}}
		\caption{Speedup.}
		\label{fig:runtime-scaling-speedup}
	\end{subfigure}
	\caption{Runtime and iteration comparison between HV and ADMM for oversampled DCT sensing matrices.}
	\label{fig:runtime-scaling-summary}
\end{figure}

The oversample DCT experiment in Figure~\ref{fig:runtime-scaling-summary} shows a crossover pattern rather than a uniform small-scale advantage. For $N=5000$, the HV solver has lower mean CPU time, although ADMM already requires fewer iterations on average. This indicates that the per-iteration linear-algebra cost of ADMM is not fully amortized at the smallest tested dimension. Starting from $N=6000$, however, ADMM becomes faster, and the CPU running time speedup of the two solvers generally grows with the ambient dimension. At $N=10000$, ADMM reduces the mean iteration count from $12860.6$ to $5039.6$, corresponding to an iteration reduction of about $60.8\%$, and reduces the mean runtime from $111.4$ seconds to $60.6$ seconds. The speedup at this largest dimension is $1.85$. Thus, for oversampled DCT matrices, the exact WDSN proximal module offers a computational advantage once the reduction in outer iterations outweighs the additional cost incurred by solving the ADMM quadratic subproblems.

\begin{figure}[htbp]
	\centering
	\begin{subfigure}[t]{0.4\textwidth}
		\centering
		\vspace{0pt}
		\includegraphics[width=0.9\linewidth,keepaspectratio]{\detokenize{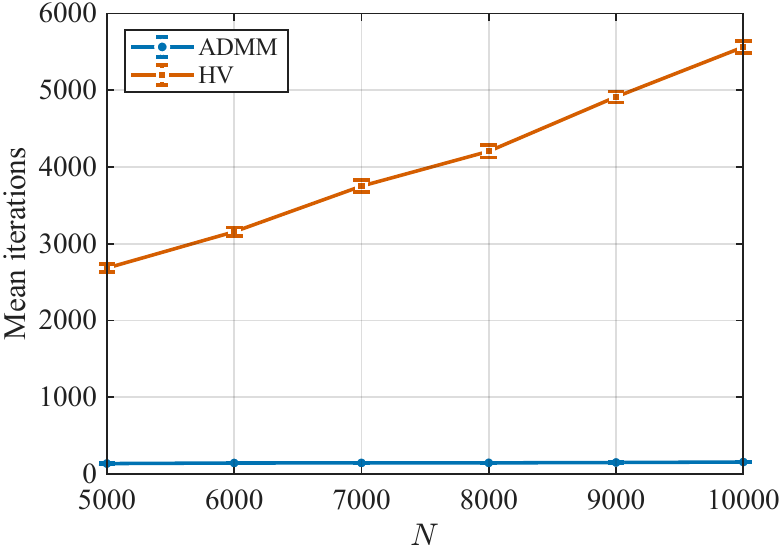}}
		\caption{Mean iterations.}
		\label{fig:runtime-scaling-gaussian-iterations}
	\end{subfigure}
	%	\hfill
	\begin{subfigure}[t]{0.4\textwidth}
		\centering
		\vspace{0pt}
		\includegraphics[width=0.9\linewidth,keepaspectratio]{\detokenize{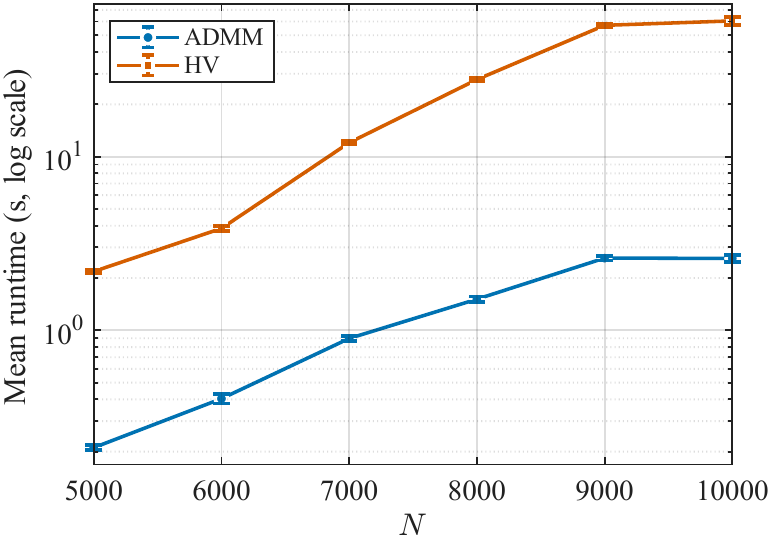}}
		\caption{Mean CPU runtime.}
		\label{fig:runtime-scaling-gaussian-runtime}
	\end{subfigure}
	
	%	\vspace{0.6em}
	
	\begin{subfigure}[t]{0.4\textwidth}
		\centering
		\vspace{0pt}
		\includegraphics[width=0.9\linewidth,keepaspectratio]{\detokenize{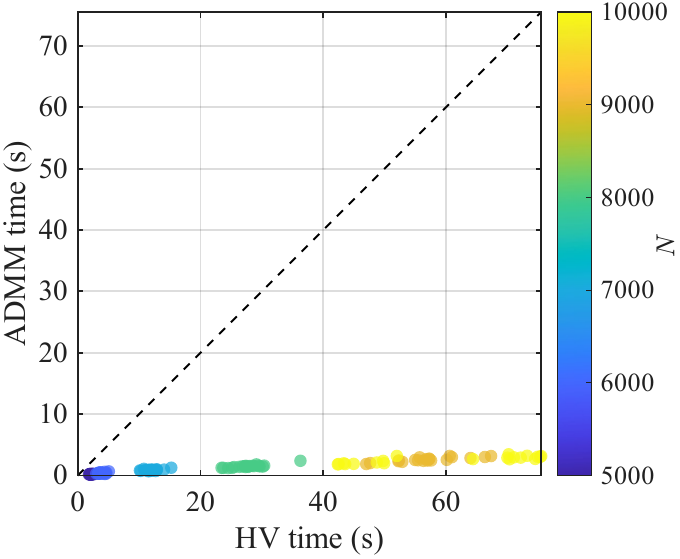}}
		\caption{CPU runtime.}
		\label{fig:runtime-scaling-gaussian-scatter}
	\end{subfigure}
	%	\hfill
	\begin{subfigure}[t]{0.4\textwidth}
		\centering
		\vspace{0pt}
		\includegraphics[width=0.9\linewidth,keepaspectratio]{\detokenize{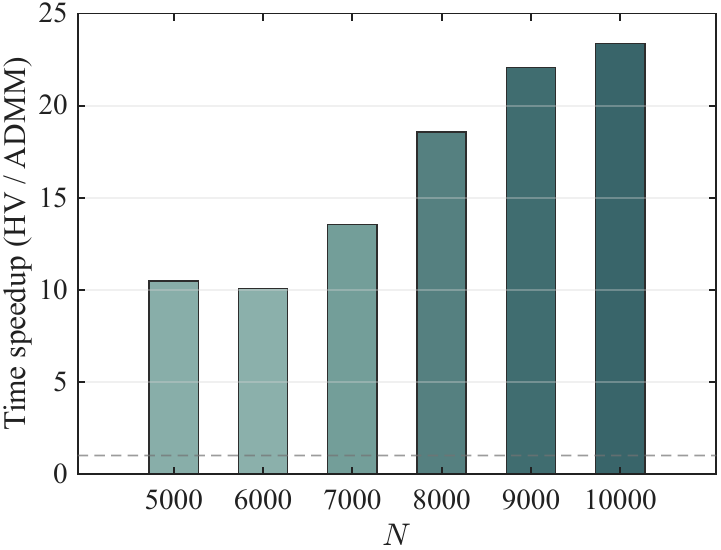}}
		\caption{Speedup.}
		\label{fig:runtime-scaling-gaussian-speedup}
	\end{subfigure}
	\caption{Runtime and iteration comparison between HV and ADMM for Gaussian sensing matrices.}
	\label{fig:runtime-scaling-gaussian-summary}
\end{figure}

The Gaussian experiment in Figure~\ref{fig:runtime-scaling-gaussian-summary} gives a substantially stronger scaling advantage for ADMM. Across all six dimensions, the mean ADMM iteration count remains between $135.3$ and $155.5$, whereas the mean HV iteration count increases from $2683.0$ to $5563.3$. This separation is reflected in CPU time: the mean ADMM time ranges from $0.21$ to $2.60$ seconds, while the mean HV time ranges from $2.18$ to $60.45$ seconds. The paired mean speedup ranges from $10.08$ to $23.36$ and reaches $23.36$ at $N=10000$, where the mean iteration reduction is about $97.2\%$. These results support the algorithmic role of the full WDSN proximal mapping: the HV update treats the negative quadratic term through the smooth gradient step and applies a proximal map only for the squared $\ell_1$ term, whereas the proposed ADMM update applies the closed-form proximal map of the full WDSN penalty in the $\bm{z}$-subproblem.

It remains to verify that the faster ADMM runs do not obtain their speed by terminating at poorer objective values. Tables~\ref{tab:admm-hv-objective-nnz} and~\ref{tab:admm-hv-objective-nnz-gaussian} therefore report the final RWDSN objective values on the same trials. The relative gap (RG) is computed as $|f_{\rm ADMM}-f_{\rm HV}|/\max\{1,|f_{\rm HV}|\}$.
\begin{table}[!htbp] 
	\centering
	\begingroup
	\scriptsize
	\setlength{\tabcolsep}{3.2pt}
	\caption{Comparison of final objective values between HV and ADMM for oversampled DCT sensing matrices. The notation $\mu\pm\sigma$ reports the mean and standard deviation; Max RG reports the worst trial.}
	\label{tab:admm-hv-objective-nnz}
	\begin{tabular}{rrrrccc}
		\toprule
		\multicolumn{3}{c}{Problem} & \multicolumn{4}{c}{Final objective value} \\
		\cmidrule(lr){1-3}\cmidrule(l){4-7}
		$N$ & $M$ & $s$ & ADMM & HV & RG ($\mu\pm\sigma$) & Max RG \\
		\midrule
		5000 & 500 & 50 & 2.686e-4 $\pm$ 6.6e-5 & 2.686e-4 $\pm$ 6.6e-5 & 2.75e-10 $\pm$ 1.08e-10 & 5.20e-10 \\
		6000 & 600 & 60 & 3.659e-4 $\pm$ 9.7e-5 & 3.659e-4 $\pm$ 9.7e-5 & 3.49e-10 $\pm$ 1.10e-10 & 5.74e-10 \\
		7000 & 700 & 70 & 4.372e-4 $\pm$ 1.3e-4 & 4.372e-4 $\pm$ 1.3e-4 & 4.74e-10 $\pm$ 1.93e-10 & 9.70e-10 \\
		8000 & 800 & 80 & 6.206e-4 $\pm$ 1.9e-4 & 6.206e-4 $\pm$ 1.9e-4 & 6.64e-10 $\pm$ 2.62e-10 & 1.21e-9 \\
		9000 & 900 & 90 & 7.071e-4 $\pm$ 1.9e-4 & 7.071e-4 $\pm$ 1.9e-4 & 6.89e-10 $\pm$ 3.48e-10 & 1.67e-9 \\
		10000 & 1000 & 100 & 8.658e-4 $\pm$ 2.1e-4 & 8.658e-4 $\pm$ 2.1e-4 & 7.36e-10 $\pm$ 2.67e-10 & 1.11e-9 \\
		\bottomrule
	\end{tabular}
	\endgroup
\end{table}

\begin{table}[!htbp]
	\centering
	\begingroup
	\scriptsize
	\setlength{\tabcolsep}{3.2pt}
	\caption{Comparison of final objective values between HV and ADMM for Gaussian sensing matrices. The notation $\mu\pm\sigma$ reports the mean and standard deviation; Max RG reports the worst trial.}
	\label{tab:admm-hv-objective-nnz-gaussian}
	\begin{tabular}{rrrrccc}
		\toprule
		\multicolumn{3}{c}{Problem} & \multicolumn{4}{c}{Final objective value} \\
		\cmidrule(lr){1-3}\cmidrule(l){4-7}
		$N$ & $M$ & $s$ & ADMM & HV & RG ($\mu\pm\sigma$) & Max RG \\
		\midrule
		5000 & 500 & 50 & 2.692e-4 $\pm$ 7.3e-5 & 2.692e-4 $\pm$ 7.3e-5 & 1.59e-9 $\pm$ 4.03e-10 & 2.43e-9 \\
		6000 & 600 & 60 & 3.165e-4 $\pm$ 1.0e-4 & 3.165e-4 $\pm$ 1.0e-4 & 2.39e-9 $\pm$ 5.32e-10 & 3.62e-9 \\
		7000 & 700 & 70 & 4.777e-4 $\pm$ 1.7e-4 & 4.777e-4 $\pm$ 1.7e-4 & 3.91e-9 $\pm$ 1.44e-9 & 9.11e-9 \\
		8000 & 800 & 80 & 5.112e-4 $\pm$ 1.6e-4 & 5.113e-4 $\pm$ 1.6e-4 & 4.76e-9 $\pm$ 1.48e-9 & 7.84e-9 \\
		9000 & 900 & 90 & 7.099e-4 $\pm$ 1.9e-4 & 7.099e-4 $\pm$ 1.9e-4 & 7.11e-9 $\pm$ 1.61e-9 & 1.18e-8 \\
		10000 & 1000 & 100 & 9.156e-4 $\pm$ 2.6e-4 & 9.156e-4 $\pm$ 2.6e-4 & 1.01e-8 $\pm$ 2.05e-9 & 1.41e-8 \\
		\bottomrule
	\end{tabular}
	\endgroup
\end{table}

In the DCT experiments, the mean RG always stays within a small range between $2.75 \times 10^{-10}$ and $7.36 \times 10^{-10}$, and the maximum gap is only $1.67 \times 10^{-9}$. Even in the Gaussian cases, where the speedup is much larger, the mean RGs are still between $1.59 \times 10^{-9}$ and $1.01 \times 10^{-8}$, with the largest gap being $1.41 \times 10^{-8}$. Therefore, in the noiseless case, ADMM and HV give very close RWDSN objective values. This shows that the advantage of the proposed ADMM is purely computational: because it solves the proximal subproblem of $\mathcal{R}_{\eta}$ exactly, it needs fewer outer iterations, especially for Gaussian sensing matrices and large oversampled DCT cases.
  
\subsection{Noisy Compressed Sensing}
\label{subsec:noise-experiments}

Next, we evaluate the robustness and computational performance of the regularized WDSN model in the presence of measurement noise. The noisy data are generated according to $\bm{b} = \bm{A}\bm{x}^* + \bm{e}$. We consider three problem sizes: $(M,N,s)=(300,3000,30)$, $(400,4000,40)$, and $(500,5000,50)$, as well as three SNR levels: $30$, $40$, and $50$ dB. For each problem size and SNR level, the reported statistical results are averaged over 20 Monte Carlo trials. In each trial, both ADMM and HV utilize the same sensing matrix and noisy observations. This experimental setup ensures that the reported differences primarily reflect the variations in the algorithm update mechanisms, rather than the impact of other varying factors.

We test two sensing matrices. In the oversampled DCT matrices, the oversampling factor is $F=10$ and the support separation is set to $2F$. In the Gaussian matrices, the correlation parameter is $r=0.5$, and the support separation is set to 20. The noise magnitude is determined by the prescribed SNR. Specifically, for $\bm{g}\sim\mathcal{N}(\bm{0},\bm{I}_M)$, we set $\sigma=\frac{\|\bm{A}\bm{x}^*\|_2}{\sqrt{M}\,10^{\mathrm{SNR}/20}}$ and $\bm{e}=\sigma\sqrt{M}\frac{\bm{g}}{\|\bm{g}\|_2}$. This construction gives $\mathrm{SNR}=20\log_{10}\frac{\|\bm{A}\bm{x}^*\|_2}{\|\bm{e}\|_2}$. 

Both solvers start from the same noisy $\ell_1$ warm start. In all tests, we set $\eta=1$, use at most $5N$ iterations, and set both the absolute and relative stopping tolerances to $10^{-8}$. For ADMM, the penalty parameter is fixed at $\rho=1$. The WDSN regularization parameter is scaled with the noise level according to $\lambda=5\times10^{-2}\sigma\sqrt{2\log N}$, and the $\ell_1$ warm-start is set to $\lambda=\sigma\sqrt{2\log N}$. To assess recovery quality, we report the reconstruction SNR
$$
	\mathrm{RSNR}(\bm{\hat{x}},\bm{x}^*)
	:=
	-10\log_{10}
	\frac{\|\bm{\hat{x}}-\bm{x}^*\|_2^2}{\|\bm{x}^*\|_2^2},
$$
where larger values indicate more accurate reconstruction. We also report top-$K$ support metrics with $K=s$: topKhit is the number of true support indices contained in the $s$ largest-magnitude entries of $\bm{\hat{x}}$, topKrec is topKhit$/s$, and topKerr is $1-\mathrm{topKrec}$. Table~\ref{tab:noisy-dct-snr-metrics} reports the results for the oversampled DCT matrices.

\begin{table}[htbp]
	\centering
	\caption{Noisy recovery results for oversampled DCT sensing matrices at three SNR levels. The notation $\mu\pm\sigma$ reports the mean and sample standard deviation.}
	\label{tab:noisy-dct-snr-metrics}
	\begingroup
	\scriptsize
	\setlength{\tabcolsep}{2pt}
	\resizebox{\textwidth}{!}{%
	\begin{tabular}{cclccccccc}
		\toprule
		SNR & $(M,N,s)$ & Solver & Iter & Obj. & Time(s) & topKhit & topKrec & topKerr & RSNR(dB) \\
		\midrule
		30 & $(300,3000,30)$ & ADMM & \ensuremath{373.5\pm272.3} & \ensuremath{1.124\mathrm{e}{-1}\pm5.477\mathrm{e}{-2}} & \ensuremath{0.50\pm0.41} & \ensuremath{27.1\pm1.0} & \ensuremath{0.905\pm0.033} & \ensuremath{0.095\pm0.033} & \ensuremath{22.40\pm1.09} \\
		30 & $(300,3000,30)$ & HV & \ensuremath{11345.6\pm3122.5} & \ensuremath{1.124\mathrm{e}{-1}\pm5.477\mathrm{e}{-2}} & \ensuremath{10.19\pm2.86} & \ensuremath{27.1\pm1.0} & \ensuremath{0.905\pm0.033} & \ensuremath{0.095\pm0.033} & \ensuremath{22.39\pm1.10} \\
		30 & $(400,4000,40)$ & ADMM & \ensuremath{319.4\pm124.2} & \ensuremath{1.393\mathrm{e}{-1}\pm5.669\mathrm{e}{-2}} & \ensuremath{0.74\pm0.28} & \ensuremath{35.4\pm1.4} & \ensuremath{0.884\pm0.035} & \ensuremath{0.116\pm0.035} & \ensuremath{20.29\pm1.72} \\
		30 & $(400,4000,40)$ & HV & \ensuremath{11564.8\pm3630.0} & \ensuremath{1.393\mathrm{e}{-1}\pm5.669\mathrm{e}{-2}} & \ensuremath{16.83\pm5.31} & \ensuremath{35.4\pm1.4} & \ensuremath{0.884\pm0.035} & \ensuremath{0.116\pm0.035} & \ensuremath{20.28\pm1.74} \\
		30 & $(500,5000,50)$ & ADMM & \ensuremath{375.6\pm340.0} & \ensuremath{1.946\mathrm{e}{-1}\pm8.067\mathrm{e}{-2}} & \ensuremath{1.39\pm1.15} & \ensuremath{43.0\pm1.4} & \ensuremath{0.861\pm0.029} & \ensuremath{0.139\pm0.029} & \ensuremath{19.87\pm1.35} \\
		30 & $(500,5000,50)$ & HV & \ensuremath{12115.8\pm5573.7} & \ensuremath{1.946\mathrm{e}{-1}\pm8.067\mathrm{e}{-2}} & \ensuremath{27.47\pm13.27} & \ensuremath{43.0\pm1.4} & \ensuremath{0.861\pm0.029} & \ensuremath{0.139\pm0.029} & \ensuremath{19.88\pm1.35} \\
		40 & $(300,3000,30)$ & ADMM & \ensuremath{463.6\pm448.2} & \ensuremath{3.640\mathrm{e}{-2}\pm1.796\mathrm{e}{-2}} & \ensuremath{0.65\pm0.62} & \ensuremath{28.9\pm0.8} & \ensuremath{0.963\pm0.026} & \ensuremath{0.037\pm0.026} & \ensuremath{31.86\pm1.19} \\
		40 & $(300,3000,30)$ & HV & \ensuremath{9809.1\pm3714.0} & \ensuremath{3.640\mathrm{e}{-2}\pm1.796\mathrm{e}{-2}} & \ensuremath{8.65\pm3.32} & \ensuremath{28.9\pm0.8} & \ensuremath{0.963\pm0.026} & \ensuremath{0.037\pm0.026} & \ensuremath{31.85\pm1.19} \\
		40 & $(400,4000,40)$ & ADMM & \ensuremath{296.4\pm126.1} & \ensuremath{4.573\mathrm{e}{-2}\pm1.888\mathrm{e}{-2}} & \ensuremath{0.72\pm0.30} & \ensuremath{38.2\pm0.9} & \ensuremath{0.955\pm0.022} & \ensuremath{0.045\pm0.022} & \ensuremath{29.68\pm1.82} \\
		40 & $(400,4000,40)$ & HV & \ensuremath{11282.0\pm4982.7} & \ensuremath{4.573\mathrm{e}{-2}\pm1.888\mathrm{e}{-2}} & \ensuremath{16.31\pm7.06} & \ensuremath{38.2\pm0.9} & \ensuremath{0.955\pm0.022} & \ensuremath{0.045\pm0.022} & \ensuremath{29.69\pm1.82} \\
		40 & $(500,5000,50)$ & ADMM & \ensuremath{331.4\pm273.3} & \ensuremath{6.446\mathrm{e}{-2}\pm2.718\mathrm{e}{-2}} & \ensuremath{1.22\pm1.02} & \ensuremath{47.2\pm1.5} & \ensuremath{0.945\pm0.030} & \ensuremath{0.055\pm0.030} & \ensuremath{29.00\pm1.42} \\
		40 & $(500,5000,50)$ & HV & \ensuremath{9564.9\pm2799.2} & \ensuremath{6.446\mathrm{e}{-2}\pm2.718\mathrm{e}{-2}} & \ensuremath{21.71\pm6.69} & \ensuremath{47.2\pm1.5} & \ensuremath{0.945\pm0.030} & \ensuremath{0.055\pm0.030} & \ensuremath{29.00\pm1.42} \\
		50 & $(300,3000,30)$ & ADMM & \ensuremath{367.4\pm411.2} & \ensuremath{1.161\mathrm{e}{-2}\pm5.757\mathrm{e}{-3}} & \ensuremath{0.48\pm0.50} & \ensuremath{29.7\pm0.5} & \ensuremath{0.990\pm0.016} & \ensuremath{0.010\pm0.016} & \ensuremath{41.71\pm1.10} \\
		50 & $(300,3000,30)$ & HV & \ensuremath{8074.2\pm3387.8} & \ensuremath{1.161\mathrm{e}{-2}\pm5.757\mathrm{e}{-3}} & \ensuremath{7.24\pm2.93} & \ensuremath{29.7\pm0.5} & \ensuremath{0.990\pm0.016} & \ensuremath{0.010\pm0.016} & \ensuremath{41.66\pm1.18} \\
		50 & $(400,4000,40)$ & ADMM & \ensuremath{296.1\pm214.7} & \ensuremath{1.465\mathrm{e}{-2}\pm6.080\mathrm{e}{-3}} & \ensuremath{0.68\pm0.45} & \ensuremath{39.0\pm0.9} & \ensuremath{0.976\pm0.024} & \ensuremath{0.024\pm0.024} & \ensuremath{39.46\pm1.81} \\
		50 & $(400,4000,40)$ & HV & \ensuremath{9260.0\pm4536.4} & \ensuremath{1.465\mathrm{e}{-2}\pm6.080\mathrm{e}{-3}} & \ensuremath{13.33\pm6.48} & \ensuremath{39.0\pm0.9} & \ensuremath{0.976\pm0.024} & \ensuremath{0.024\pm0.024} & \ensuremath{39.46\pm1.81} \\
		50 & $(500,5000,50)$ & ADMM & \ensuremath{285.1\pm165.4} & \ensuremath{2.072\mathrm{e}{-2}\pm8.791\mathrm{e}{-3}} & \ensuremath{1.08\pm0.72} & \ensuremath{49.1\pm0.7} & \ensuremath{0.983\pm0.015} & \ensuremath{0.017\pm0.015} & \ensuremath{38.78\pm1.45} \\
		50 & $(500,5000,50)$ & HV & \ensuremath{8050.1\pm2916.2} & \ensuremath{2.072\mathrm{e}{-2}\pm8.791\mathrm{e}{-3}} & \ensuremath{18.52\pm6.50} & \ensuremath{49.1\pm0.7} & \ensuremath{0.983\pm0.015} & \ensuremath{0.017\pm0.015} & \ensuremath{38.78\pm1.44} \\
		\bottomrule
	\end{tabular}%  
	}
	\endgroup
\end{table}

Experimental results on the oversampled DCT matrix show that recovery quality depends on the noise level. As the SNR increases from 30 dB to 50 dB, topKhit improves from $0.86$--$0.91$ to $0.98$--$0.99$, and the RSNR increases by about 18--20 dB. At the same noise level, ADMM and HV yield almost identical objective function values, support metrics, and RSNR. The computational difference, however, remains clear. Across the nine DCT settings, ADMM reduces the average number of iterations by about $95.3\%$--$97.4\%$ and gives an average runtime speedup between $13.3$ and $22.7$, based on the reported mean CPU times. In every DCT setting, the average runtime of ADMM is lower than that of HV.

\begin{table}[htbp]
	\centering
	\caption{Noisy recovery results for Gaussian sensing matrices at three SNR levels. Values are computed over 20 paired Monte Carlo trials. The notation $\mu\pm\sigma$ reports the mean and sample standard deviation.}
	\label{tab:noisy-gaussian-snr-metrics}
	\begingroup
	\scriptsize
	\setlength{\tabcolsep}{2pt}
	\resizebox{\textwidth}{!}{%
	\begin{tabular}{cclccccccc}
		\toprule
		SNR & $(M,N,s)$ & Solver & Iter & Obj. & Time(s) & topKhit & topKrec & topKerr & RSNR(dB) \\
		\midrule
		30 & $(300,3000,30)$ & ADMM & \ensuremath{571.1\pm207.6} & \ensuremath{9.033\mathrm{e}{-3}\pm4.935\mathrm{e}{-3}} & \ensuremath{0.86\pm0.41} & \ensuremath{28.6\pm0.8} & \ensuremath{0.953\pm0.027} & \ensuremath{0.047\pm0.027} & \ensuremath{25.86\pm1.94} \\
		30 & $(300,3000,30)$ & HV & \ensuremath{15000.0\pm0.0} & \ensuremath{9.041\mathrm{e}{-3}\pm4.940\mathrm{e}{-3}} & \ensuremath{14.06\pm5.98} & \ensuremath{28.4\pm0.9} & \ensuremath{0.948\pm0.030} & \ensuremath{0.052\pm0.030} & \ensuremath{25.39\pm2.00} \\
		30 & $(400,4000,40)$ & ADMM & \ensuremath{391.4\pm117.1} & \ensuremath{1.809\mathrm{e}{-2}\pm1.409\mathrm{e}{-2}} & \ensuremath{1.26\pm0.39} & \ensuremath{37.9\pm1.6} & \ensuremath{0.948\pm0.039} & \ensuremath{0.052\pm0.039} & \ensuremath{25.94\pm2.81} \\
		30 & $(400,4000,40)$ & HV & \ensuremath{20000.0\pm0.0} & \ensuremath{1.810\mathrm{e}{-2}\pm1.410\mathrm{e}{-2}} & \ensuremath{39.31\pm8.95} & \ensuremath{37.8\pm1.7} & \ensuremath{0.945\pm0.042} & \ensuremath{0.055\pm0.042} & \ensuremath{25.58\pm2.83} \\
		30 & $(500,5000,50)$ & ADMM & \ensuremath{263.3\pm56.9} & \ensuremath{2.310\mathrm{e}{-2}\pm9.040\mathrm{e}{-3}} & \ensuremath{0.96\pm0.20} & \ensuremath{47.0\pm1.7} & \ensuremath{0.940\pm0.034} & \ensuremath{0.060\pm0.034} & \ensuremath{26.41\pm1.82} \\
		30 & $(500,5000,50)$ & HV & \ensuremath{25000.0\pm0.0} & \ensuremath{2.311\mathrm{e}{-2}\pm9.045\mathrm{e}{-3}} & \ensuremath{55.97\pm0.91} & \ensuremath{47.0\pm1.7} & \ensuremath{0.940\pm0.034} & \ensuremath{0.060\pm0.034} & \ensuremath{26.15\pm1.84} \\
		40 & $(300,3000,30)$ & ADMM & \ensuremath{467.4\pm171.9} & \ensuremath{2.684\mathrm{e}{-3}\pm1.432\mathrm{e}{-3}} & \ensuremath{0.65\pm0.24} & \ensuremath{29.3\pm0.7} & \ensuremath{0.977\pm0.024} & \ensuremath{0.023\pm0.024} & \ensuremath{35.64\pm1.88} \\
		40 & $(300,3000,30)$ & HV & \ensuremath{15000.0\pm0.0} & \ensuremath{2.684\mathrm{e}{-3}\pm1.433\mathrm{e}{-3}} & \ensuremath{13.68\pm5.87} & \ensuremath{29.2\pm0.8} & \ensuremath{0.975\pm0.026} & \ensuremath{0.025\pm0.026} & \ensuremath{35.14\pm1.91} \\
		40 & $(400,4000,40)$ & ADMM & \ensuremath{302.4\pm75.1} & \ensuremath{5.420\mathrm{e}{-3}\pm4.196\mathrm{e}{-3}} & \ensuremath{1.06\pm0.62} & \ensuremath{39.3\pm0.9} & \ensuremath{0.982\pm0.023} & \ensuremath{0.018\pm0.023} & \ensuremath{35.84\pm2.81} \\
		40 & $(400,4000,40)$ & HV & \ensuremath{20000.0\pm0.0} & \ensuremath{5.421\mathrm{e}{-3}\pm4.198\mathrm{e}{-3}} & \ensuremath{39.59\pm10.54} & \ensuremath{39.4\pm0.9} & \ensuremath{0.984\pm0.023} & \ensuremath{0.016\pm0.023} & \ensuremath{35.49\pm2.83} \\
		40 & $(500,5000,50)$ & ADMM & \ensuremath{207.2\pm42.6} & \ensuremath{6.962\mathrm{e}{-3}\pm2.747\mathrm{e}{-3}} & \ensuremath{0.78\pm0.18} & \ensuremath{48.7\pm1.0} & \ensuremath{0.974\pm0.020} & \ensuremath{0.026\pm0.020} & \ensuremath{36.23\pm1.92} \\
		40 & $(500,5000,50)$ & HV & \ensuremath{25000.0\pm0.0} & \ensuremath{6.964\mathrm{e}{-3}\pm2.747\mathrm{e}{-3}} & \ensuremath{56.00\pm1.18} & \ensuremath{48.7\pm1.0} & \ensuremath{0.974\pm0.020} & \ensuremath{0.026\pm0.020} & \ensuremath{35.96\pm1.96} \\
		50 & $(300,3000,30)$ & ADMM & \ensuremath{335.4\pm103.8} & \ensuremath{8.317\mathrm{e}{-4}\pm4.406\mathrm{e}{-4}} & \ensuremath{0.47\pm0.17} & \ensuremath{29.7\pm0.5} & \ensuremath{0.990\pm0.016} & \ensuremath{0.010\pm0.016} & \ensuremath{45.57\pm1.88} \\
		50 & $(300,3000,30)$ & HV & \ensuremath{14986.0\pm62.4} & \ensuremath{8.318\mathrm{e}{-4}\pm4.406\mathrm{e}{-4}} & \ensuremath{14.93\pm5.59} & \ensuremath{29.6\pm0.5} & \ensuremath{0.988\pm0.016} & \ensuremath{0.012\pm0.016} & \ensuremath{45.05\pm1.90} \\
		50 & $(400,4000,40)$ & ADMM & \ensuremath{227.6\pm58.4} & \ensuremath{1.684\mathrm{e}{-3}\pm1.301\mathrm{e}{-3}} & \ensuremath{0.76\pm0.28} & \ensuremath{39.9\pm0.4} & \ensuremath{0.996\pm0.009} & \ensuremath{0.004\pm0.009} & \ensuremath{45.79\pm2.84} \\
		50 & $(400,4000,40)$ & HV & \ensuremath{19675.8\pm1190.0} & \ensuremath{1.685\mathrm{e}{-3}\pm1.301\mathrm{e}{-3}} & \ensuremath{39.43\pm9.43} & \ensuremath{39.8\pm0.4} & \ensuremath{0.995\pm0.010} & \ensuremath{0.005\pm0.010} & \ensuremath{45.43\pm2.87} \\
		50 & $(500,5000,50)$ & ADMM & \ensuremath{163.4\pm28.3} & \ensuremath{2.168\mathrm{e}{-3}\pm8.577\mathrm{e}{-4}} & \ensuremath{0.62\pm0.13} & \ensuremath{49.6\pm0.5} & \ensuremath{0.992\pm0.010} & \ensuremath{0.008\pm0.010} & \ensuremath{46.10\pm1.95} \\
		50 & $(500,5000,50)$ & HV & \ensuremath{24827.7\pm770.8} & \ensuremath{2.168\mathrm{e}{-3}\pm8.577\mathrm{e}{-4}} & \ensuremath{55.80\pm2.29} & \ensuremath{49.6\pm0.5} & \ensuremath{0.992\pm0.010} & \ensuremath{0.008\pm0.010} & \ensuremath{45.80\pm2.00} \\
		\bottomrule
	\end{tabular}%
	}
	\endgroup
\end{table}

The Gaussian results are reported in Table~\ref{tab:noisy-gaussian-snr-metrics}. These tests show an even clearer difference in computational efficiency. In all settings, HV reaches, or nearly reaches, the prescribed maximum number of iterations, whereas ADMM stops after an average of $163.4$--$571.1$ iterations. Based on the reported averages, ADMM reduces the iteration count by about $96.2\%$--$99.3\%$ and gives an average runtime speedup between $16.3$ and $90.0$. The recovery performance remains comparable. The two solvers give almost the same objective values and similar top-$K$ recall. In these Gaussian tests, the average RSNR of ADMM is also slightly higher, by about $0.26$--$0.52$ dB. This difference, however, should be viewed as a minor effect, since the two solvers use the same regularization parameter and the same initialization. The main point is still the large gain in computational efficiency.

Overall, the noisy experiments show that our ADMM algorithm substantially improves computational efficiency while maintaining the same recovery quality as the HV solver. This indicates that the exact treatment of the WDSN proximal operator in ADMM provides a clear algorithmic advantage without degrading the objective function value, support recovery accuracy, or reconstruction accuracy.

\subsection{Amplitude Rescaling Experiment}
\label{subsec:amplitude-rescaling}

We next test the effect of amplitude rescaling when the regularization parameter is selected at one signal amplitude and then kept fixed. The WDSN functional is two-homogeneous, namely $\mathcal{R}_{\eta}(c\bm{x})=c^2\mathcal{R}_{\eta}(\bm{x})$ with $c>0$, which has the same homogeneity as the quadratic data-fidelity term $\frac12\|\bm{A}(c\bm{x})-c\bm{b}\|_2^2=c^2\frac12\|\bm{A}\bm{x}-\bm{b}\|_2^2$. Thus the relative balance between the data-fidelity term and the WDSN penalty is preserved when the data amplitude is rescaled. In contrast, the $\ell_1-\alpha\ell_2$ penalty is one-homogeneous, and therefore keeping the same $\lambda$ after amplitude rescaling changes its effective regularization strength.

We generate Gaussian sensing matrices with $(M,N,s)=(40,100,8)$ and SNR $=50$ dB. For each Monte Carlo trial, the WDSN parameter $\lambda_{\mathrm{WDSN}}$ and the $\ell_1-\alpha\ell_2$ parameter $\lambda_{\ell_1-\alpha\ell_2}$ are selected separately on the original problem with $c=1$. We then solve the rescaled problems $\bm{b}_c=c\bm{b}$ with $c\in\{10^{-3},10^{-2},10^{-1},1,10,10^2,10^3\}$, while each method keeps its own value of $\lambda$ selected at $c=1$. Both methods use the same $\ell_1$ warm start for each scaled problem. We set $\eta=\alpha=1$, use $\rho=1$, and average the results over 100 Monte Carlo trials. The metrics reported in Table~\ref{tab:scale-transfer} are defined as follows: $\mathrm{RelErr}(c):=\|\widehat{\bm{x}}(c\bm{b})-\bm{x}_c^*\|_2/\|\bm{x}_c^*\|_2$, the relative reconstruction error with respect to the rescaled ground truth $\bm{x}_c^*=c\bm{x}^*$; $\mathrm{ScaleErr}(c):=\|\widehat{\bm{x}}(c\bm{b})-c\widehat{\bm{x}}(\bm{b})\|_2/\|c\widehat{\bm{x}}(\bm{b})\|_2$, the deviation from exact output scaling; $\widehat{k}$, the number of recovered entries whose magnitudes exceed $10^{-2}\|\widehat{\bm{x}}(c\bm{b})\|_{\infty}$; FP, the number of such recovered entries outside the true support; Conv., the proportion of trials satisfying the ADMM convergence criterion.

\begin{table}[t]
	\caption{Amplitude rescaling experiment results. Values are reported as mean $\pm$ sample standard deviation.}
	\label{tab:scale-transfer}
	\centering
	\scriptsize
	\resizebox{\textwidth}{!}{%
	\begin{tabular}{ccccccccccc}
		\toprule
		& \multicolumn{2}{c}{RelErr} & \multicolumn{2}{c}{ScaleErr} & \multicolumn{2}{c}{$\widehat{k}$} & \multicolumn{2}{c}{FP} & \multicolumn{2}{c}{Conv.} \\
		\cmidrule(lr){2-3}\cmidrule(lr){4-5}\cmidrule(lr){6-7}\cmidrule(lr){8-9}\cmidrule(lr){10-11}
		$c$ & WDSN & $\ell_1-\ell_2$ & WDSN & $\ell_1-\ell_2$ & WDSN & $\ell_1-\ell_2$ & WDSN & $\ell_1-\ell_2$ & WDSN & $\ell_1-\ell_2$ \\
		\midrule
		$10^{-3}$ & $5.619{\times}10^{-3}\pm2.155{\times}10^{-3}$ & $9.145{\times}10^{-1}\pm1.298{\times}10^{-1}$ & $2.508{\times}10^{-3}\pm1.200{\times}10^{-3}$ & $9.143{\times}10^{-1}\pm1.302{\times}10^{-1}$ & $8.00\pm0.00$ & $1.82\pm1.90$ & $0.00\pm0.00$ & $0.24\pm0.73$ & $1.00$ & $0.74$ \\
		$10^{-2}$ & $4.878{\times}10^{-3}\pm1.984{\times}10^{-3}$ & $1.380{\times}10^{-1}\pm5.205{\times}10^{-2}$ & $1.345{\times}10^{-3}\pm1.001{\times}10^{-3}$ & $1.362{\times}10^{-1}\pm5.210{\times}10^{-2}$ & $8.00\pm0.00$ & $8.08\pm0.31$ & $0.00\pm0.00$ & $0.08\pm0.31$ & $1.00$ & $1.00$ \\
		$10^{-1}$ & $4.233{\times}10^{-3}\pm1.424{\times}10^{-3}$ & $1.353{\times}10^{-2}\pm4.399{\times}10^{-3}$ & $1.910{\times}10^{-4}\pm1.795{\times}10^{-4}$ & $1.149{\times}10^{-2}\pm4.231{\times}10^{-3}$ & $8.00\pm0.00$ & $8.00\pm0.00$ & $0.00\pm0.00$ & $0.00\pm0.00$ & $1.00$ & $1.00$ \\
		$1$ & $4.176{\times}10^{-3}\pm1.383{\times}10^{-3}$ & $2.991{\times}10^{-3}\pm6.644{\times}10^{-4}$ & $0\pm0$ & $0\pm0$ & $8.00\pm0.00$ & $8.00\pm0.00$ & $0.00\pm0.00$ & $0.00\pm0.00$ & $1.00$ & $1.00$ \\
		$10$ & $4.171{\times}10^{-3}\pm1.382{\times}10^{-3}$ & $4.185{\times}10^{-3}\pm1.264{\times}10^{-3}$ & $2.087{\times}10^{-5}\pm1.771{\times}10^{-5}$ & $2.675{\times}10^{-3}\pm9.579{\times}10^{-4}$ & $8.00\pm0.00$ & $8.00\pm0.00$ & $0.00\pm0.00$ & $0.00\pm0.00$ & $1.00$ & $0.69$ \\
		$10^2$ & $4.171{\times}10^{-3}\pm1.382{\times}10^{-3}$ & $5.261{\times}10^{-3}\pm1.752{\times}10^{-3}$ & $2.239{\times}10^{-5}\pm1.957{\times}10^{-5}$ & $3.645{\times}10^{-3}\pm1.326{\times}10^{-3}$ & $8.00\pm0.00$ & $8.00\pm0.00$ & $0.00\pm0.00$ & $0.00\pm0.00$ & $1.00$ & $0.00$ \\
		$10^3$ & $4.171{\times}10^{-3}\pm1.382{\times}10^{-3}$ & $5.874{\times}10^{-3}\pm1.920{\times}10^{-3}$ & $2.252{\times}10^{-5}\pm1.977{\times}10^{-5}$ & $4.184{\times}10^{-3}\pm1.490{\times}10^{-3}$ & $8.00\pm0.00$ & $8.00\pm0.00$ & $0.00\pm0.00$ & $0.00\pm0.00$ & $1.00$ & $0.00$ \\
		\bottomrule
	\end{tabular}
	}
\end{table}

The results in Table~\ref{tab:scale-transfer} show the expected scaling behavior of WDSN. With the same value of $\lambda$ selected at $c=1$, WDSN keeps the relative reconstruction error in the narrow range from $4.171\times10^{-3}$ to $5.619\times10^{-3}$ over six orders of magnitude in amplitude, and its scale error remains below $2.51\times10^{-3}$. It also recovers $\widehat{k}=8$ with zero false positives and converges in all tested trials. For the $\ell_1-\ell_2$ model, its relative error increases to $9.15\times10^{-1}$ at $c=10^{-3}$, where it recovers only $1.82$ entries on average. Its scale error remains much larger than that of WDSN away from the tuning scale, and the proportion of converged runs drops to $0$ at both $c=10^2$ and $c=10^3$. These results show that WDSN can keep the same $\lambda$ across data amplitudes, whereas the one-homogeneous $\ell_1-\ell_2$ penalty is sensitive to amplitude rescaling.

\section{Conclusion}\label{conclusion}

This paper developed a recovery theory and an algorithmic framework for sparse recovery with the WDSN penalty
$\mathcal{R}_{\eta}(\bm{x})=\|\bm{x}\|_1^2-\eta\|\bm{x}\|_2^2$.
On the theoretical side, we introduced an $\eta$-augmented NSP and used it to obtain a sufficient condition for uniform exact recovery in the noiseless model. For noisy measurements, we established RIP-based stable recovery guarantees for both exactly $k$-sparse signals and general signals. In the general case, the error bound contains both the measurement noise and the best $k$-term approximation error, and it reduces to the exactly $k$-sparse case when the approximation error vanishes.

On the algorithmic side, we derived a closed-form proximal mapping of $\mathcal{R}_{\eta}$ in the three regimes determined by $1-2t\eta$. Embedding this exact WDSN proximal step into a ADMM scheme gives an efficient solver for the regularized WDSN model. Under coercivity and a sufficiently large penalty parameter, we proved subsequential convergence of the ADMM algorithm. Under the Kurdyka--\L{}ojasiewicz property, we further established convergence of the whole sequence to a stationary point of the split problem.

The numerical experiments show the computational benefit of the exact WDSN proximal step. With the same initialization, parameter settings, and stopping criteria, the proposed ADMM method reaches final objective values and support-recovery accuracy comparable to those of the HV method, while using far fewer iterations and less CPU time. The improvement is clearer for Gaussian sensing matrices and in the noisy tests. The amplitude-rescaling experiment further shows that WDSN can keep the same regularization parameter across a wide range of signal amplitudes, whereas the one-homogeneous $\ell_1-\ell_2$ penalty is more sensitive to this rescaling. Future work includes improving the recovery constants, designing adaptive rules for choosing $\lambda$, $\eta$, and $\rho$, and comparing WDSN with other sparse recovery penalties.

\begin{appendices}
%\section{Implementation Note}
%The MATLAB implementation of Algorithm~\ref{alg:prox_l1sq_l2sq_compact} sorts $|v|$ once in the case $c>0$ and is therefore dominated by $O(n\log n)$ arithmetic. The cases $c\le0$ require only finding the maximum entry of $|v|$ and are $O(n)$. These formulas are used in the WDSN proximal steps tested in Section \ref{sec-num}.
\end{appendices}
\bibliographystyle{elsarticle-num}
\bibliography{thesis}
\end{document}